\documentclass[a4paper,12pt,leqno]{amsart}
\usepackage{latexsym}
\usepackage[all]{xy}

\usepackage{amssymb} 
\usepackage{amsmath} 
\usepackage{color}
\usepackage{comment}
\usepackage{mathtools}

\usepackage{graphicx}
\usepackage{stmaryrd}

\definecolor{gray}{gray}{0.7}
\definecolor{Gray}{gray}{0.3}

\usepackage{amscd}

\numberwithin{equation}{section}

\theoremstyle{break}
 \newtheorem{theorem}{Theorem}[section]
 \newtheorem{proposition}[theorem]{Proposition}
 \newtheorem{corollary}[theorem]{Corollary}
 \newtheorem{lemma}[theorem]{Lemma}

 \theoremstyle{definition}
 \newtheorem{definition}[theorem]{Definition}
 \newtheorem{remark}[theorem]{Remark}
 \newtheorem{example}[theorem]{Example}

\allowdisplaybreaks[3]

\def\C{\mathbb C}
\def\R{\mathbb R}

\def\Z{\mathbb Z}

\def\CR{\mathcal{R}}

\DeclareMathOperator{\Poin}{Poin}
\DeclareMathOperator{\Hess}{Hess}
\DeclareMathOperator{\point}{pt}

\DeclareMathOperator{\height}{ht}

\def\t{t}

\def\e{e}
\def\x{x}
\def\m{m}
\def\Hilb{F}

\begin{document}
  
\title[An additive basis for the cohomlogy of Hessenberg varieties]{An additive basis for the cohomology rings of regular nilpotent Hessenberg varieties}
\author [M. Enokizono]{Makoto Enokizono}
\address{Department of Mathematics, Faculty of Science and Technology, 
Tokyo University of Science, 
2641 Yamazaki, Noda, Chiba 278-8510, Japan}
\email{enokizono\_makoto@ma.noda.tus.ac.jp}

\author [T. Horiguchi]{Tatsuya Horiguchi}
\address{National institute of technology, Ube college, 2-14-1, Tokiwadai, Ube, Yamaguchi, Japan 755-8555}
\email{tatsuya.horiguchi0103@gmail.com}

\author [T. Nagaoka]{Takahiro Nagaoka}
\address{Research Institute for Mathematical Sciences, Kyoto University, Kyoto 606-8502, Japan}
\email{takahiro.nagaoka3617@gmail.com}

\author [A. Tsuchiya]{Akiyoshi Tsuchiya}
\address{Department of Information Science, Faculty of Science, Toho University, Miyama, Funabashi-shi, Chiba, 274-8510, Japan}
\email{akiyoshi@is.sci.toho-u.ac.jp}

\keywords{flag varieties, Hessenberg varieties, Poincar\'e dual.} 

\begin{abstract}
In this paper we construct an additive basis for the cohomology ring of a regular nilpotent Hessenberg variety which is obtained by extending all Poincar\'e duals of smaller regular nilpotent Hessenberg subvarieties.
In particular, all of the Poincar\'e duals of smaller regular nilpotent Hessenberg subvarieties are linearly independent. 
\end{abstract}

\maketitle

\setcounter{tocdepth}{1}

\tableofcontents

\section{Introduction}
\label{sec:introduction}

Hessenberg varieties are subvarieties of a full flag variety, which are introduced by De~Mari--Procesi--Shayman (\cite{dMPS, dMS}).
Their topology makes connections with other research areas such as the logarithmic derivation modules in hyperplane arrangements and Stanley's chromatic symmetric functions in graph theory (see e.g. the survey article \cite{AH}).

In this paper, we consider regular nilpotent Hessenberg varieties $\Hess(N,I)$ where $N$ is a regular nilpotent element and $I$ is a lower ideal (see Section~\ref{section:Hessenberg varieties} for the definitions).
They can be regarded as a (discrete) family of subvarieties of the flag variety connecting the Peterson variety and the flag variety itself. 
Here, the Peterson variety arises in the study of the quantum cohomology
of the flag variety (\cite{Ko, R}).
An explicit presentation by generators and relations of the cohomology ring\footnote{In this paper, unless stated otherwise, we work with singular cohomology with coefficients in $\R$.} of the Peterson variety is given by \cite{FHM} in type $A$, and soon after is given by \cite{HHM} in all Lie types. 
Then, the results in the special case of Peterson varieties can be generalized to arbitrary regular nilpotent Hessenberg varieties (\cite{AHHM, AHMMS, EHNT1}).

Our result gives an explicit additive basis for the cohomology ring of a regular nilpotent Hessenberg variety in classical types and type $G$ using an explicit presentation of the cohomology ring.
This basis is motivated as follows. 
The flag variety admits a complex cellular decomposition by the Schubert cells, which implies that the Poincar\'e duals of Schubert varieties in the flag variety form an additive basis for the cohomology of the flag variety. 
More generally, given a Schubert variety $X_w$, the homology classes of smaller Schubert varieties $X_{w'}$ in $X_w$  ($w' \leq w$ in Bruhat order)  form an additive basis for the homology of the given Schubert variety $X_w$.
Motivated by this result, we can also ask related questions:

\begin{enumerate}
\item Given a regular nilpotent Hessenberg variety $\Hess(N,I)$, is the set of Poincar\'e duals of smaller regular nilpotent Hessenberg subvarieties $\Hess(N,I')$ in $\Hess(N,I)$ ($I' \subset I$ in a natural inclusion) linearly independent? 
\item If the question $(1)$ is true, then can we find a ``natural'' basis $\mathcal{B}$ for the cohomology of $\Hess(N,I)$ so that $\mathcal{B}$ extends the set of Poincar\'e duals of smaller regular nilpotent Hessenberg subvarieties $\Hess(N,I')$ in $\Hess(N,I)$? 
Here, we say that a basis $\mathcal{B}$ of $H^*(\Hess(N,I))$ extends the set of Poincar\'e duals
of $\Hess(N,I')$ for $I' \subset I$ if $[\Hess(N,I')] \in \mathcal{B}$ for all $I' \subset I$.
\end{enumerate}
In the question $(1)$ above, the Poincar\'e dual of $\Hess(N,I')$ can be written as certain monomial in positive roots.
From this fact, we can expect a natural basis for the cohomology of $\Hess(N,I)$ whose elements are monomials in positive roots in the question $(2)$.

Our first main result is to construct such a basis. In order to describe our basis, we need a notion of Hessenberg functions $h$. The Hessenberg function $h$ uniquely determines a lower ideal $I$, but we remark that this notion depends on a choice of a decomposition $\Phi^+=\coprod_{i=1}^n \ \Phi^+_i$ of the set of positive roots $\Phi^+$. (See Section~\ref{section:Gysin map} for more details.) Here and below, we choose some decomposition $\Phi^+=\coprod_{i=1}^n \ \Phi^+_i$ described in Sections~\ref{section:Main theorem} and \ref{section:typeD}, and $\Hess(N,I)$ is denoted by $\Hess(N,h)$ where $h$ is the Hessenberg function which corresponds to the lower ideal $I$.
We now describe the main result of this paper. 
All undefined terms are specified in the sections below.

\begin{theorem} \label{theorem_intro}
Let $\Hess(N,h)$ be a regular nilpotent Hessenberg variety of types $A, B, C, G$ of rank $n$.
We fix a permutation $w^{(i)}$ on a set $\{i+1, i+2, \dots, h(i)\}$ for each $i=1,\ldots,n$.
Then, the following cohomology classes 
\begin{equation} \label{eq:basis_Intro}
\prod_{i=1}^{n} \alpha_{i,w^{(i)}(h(i))} \cdot \alpha_{i,w^{(i)}(h(i)-1)} \cdots \alpha_{i,w^{(i)}(h(i)-\m_i+1)},
\end{equation}
with $0 \leq \m_i \leq h(i)-i$, form an additive basis for the cohomology $H^*(\Hess(N,h))$ over $\R$.
Here, we take the convention $\alpha_{i,w^{(i)}(h(i))} \cdot \alpha_{i,w^{(i)}(h(i)-1)} \cdots \alpha_{i,w^{(i)}(h(i)-\m_i+1)}=1$ whenever $\m_i=0$.
\end{theorem}

The $\alpha_{i,j}$ in \eqref{eq:basis_Intro} denotes a positive root (see Section~\ref{section:Main theorem} for the definition).
The basis in Theorem~\ref{theorem_intro} is, in fact, an extended basis of all Poincar\'e duals $[\Hess(N,h')]$ of smaller regular nilpotent Hessenberg varieties $\Hess(N,h')$ in $H^*(\Hess(N,h))$ ($h' \subset h$ in a natural inclusion, i.e., $h'(i) \leq h(i)$ for any $i=1,\ldots,n$).
More specifically, take the identity maps as permutations $w^{(i)}$ in Theorem~\ref{theorem_intro}.
Then we will see that the cohomology class in \eqref{eq:basis_Intro} for $m_i=h(i)-h'(i) \ (1 \leq i \leq n)$ is the Poincar\'e dual $[\Hess(N,h')]$ in $H^*(\Hess(N,h))$ up to a non-zero scalar multiplication in Section~\ref{section:Gysin map}.
Hence, Theorem~\ref{theorem_intro} gives an affirmative answer to the questions $(1)$ and $(2)$ for types $A,B,C,G$.

It is technically difficult to construct a positive root basis for the cohomology of $\Hess(N,h)$ in type $D$.
Nevertheless we are able to prove the linear independence of the set of Poincar\'e duals of smaller regular nilpotent Hessenberg varieties $\Hess(N,h')$ by constructing a different type of a basis as follows.

\begin{theorem} \label{theorem_intro_typeD}
Let $\Hess(N,h)$ be a regular nilpotent Hessenberg variety of type $D$ of rank $n$.
Then, the following cohomology classes 
\begin{equation} \label{eq:basis_Intro_typeD}
\prod_{i=1}^{n} \alpha_{i,h(i)}^{(h-m)} \cdot \alpha_{i,h(i)-1}^{(h-m)} \cdots \alpha_{i,h(i)-m_i+1}^{(h-m)},
\end{equation}
with $0 \leq \m_i \leq h(i)-i$, form an additive basis for the cohomology $H^*(\Hess(N,h))$ over $\R$.
Here, we take the convention $ \alpha_{i,h(i)}^{(h-m)} \cdot \alpha_{i,h(i)-1}^{(h-m)} \cdots \alpha_{i,h(i)-m_i+1}^{(h-m)}=1$ whenever $\m_i=0$.
\end{theorem}

Note that $\alpha^{(h-m)}_{i,j}$ in \eqref{eq:basis_Intro_typeD} is not necessarily a positive root (see Section~\ref{section:typeD} for the definition), but it is a positive root if we take $m_i=h(i)-h'(i)$ for $1 \leq i \leq n$ (Lemma~\ref{lemma:linear independence_typeD}) where $h'$ denotes a Hessenberg function with $h' \subset h$.
Thus, the cohomology class in \eqref{eq:basis_Intro_typeD} for $m_i=h(i)-h'(i) \ (1 \leq i \leq n)$ is  the Poincar\'e dual $[\Hess(N,h')]$ in $H^*(\Hess(N,h))$ up to a non-zero scalar multiplication (see Section~\ref{section:Gysin map} for the details).
Namely, the basis in Theorem~\ref{theorem_intro_typeD} is also an extended basis of all Poincar\'e duals $[\Hess(N,h')]$ of smaller regular nilpotent Hessenberg varieties $\Hess(N,h')$ in $H^*(\Hess(N,h))$.

We do expect the analogue of the linear independence to hold for other exceptional types
and give an affirmative answer to question $(1)$ for types $F_4$ and $E_6$ by using Maple.
We summarize the results as follows.

\begin{theorem} \label{theorem:Intro_ABCDE6FG}
Let $\Hess(N,h)$ be a regular nilpotent Hessenberg variety of types $A_{n},B_n,C_n$, $D_n, E_6, F_4, G_2$. 
Then, the set $\{[\Hess(N,h')] \in H^*(\Hess(N,h)) \mid h' \subset h \}$ of the Poincar\'e duals is linearly independent. 
\end{theorem}

The paper is organized as follows.
After reviewing the definition and properties of Hessenberg varieties in Section~\ref{section:Hessenberg varieties}, we introduce the main object of the paper, namely the Poincar\'e duals of regular nilpotent Hessenberg varieties, in Section~\ref{section:Gysin map}. 
A key lemma in commutative algebra, necessary for the proof of Theorems~\ref{theorem_intro} and \ref{theorem_intro_typeD}, is recorded in Section~\ref{section:preliminaries}. 
We prove Theorem~\ref{theorem_intro} in type $A$ and explain how to adjust the proof in types $B$, $C$, and $G$, in Section~\ref{section:Main theorem}. 
We sketch the outline of the proof of Theorem~\ref{theorem_intro_typeD} in Section~\ref{section:typeD} and discuss the linear independence of the Poincar\'e duals (Theorem~\ref{theorem:Intro_ABCDE6FG}) in Section~\ref{section:Poincareduals}.

\bigskip
\noindent \textbf{Acknowledgements.}  

We are grateful to Satoshi Murai for his invaluable help in the commutative algebra. 
We appreciate Hiraku Abe for fruitful discussions on a flat family of regular Hessenberg varieties.
We also appreciate Martha Precup for kindly explaining a connectedness of regular semisimple Hessenberg varieties.
The first author was partially supported by JSPS Grant-in-Aid for Research Activity Start-up: 19K23407. 
The second author was partially supported by JSPS Grant-in-Aid for JSPS Research Fellow: 17J04330. 
He was also partially supported by JSPS Grant-in-Aid for Young Scientists: 19K14508. 
The third author was partially supported by JSPS Grant-in-Aid for JSPS Research Fellow: 19J11207.
The fourth author was partially supported by JSPS Grant-in-Aid for JSPS Research Fellow: 19J00312. 
He was also partially supported by JSPS Grant-in-Aid for Young Scientists: 19K14505.

\bigskip

\section{Hessenberg varieties}
\label{section:Hessenberg varieties}

Let $G$ be a semisimple linear algebraic group of rank $n$ and fix a Borel subgroup $B$ of $G$. We denote by $\mathfrak{g}$ and $\mathfrak{b}$ the Lie algebras of $G$ and $B$, respectively.
A \textbf{Hessenberg space} is defined to be a $\mathfrak{b}$-submodule of $\mathfrak{g}$ containing $\mathfrak{b}$.  
Fix a set $\{\alpha_1,\ldots,\alpha_n \}$ of simple roots and we denote by $\Phi^+$ the set of positive roots.
A subset $I \subset \Phi^+$ is called a \textbf{lower ideal} if it satisfies the following condition:
\begin{align*}
\textrm{if $\alpha \in \Phi^+$ and $\beta \in I$ with $\alpha \preceq \beta$, then $\alpha \in I$,}
\end{align*}
where $\alpha \preceq \beta$ if and only if $\beta - \alpha$ can be written as a linear combination of the simple roots $\alpha_1,\ldots,\alpha_n$ with non-negative coefficients.

One can see that there is one-to-one correspondence between the set of lower ideals and the set of Hessenberg spaces which sends $I \subset \Phi^+$ to 
$$
H(I):=\mathfrak{b}\oplus \big(\bigoplus_{\alpha \in I} \mathfrak{g}_{-\alpha} \big),
$$ 
where $\mathfrak{g}_{\alpha}$ is the root space for a root $\alpha$. 
The \textbf{Hessenberg variety} $\Hess(X,I)$ associated with an element $X\in \mathfrak{g}$ and a lower ideal $I \subset \Phi^+$ is defined to be the following subvariety of the flag variety $G/B$: 
\begin{equation*} 
\Hess(X,I):=\{g B \in G/B \mid \mbox{Ad}(g^{-1})(X) \in H(I)\}.
\end{equation*}
An element $X\in \mathfrak{g}$ is \textbf{nilpotent} if $\mbox{ad}(X)$ is nilpotent, i.e.,\ $\mbox{ad}(X)^k=0$ for some $k>0$.
An element $X\in \mathfrak{g}$ is \textbf{regular} if the $G$-orbit $G \cdot X$ under the adjoint action has the largest possible dimension. 
In this paper, we concentrate on Hessenberg varieties $\Hess(N,I)$ for a regular nilpotent element $N \in \mathfrak{g}$ which are called \textbf{regular nilpotent Hessenberg varieties}. 
Note that if $I=\Phi^+$, then $\Hess(N,I)$ coincides with the flag variety $G/B$. 
If we take $I$ as the set $\{\alpha_1,\ldots,\alpha_n \}$ of simple roots, then the associated regular nilpotent Hessenberg variety is called the \textbf{Peterson variety} which arises in the study of the quantum cohomology of the flag variety (\cite{Ko, R}).
The following results are some basic properties of $\Hess(N,I)$:
\begin{enumerate}
\item $\Hess(N,I)$ has no odd degree cohomology and the complex dimension of $\Hess(N,I)$ is given by $|I|$ ({\cite{Pre13}).
\item $\Hess(N,I)$ is irreducible (\cite{AFZ, Pre18}).
\item $\Hess(N,I)$ has a singular point in general (\cite{IY, Ko}}).
\end{enumerate}

Let $T$ be a maximal torus contained in the Borel subgroup $B$ and $\hat T$ the character group of $T$.
Any character $\alpha \in \hat T$ extends to a character of $B$, so $\alpha$ defines a complex line bundle $L_\alpha:=G \times_B \C$ over the flag variety $G/B$.
Here, $L_\alpha$ is the quotient of the product $G \times \C$ by the right $B$-action given by $(g,z) \cdot b  = (gb, \alpha(b)z)$ for $b \in B$ and $(g,z) \in G \times \C$.
To each $\alpha \in \hat T$ we assign the Euler class $e(L_\alpha) \in H^2(G/B)$. 
In what follows, we regard $\hat T$ as an additive group so that $\hat T \otimes_{\Z} \R$ is identified with the dual space $\mathfrak{t}^*$ of the Lie algebra of the maximal compact torus $T_\R$.
According to Borel's theorem \cite{B}, the ring homomorphism 
\begin{equation} \label{eq:varphi}
\varphi\colon \CR=\mbox{Sym}(\mathfrak{t}^*) \to H^*(G/B); \ \ \ \alpha \mapsto e(L_{\alpha})
\end{equation}
is surjective. 
Composing $\varphi$ with the restriction map $H^*(G/B) \to H^*(\Hess(N,I))$, we have a ring homomorphism
\begin{equation} \label{eq:varphiNilpotent}
\varphi_I\colon \CR\to H^*(\Hess(N,I)); \ \ \ \alpha \mapsto e(L_{\alpha})|_{\Hess(N,I)}
\end{equation}
where $e(L_{\alpha})|_{\Hess(N,I)} \in H^2(\Hess(N,I))$ denotes the restriction of the Euler class $e(L_{\alpha}) \in H^2(G/B)$.

\begin{theorem}[{\cite{AHMMS}}] \label{theorem:AHMMS}
Let $I$ be a lower ideal and $\Hess(N,I)$ the associated regular nilpotent Hessenberg variety.
Then, the following holds.
\begin{enumerate}
\item The ring homomorphism $\varphi_I$ in \eqref{eq:varphiNilpotent} is surjective. 
\item The cohomology ring $H^*(\Hess(N,I))$ is a Poincar\'e duality algebra, namely the usual paring $H^{2k}(\Hess(N,I)) \times H^{2|I|-2k}(\Hess(N,I)) \to H^{2|I|}(\Hess(N,I)) \cong \R$ via the cup product is non-degenerate.
\item The Poincar\'e polynomial of $\Hess(N,I)$ is equal to
\begin{equation} \label{eq:PoinHess0}
\Poin(\Hess(N,I),\sqrt\t)=\prod_{\alpha \in I} \frac{1-t^{\height(\alpha)+1}}{1-t^{\height(\alpha)}}
\end{equation}
where $\height(\alpha)$ denotes the height of a root $\alpha$.
\end{enumerate}
\end{theorem}

\begin{remark}
A regular nilpotent Hessenberg variety is singular in general, but its cohomology ring is a Poincar\'e duality algebra.
\end{remark}

\bigskip

\section{Gysin map}
\label{section:Gysin map}

In this section we begin with the Gysin map among regular nilpotent Hessenberg varieties. 
We then define the Poincar\'e duals of regular nilpotent Hessenberg varieties and explain their properties.

Let $I$ and $I'$ be two lower ideals with $I' \subset I$. 
By the definition of Hessenberg varieties, we have the inclusion map $\iota: \Hess(N,I') \subset \Hess(N,I)$.
Since the cohomology rings of regular nilpotent Hessenberg varieties are Poincar\'e duality algebras, 
we can algebraically define the Gysin map  
\begin{equation} \label{eq:GysinMap}
\iota_{!}: H^{*}(\Hess(N,I')) \to H^{*+2c}(\Hess(N,I))
\end{equation}
where $c:=|I|-|I'|$ is the complex codimention of $\Hess(N,I')$ in $\Hess(N,I)$. 
More specifically, the Gysin map is defined to be the following composition map
$$
\iota_{!}: H^{2k}(\Hess(N,I')) \cong H_{2|I'|-2k}(\Hess(N,I')) \xrightarrow{\iota_*} H_{2|I'|-2k}(\Hess(N,I)) \cong H^{2k+2c}(\Hess(N,I))
$$
where the first and the third isomorphisms are the Poincar\'e duality maps and the second map $\iota_*$ denotes the pushforward of the inclusion $\iota$.
Here, we need to determine an identification between the top degree cohomology $H^{2|I|}(\Hess(N,I))$ and $\R$.
In this paper we take as a generator $x_I^N$ for $H^{2|I|}(\Hess(N,I))$ an element  
$$
x_I^N:= \frac{1}{|W_I|}\prod_{\alpha \in I} e(L_{\alpha})|_{\Hess(N,I)},
$$
where $W_I$ is the parabolic subgroup of the Weyl group $W$ associated with simple roots contained in $I$, namely $W_I$ is generated by simple reflections $s_i$ associated with simple roots $\alpha_i$ in $I$.
Then, the evaluation map for $\Hess(N,I)$ is given by 
\begin{equation} \label{eq:intNilpotent}
\int: H^{2|I|}(\Hess(N,I)) \xrightarrow{\cong} \R; \ \ \  \x_I^N \mapsto 1. 
\end{equation}
We briefly explain below why we take $x_I^N$ as the generator for $H^{2|I|}(\Hess(N,I))$.
Let $S \in \mathfrak{g}$ be a regular semisimple element (i.e., $\mbox{ad}(X)$ is diagonalizable) and we consider the associated Hessenberg varieties $\Hess(S,I)$ which are called  \textbf{regular semisimple Hessenberg varieties}.
It is known that $\Hess(S,I)$ is connected if and only if $I$ contains all simple roots ({\cite[Corollary 9]{dMPS}}). 
We first consider the case when $\Hess(S,I)$ is connected, namely $I$ contains all simple roots.
Since $\Hess(S,I)$ is compact and smooth ({\cite[Theorem 6]{dMPS}}), the evaluation map on the fundamental class of $\Hess(S,I)$ gives the isomorphism 
\begin{equation} \label{eq:orientationHess(S,I)}
\int: H^{2|I|}(\Hess(S,I);\Z) \xrightarrow{\cong} \Z.
\end{equation}
Here, the Poincar\'e dual of a point in $\Hess(S,I)$ is given by 
\begin{equation} \label{eq:ptPoincaredualcnnected}
[\point]=\frac{1}{|W|}\prod_{\alpha \in I} e(L_{\alpha})|_{\Hess(S,I)} \ \ \ \ \ \ {\rm in} \ H^*(\Hess(S,I)).
\end{equation}
In fact, $\prod_{\alpha \in I} e(L_{\alpha})|_{\Hess(S,I)}$ is the Euler class of the tangent bundle of $\Hess(S,I)$ and the Euler characteristic is equal to $|W|$ (\cite[Lemma~7 and Theorem~8]{dMPS}).

The cohomology ring $H^*(\Hess(S,I))$ admits the Weyl group action introduced by Tymoczko (\cite{Tym08}). The Euler class $e(L_{\alpha})|_{\Hess(S,I)}$ is $W$-invariant for arbitrary roots $\alpha$ (\cite[Lemma 8.8.]{AHMMS}), so the Poincar\'e dual of a point is also $W$-invariant by \eqref{eq:ptPoincaredualcnnected}.
Hence, \eqref{eq:orientationHess(S,I)} leads us to the evaluation map 
\begin{align} \label{eq:evaluation_W-inv}
\int: H^{2|I|}(\Hess(S,I))^W = H^{2|I|}(\Hess(S,I)) \xrightarrow{\cong} \R; \ \ \ \frac{1}{|W|}\prod_{\alpha \in I} e(L_{\alpha})|_{\Hess(S,I)} \mapsto 1.
\end{align}
However, by the results of \cite{AHHM, AHMMS} there is a ring isomorphism between $H^*(\Hess(N,I))$ and $H^*(\Hess(S,I))^W$ so that the following diagram commutes:
\begin{equation} \label{eq:AHMMScd}
  \xymatrix{
  & H^*(G/B) \ar[dl] \ar[dr]& \\
  H^*(\Hess(N,I)) \ar[rr]^{\cong}& & H^*(\Hess(S,I))^W
  }
\end{equation}
where the slanting arrows denote restriction maps which are surjective.
In particular, the isomorphism $H^*(\Hess(N,I)) \cong H^*(\Hess(S,I))^W$ sends the Euler class $e(L_{\alpha})|_{\Hess(N,I)}$ in $\Hess(N,I)$ to the Euler class $e(L_{\alpha})|_{\Hess(S,I)}$ in $\Hess(S,I)$.
The evaluation map \eqref{eq:evaluation_W-inv} through the isomorphism $H^*(\Hess(N,I)) \cong H^*(\Hess(S,I))^W$ is nothing but \eqref{eq:intNilpotent} because now $I$ contains all simple roots.

When $I$ is an arbitrary lower ideal, we put 
$$
x_I^S: = \frac{1}{|W_I|}\prod_{\alpha \in I} e(L_{\alpha})|_{\Hess(S,I)}
$$
and define the evaluation map by
\begin{equation} \label{eq:intSemisimpleW}
\int: H^{2|I|}(\Hess(S,I))^W \xrightarrow{\cong} \R ; \ \ \ x_I^S \mapsto 1.
\end{equation}
Then, \eqref{eq:intSemisimpleW} corresponds to \eqref{eq:intNilpotent} under the isomorphism $H^*(\Hess(N,I)) \cong H^*(\Hess(S,I))^W$.

Using the Gysin map \eqref{eq:GysinMap}, we define the Poincar\'e dual of $\Hess(N,I')$ in $\Hess(N,I)$ by 
$$
[\Hess(N,I')]:=\iota_{!}(1) \in H^{2(|I|-|I'|)}(\Hess(N,I)).
$$

\begin{lemma} [\cite{AHMMS}] \label{lemma:Gysinmap}
Let $I'$ and $I$ be lower ideals with $I' \subset I$.
Then, the Gysin map $\iota_{!}: H^*(\Hess(N,I')) \to H^{*}(\Hess(N,I))$ in \eqref{eq:GysinMap} is injective.
The Poincar\'e dual of $\Hess(N,I')$ in $\Hess(N,I)$ is equal to 
\begin{equation} \label{eq:PoincaredualNilpotent}
[\Hess(N,I')]=\frac{|W_{I'}|}{|W_{I}|}\prod_{\alpha \in I \setminus I'} e(L_{\alpha})|_{\Hess(N,I)} \ \ \ \ \ \ {\rm in} \ H^*(\Hess(N,I)).
\end{equation}
\end{lemma}

For the convenience of the reader, we here explain Lemma~\ref{lemma:Gysinmap} briefly.
By \cite[Lemma~8.11]{AHMMS} the Poincar\'e dual of $\Hess(S,I')$ in $\Hess(S,I)$ is equal to 
\begin{equation} \label{eq:PoincaredualSemisimple}
[\Hess(S,I')]=\prod_{\alpha \in I \setminus  I'}e(L_{\alpha})|_{\Hess(S,I)} \ \ \ \ \ \ {\rm in} \ H^*(\Hess(S,I)).
\end{equation} 
Since the Gysin map $j_{!}: H^*(\Hess(S,I')) \to H^{*+2(|I|-|I'|)}(\Hess(S,I))$ is $W$-equivariant (\cite[Proposition~8.12]{AHMMS}), it induces $j_{!}: H^*(\Hess(S,I'))^W \to H^{*+2(|I|-|I'|)}(\Hess(S,I))^W$ and we have
$$
j_{!}(x^S_{I'})=\left(\frac{1}{|W_{I'}|}\prod_{\alpha \in I'} e(L_{\alpha})|_{\Hess(S,I)} \right) \cdot \left(\prod_{\alpha \in I \setminus I'} e(L_{\alpha})|_{\Hess(S,I)} \right)=\frac{|W_I|}{|W_{I'}|} x^S_{I}.
$$
On the other hand, we can define algebraically the Gysin map $j^W_{!}: H^*(\Hess(S,I'))^W \to H^{*+2(|I|-|I'|)}(\Hess(S,I))^W$ by using the identification in \eqref{eq:intSemisimpleW}.
Then, $j^W_{!}$ maps the generator $x_{I'}^S \in H^{2|I'|}(\Hess(S,I'))^W$ to the generator $x_I^S \in H^{2|I|}(\Hess(S,I))^W$. 
That is, this makes the following commutative diagram:
\[\xymatrix{H^*(\Hess(S,I'))\ar[r]^-{\frac{|W_{I'}|}{|W_{I}|}j_{!}} &H^*(\Hess(S,I))\\ H^*(\Hess(S,I'))^W\ar@{^{(}->}[u]\ar[r]^-{j^W_{!}}&H^*(\Hess(S,I))^W.\ar@{^{(}->}[u]}\]
Therefore, we conclude that $j^W_{!}(1) = \frac{|W_{I'}|}{|W_{I}|} \prod_{\alpha \in I \setminus  I'}e(L_{\alpha})|_{\Hess(S,I)}$, which implies \eqref{eq:PoincaredualNilpotent}.

Next, we see the injectivity for the Gysin map.
One has $H^*(\Hess(N,I)) \cong \CR/\ker \varphi_I$ by Theorem~\ref{theorem:AHMMS}.
Setting $\beta^I_{I'}:=\frac{|W_{I'}|}{|W_{I}|}\prod_{\alpha \in I \setminus  I'} \alpha \in \CR$, a relation between $\ker \varphi_I$ and $\ker \varphi_{I'}$ is given by $\ker \varphi_{I'} = \ker \varphi_I : \beta^I_{I'}$ (\cite[Proposition~4.2]{AHMMS}) and hence the multiplication map $\CR/\ker \varphi_{I'} \xrightarrow{\times \beta^I_{I'}} \CR/\ker \varphi_I$ by $\beta^I_{I'}$ is injective.
One can easily see that this map corresponds to the Gysin map $\iota_{!}: H^*(\Hess(N,I')) \to H^{*}(\Hess(N,I))$ under the isomorphism $H^*(\Hess(N,I)) \cong \CR/\ker \varphi_I$.
More specifically, the following commutative diagram holds. 
\begin{equation} \label{eq:CDGysinNilpotent}
  \begin{CD}
     H^*(\Hess(N,I')) @>{\iota_{!}}>> H^{*}(\Hess(N,I)) \\
  @V{\cong}VV    @VV{\cong}V \\
     \CR/\ker \varphi_{I'}   @>{\times \beta^I_{I'}}>>  \CR/\ker \varphi_I
  \end{CD}
\end{equation}
Therefore, we conclude that $\iota_{!}$ is injective.

\begin{remark}
Let $I = \Phi^+$. If $I'$ contains all simple roots, then the formula in \eqref{eq:PoincaredualNilpotent} coincides with the usual definition of the Poincar\'e dual of the subvariety $\Hess(N,I')$ (\cite[Corollary~3.9]{AFZ}). 
\end{remark}

The main theorem is to construct an additive basis $\mathcal{B}$ for $H^*(\Hess(N,I))$ such that $\mathcal{B}$ extends the set $\{[\Hess(N,I')] \in H^*(\Hess(N,I)) \mid I' \subset I \}$. 
In order to construct an additive basis, we need a notion of Hessenberg functions as described below. 

\begin{definition} \label{definition;decomposition}
Let $\e_1,\e_2,\ldots,\e_n$ be the exponents\footnote{For the list of exponents $e_1,\ldots,e_n$, see for example \cite[p.59 Table~1 and p.81 Theorem~3.19]{Hum1990}.} of the Weyl group $W$. 
We define a decomposition $\Phi^+=\coprod_{i=1}^n \ \Phi^+_i$ as follows. 
For $i =1,2,\ldots,n$, $\Phi^+_i$ is a set of positive roots $\alpha_{i,i+1}, \alpha_{i,i+2}, \ldots, \alpha_{i,i+\e_i}$ such that 
\begin{align}
&\alpha_{i,i+1}=\alpha_i \ \ \ {\rm the \ simple \ root}, \label{eq:decomposition1} \\
&\alpha_{i,j} \lessdot \alpha_{i,j+1} \ \ {\rm for \ any} \ j \ {\rm with} \ i<j<i+\e_i. \label{eq:decomposition2}
\end{align}
Here, we denote the covering relation by the symbol $\lessdot$, namely, there is no element $\beta \in \Phi^+$ such that $\alpha_{i,j} < \beta < \alpha_{i,j+1}$.
\end{definition}

\begin{definition} \label{definition;Hessenberg function}
Let $\Phi^+=\coprod_{i=1}^n \ \Phi^+_i$ be a decomposition in Definition~\ref{definition;decomposition}.
Then the \textbf{Hessenberg function $h_I: \{1,2,\ldots,n\} \to \Z_{\geq 0}$ associated with a lower ideal $I$} is defined by 
\begin{equation} \label{eq:Hessft}
h_I(i):=i+ \# (I \cap \Phi^+_i) 
\end{equation}
for $1 \leq i \leq n$.
\end{definition}

\begin{example}
Consider the $A_3$ root system with $\Phi^+=\{x_i-x_j \mid 1 \leq i < j \leq 4 \}$. We take the simple roots $\alpha_1, \alpha_2, \alpha_3$ as $\alpha_i=x_i-x_{i+1}$ for $1 \leq i \leq 3$.
\begin{enumerate}
\item Let $e_1=3, e_2=2, e_3=1$ and define 
\begin{align*}
\Phi^+_1=\{x_1-x_2, x_1-x_3, x_1-x_4 \}, \Phi^+_2=\{x_2-x_3, x_2-x_4 \}, \Phi^+_3=\{x_3-x_4 \}.
\end{align*}
Then, the decomposition $\Phi^+=\coprod_{i=1}^3 \ \Phi^+_i$ satisfies \eqref{eq:decomposition1} and \eqref{eq:decomposition2}.
If we take a lower ideal $I$ as $I=\{x_1-x_2, x_1-x_3, x_2-x_3, x_3-x_4 \}$, then the corresponding Hessenberg function $h_I$ is given by $h_I(1)=2, h_I(2)=1$, and $h_I(3)=1$.
\item Taking $e_1=1, e_2=3, e_3=2$ and  
\begin{align*}
\Phi^+_1=\{x_1-x_2 \}, \Phi^+_2=\{x_2-x_3, x_1-x_3, x_1-x_4 \}, \Phi^+_3=\{x_3-x_4, x_2-x_4 \},
\end{align*}
the decomposition $\Phi^+=\coprod_{i=1}^3 \ \Phi^+_i$ also satisfies \eqref{eq:decomposition1} and \eqref{eq:decomposition2}.
In this case, the Hessenberg function $h_I$ associated with $I=\{x_1-x_2, x_1-x_3, x_2-x_3, x_3-x_4 \}$ is defined by $h_I(1)=1, h_I(2)=2$, and $h_I(3)=1$.
\end{enumerate}
\end{example}

\begin{remark}
As seen in the example above, the definition of Hessenberg function $h_I$ depends on a choice of a decomposition $\Phi^+=\coprod_{i=1}^n \ \Phi^+_i$ in Definition~\ref{definition;decomposition}.
In later sections we will choose some decomposition $\Phi^+=\coprod_{i=1}^n \ \Phi^+_i$ on each type, which corresponds to the original definition of Hessenberg function (see \cite{dMPS, dMS, Tym06}).
\end{remark}

Let $\Phi^+=\coprod_{i=1}^n \ \Phi^+_i$ be a decomposition in Definition~\ref{definition;decomposition}.
Let $h$ be the Hessenberg function associated with a lower ideal $I$. 
Then, we denote by $\Hess(N,h)$ the regular nilpotent Hessenberg variety $\Hess(N,I)$.
We can rewrite \eqref{eq:PoinHess0} as
\begin{equation} \label{eq:PoinHess}
\Poin(\Hess(N,h),\sqrt\t)=\prod_{i=1}^n (1+\t+\t^2+\cdots+\t^{h(i)-i}).
\end{equation}
In particular, one has
$$
\dim_{\C} \Hess(N,h)=\sum_{i=1}^{n} (h(i)-i).
$$ 
For any two Hessenberg functions $h$ and $h'$, we define $h' \subset h$ if $h'(i) \leq h(i)$ for any $i=1,\ldots,n$.
Then, one can easily see that 
\begin{equation} \label{eq:inclusion}
h' \subset h \iff \Hess(N,h') \subset \Hess(N,h).
\end{equation}
We also rewrite \eqref{eq:PoincaredualNilpotent} as
\begin{equation} \label{eq:PoincaredualNilpotent_h}
[\Hess(N,h')]=\frac{|W_{h'}|}{|W_{h}|}\prod_{i=1}^{n} \prod_{j=h'(i)+1}^{h(i)} e(L_{\alpha_{i,j}})|_{\Hess(N,h)} \ \ \ \ \ \ {\rm in} \ H^*(\Hess(N,h)),
\end{equation}
where $W_h$ is the parabolic subgroup generated by simple reflections $s_i$ for $i$ with $h(i)>i$.
In particular, the coefficient $\frac{|W_{h'}|}{|W_{h}|}$ in \eqref{eq:PoincaredualNilpotent_h} is equal to $1$ whenever $h(i)>i$ and $h'(i)>i$ for all $i=1,\ldots,n$.

\begin{remark}
We use two notations $\Hess(N,I)$ and $\Hess(N,h)$ for regular nilpotent Hessenberg varieties.
The first notation $\Hess(N,I)$ is useful to uniformly define Hessenberg varieties across Lie types, whereas the second notation $\Hess(N,h)$ is needed for a construction of an additive basis for $H^*(\Hess(N,h))$ which will be explained in next section. 
\end{remark}

\bigskip

\section{A preliminary lemma in commutative algebra}
\label{section:preliminaries}

Our arguments depend heavily on techniques from commutative algebra.
We begin with the definition of a regular sequence.

\begin{definition}
Let $S$ be a ring. A sequence $\theta_1,\ldots,\theta_r \in S$ is called a \textbf{regular sequence} of $S$ if it satisfies the following two conditions:
\begin{itemize}
\item[(i)] $\theta_i$ is not a zero-divisor in $S/(\theta_1,\ldots,\theta_{i-1})$ for $i = 1, \ldots , r$,
\item[(ii)] $S/(\theta_1,\ldots,\theta_r) \neq 0$.
\end{itemize}
\end{definition}

Recall that a finitely generated graded $\R$-algebra $R$ is Artinian if and only if it is a finite-dimensional vector space over $\R$.
Let $S=\R[x_1,\ldots,x_n]$ be a polynomial ring with $\deg x_i=1$ for any $i=1,\ldots,n$ and $I$ a homogeneous ideal of $S$.
Then, the Hilbert series of the quotient graded ring $R:=S/I$ is defined to be 
$$
\Hilb(R,\t)=\sum_{i}(\dim_{\R} R_i) \, \t^i
$$
where $R_i$ is the degree $i$ piece of $R$. 
The graded $\R$-algebra $R=S/I$ is a \textbf{complete intersection} if $I$ is generated by a regular sequence of $S$.
Let $f_1,\ldots,f_i$ be a regular sequence of $S$. 
It is known that if $R=S/(f_1,\ldots,f_i)$ is complete intersection, then its Krull dimension is $n-i$.
In particular, $R$ is Artinian if and only if $i=n$. 
The following fact is well-known in commutative algebra (see e.g. \cite[p.35]{Stanley1996}.)

\begin{lemma} \label{lemma:wellknown}
If the quotient ring $R=S/I$ is Artinian and $I$ is generated by $n$ homogeneous polynomials $f_1,\ldots,f_n$, then $R$ is complete intersection.
Moreover, the Hilbert series of $R$ is given by
$$
\Hilb(R,\t)=\prod_{i=1}^n (1+\t+\t^2+\cdots+\t^{\deg f_i-1}).
$$
\end{lemma}

The following lemma is useful for proving the main theorem.

\begin{lemma} \label{lemma:key}
Let $g_1,\ldots,g_n \in S$ be a regular sequence of $S$.
Assume that $g_n=g'_n \cdot g''_n$ for some polynomials $g'_n, g''_n \in S$.
Then, the map given by multiplication of $g''_n$ 
$$
\times g''_n: S/(g_1,\ldots,g_{n-1},g'_n) \to S/(g_1,\ldots,g_{n-1},g_n)
$$
is injective.
\end{lemma}

\begin{proof}
Since $g_n$ is not a zero-divisor in $S/(g_1,\ldots,g_{n-1})$, the map given by multiplication of $g_n$ 
$$
\times g_n: S/(g_1,\ldots,g_{n-1}) \xrightarrow{\times g_n''} S/(g_1,\ldots,g_{n-1}) \xrightarrow{\times g_n'} S/(g_1,\ldots,g_{n-1})
$$
is injective, and hence the multiplication map 
\begin{equation} \label{eq:keyproof}
\times g''_n: S/(g_1,\ldots,g_{n-1}) \to S/(g_1,\ldots,g_{n-1})
\end{equation}
is also injective. 
Let $f \in S$ such that $f \cdot g_n'' \in (g_1,\ldots,g_{n-1},g_n)$.
Then it is enough to show that $f \in (g_1,\ldots,g_{n-1},g'_n)$.
If we write $f \cdot g_n''=\sum_{i=1}^n F_i \cdot g_i$ for some polynomials $F_i \in S$, then we have $(f-F_n \cdot g'_n) \cdot g''_n =\sum_{i=1}^{n-1} F_i \cdot g_i \in (g_1,\ldots,g_{n-1})$.
By the injectivity of \eqref{eq:keyproof} we have $f-F_n \cdot g'_n \in (g_1,\ldots,g_{n-1})$ and hence $f \in (g_1,\ldots,g_{n-1},g'_n)$, as desired.
\end{proof}

\bigskip

\section{Main theorem}
\label{section:Main theorem}

In this section we discuss in detail Theorem~\ref{theorem_intro} for type $A_{n-1}$. 
The proof of Theorem~\ref{theorem_intro} for types $B,C,G$ is similar to the proof for
type $A$. We leave the proof for types $B,C,G$ as an exercise for the reader.
In what follows, we frequently use the symbol $[n]:=\{1,2,\ldots,n\}$.

A \textbf{Hessenberg function for type $A_{n-1}$} is defined to be a function $h: [n] \to [n]$ satisfying the following two conditions
\begin{enumerate}
\item $h(1) \leq h(2) \leq \cdots \leq h(n)$, 
\item $h(i) \geq i$ \ \ \ for $i \in [n]$.
\end{enumerate}
Note that $h(n)=n$ by the definition.
We frequently write a Hessenberg function by listing its values in sequence, that is,
$h = (h(1), h(2), \ldots , h(n))$.
It is useful to express a Hessenberg function $h$ pictorially by drawing a configuration of boxes on a square grid of size $n \times n$ whose shaded boxes consist of boxes in the $i$-th row and the $j$-th column such that $j \leq h(i)$ for $i, j \in [n]$.

\begin{example}
Let $n=5$. Then, $h=(3,5,5,5,5)$ is a Hessenberg function for type $A_4$ and the configuration of the shaded boxes is shown in Figure~\ref{picture:typeAHessenbergFunction}.
\begin{figure}[h]
\begin{center}
\begin{picture}(75,75)
\put(0,63){\colorbox{gray}}
\put(0,67){\colorbox{gray}}
\put(0,72){\colorbox{gray}}
\put(4,63){\colorbox{gray}}
\put(4,67){\colorbox{gray}}
\put(4,72){\colorbox{gray}}
\put(9,63){\colorbox{gray}}
\put(9,67){\colorbox{gray}}
\put(9,72){\colorbox{gray}}

\put(15,63){\colorbox{gray}}
\put(15,67){\colorbox{gray}}
\put(15,72){\colorbox{gray}}
\put(19,63){\colorbox{gray}}
\put(19,67){\colorbox{gray}}
\put(19,72){\colorbox{gray}}
\put(24,63){\colorbox{gray}}
\put(24,67){\colorbox{gray}}
\put(24,72){\colorbox{gray}}

\put(30,63){\colorbox{gray}}
\put(30,67){\colorbox{gray}}
\put(30,72){\colorbox{gray}}
\put(34,63){\colorbox{gray}}
\put(34,67){\colorbox{gray}}
\put(34,72){\colorbox{gray}}
\put(39,63){\colorbox{gray}}
\put(39,67){\colorbox{gray}}
\put(39,72){\colorbox{gray}}

%

\put(0,48){\colorbox{gray}}
\put(0,52){\colorbox{gray}}
\put(0,57){\colorbox{gray}}
\put(4,48){\colorbox{gray}}
\put(4,52){\colorbox{gray}}
\put(4,57){\colorbox{gray}}
\put(9,48){\colorbox{gray}}
\put(9,52){\colorbox{gray}}
\put(9,57){\colorbox{gray}}

\put(15,48){\colorbox{gray}}
\put(15,52){\colorbox{gray}}
\put(15,57){\colorbox{gray}}
\put(19,48){\colorbox{gray}}
\put(19,52){\colorbox{gray}}
\put(19,57){\colorbox{gray}}
\put(24,48){\colorbox{gray}}
\put(24,52){\colorbox{gray}}
\put(24,57){\colorbox{gray}}

\put(30,48){\colorbox{gray}}
\put(30,52){\colorbox{gray}}
\put(30,57){\colorbox{gray}}
\put(34,48){\colorbox{gray}}
\put(34,52){\colorbox{gray}}
\put(34,57){\colorbox{gray}}
\put(39,48){\colorbox{gray}}
\put(39,52){\colorbox{gray}}
\put(39,57){\colorbox{gray}}

\put(45,48){\colorbox{gray}}
\put(45,52){\colorbox{gray}}
\put(45,57){\colorbox{gray}}
\put(49,48){\colorbox{gray}}
\put(49,52){\colorbox{gray}}
\put(49,57){\colorbox{gray}}
\put(54,48){\colorbox{gray}}
\put(54,52){\colorbox{gray}}
\put(54,57){\colorbox{gray}}

\put(60,48){\colorbox{gray}}
\put(60,52){\colorbox{gray}}
\put(60,57){\colorbox{gray}}
\put(64,48){\colorbox{gray}}
\put(64,52){\colorbox{gray}}
\put(64,57){\colorbox{gray}}
\put(69,48){\colorbox{gray}}
\put(69,52){\colorbox{gray}}
\put(69,57){\colorbox{gray}}

\put(0,33){\colorbox{gray}}
\put(0,37){\colorbox{gray}}
\put(0,42){\colorbox{gray}}
\put(4,33){\colorbox{gray}}
\put(4,37){\colorbox{gray}}
\put(4,42){\colorbox{gray}}
\put(9,33){\colorbox{gray}}
\put(9,37){\colorbox{gray}}
\put(9,42){\colorbox{gray}}

\put(15,33){\colorbox{gray}}
\put(15,37){\colorbox{gray}}
\put(15,42){\colorbox{gray}}
\put(19,33){\colorbox{gray}}
\put(19,37){\colorbox{gray}}
\put(19,42){\colorbox{gray}}
\put(24,33){\colorbox{gray}}
\put(24,37){\colorbox{gray}}
\put(24,42){\colorbox{gray}}

\put(30,33){\colorbox{gray}}
\put(30,37){\colorbox{gray}}
\put(30,42){\colorbox{gray}}
\put(34,33){\colorbox{gray}}
\put(34,37){\colorbox{gray}}
\put(34,42){\colorbox{gray}}
\put(39,33){\colorbox{gray}}
\put(39,37){\colorbox{gray}}
\put(39,42){\colorbox{gray}}

\put(45,33){\colorbox{gray}}
\put(45,37){\colorbox{gray}}
\put(45,42){\colorbox{gray}}
\put(49,33){\colorbox{gray}}
\put(49,37){\colorbox{gray}}
\put(49,42){\colorbox{gray}}
\put(54,33){\colorbox{gray}}
\put(54,37){\colorbox{gray}}
\put(54,42){\colorbox{gray}}

\put(60,33){\colorbox{gray}}
\put(60,37){\colorbox{gray}}
\put(60,42){\colorbox{gray}}
\put(64,33){\colorbox{gray}}
\put(64,37){\colorbox{gray}}
\put(64,42){\colorbox{gray}}
\put(69,33){\colorbox{gray}}
\put(69,37){\colorbox{gray}}
\put(69,42){\colorbox{gray}}

\put(0,18){\colorbox{gray}}
\put(0,22){\colorbox{gray}}
\put(0,27){\colorbox{gray}}
\put(4,18){\colorbox{gray}}
\put(4,22){\colorbox{gray}}
\put(4,27){\colorbox{gray}}
\put(9,18){\colorbox{gray}}
\put(9,22){\colorbox{gray}}
\put(9,27){\colorbox{gray}}

\put(15,18){\colorbox{gray}}
\put(15,22){\colorbox{gray}}
\put(15,27){\colorbox{gray}}
\put(19,18){\colorbox{gray}}
\put(19,22){\colorbox{gray}}
\put(19,27){\colorbox{gray}}
\put(24,18){\colorbox{gray}}
\put(24,22){\colorbox{gray}}
\put(24,27){\colorbox{gray}}

\put(30,18){\colorbox{gray}}
\put(30,22){\colorbox{gray}}
\put(30,27){\colorbox{gray}}
\put(34,18){\colorbox{gray}}
\put(34,22){\colorbox{gray}}
\put(34,27){\colorbox{gray}}
\put(39,18){\colorbox{gray}}
\put(39,22){\colorbox{gray}}
\put(39,27){\colorbox{gray}}

\put(45,18){\colorbox{gray}}
\put(45,22){\colorbox{gray}}
\put(45,27){\colorbox{gray}}
\put(49,18){\colorbox{gray}}
\put(49,22){\colorbox{gray}}
\put(49,27){\colorbox{gray}}
\put(54,18){\colorbox{gray}}
\put(54,22){\colorbox{gray}}
\put(54,27){\colorbox{gray}}

\put(60,18){\colorbox{gray}}
\put(60,22){\colorbox{gray}}
\put(60,27){\colorbox{gray}}
\put(64,18){\colorbox{gray}}
\put(64,22){\colorbox{gray}}
\put(64,27){\colorbox{gray}}
\put(69,18){\colorbox{gray}}
\put(69,22){\colorbox{gray}}
\put(69,27){\colorbox{gray}}

\put(0,3){\colorbox{gray}}
\put(0,7){\colorbox{gray}}
\put(0,12){\colorbox{gray}}
\put(4,3){\colorbox{gray}}
\put(4,7){\colorbox{gray}}
\put(4,12){\colorbox{gray}}
\put(9,3){\colorbox{gray}}
\put(9,7){\colorbox{gray}}
\put(9,12){\colorbox{gray}}

\put(15,3){\colorbox{gray}}
\put(15,7){\colorbox{gray}}
\put(15,12){\colorbox{gray}}
\put(19,3){\colorbox{gray}}
\put(19,7){\colorbox{gray}}
\put(19,12){\colorbox{gray}}
\put(24,3){\colorbox{gray}}
\put(24,7){\colorbox{gray}}
\put(24,12){\colorbox{gray}}

\put(30,3){\colorbox{gray}}
\put(30,7){\colorbox{gray}}
\put(30,12){\colorbox{gray}}
\put(34,3){\colorbox{gray}}
\put(34,7){\colorbox{gray}}
\put(34,12){\colorbox{gray}}
\put(39,3){\colorbox{gray}}
\put(39,7){\colorbox{gray}}
\put(39,12){\colorbox{gray}}

\put(45,3){\colorbox{gray}}
\put(45,7){\colorbox{gray}}
\put(45,12){\colorbox{gray}}
\put(49,3){\colorbox{gray}}
\put(49,7){\colorbox{gray}}
\put(49,12){\colorbox{gray}}
\put(54,3){\colorbox{gray}}
\put(54,7){\colorbox{gray}}
\put(54,12){\colorbox{gray}}

\put(60,3){\colorbox{gray}}
\put(60,7){\colorbox{gray}}
\put(60,12){\colorbox{gray}}
\put(64,3){\colorbox{gray}}
\put(64,7){\colorbox{gray}}
\put(64,12){\colorbox{gray}}
\put(69,3){\colorbox{gray}}
\put(69,7){\colorbox{gray}}
\put(69,12){\colorbox{gray}}

\put(0,0){\framebox(15,15)}
\put(15,0){\framebox(15,15)}
\put(30,0){\framebox(15,15)}
\put(45,0){\framebox(15,15)}
\put(60,0){\framebox(15,15)}
\put(0,15){\framebox(15,15)}
\put(15,15){\framebox(15,15)}
\put(30,15){\framebox(15,15)}
\put(45,15){\framebox(15,15)}
\put(60,15){\framebox(15,15)}
\put(0,30){\framebox(15,15)}
\put(15,30){\framebox(15,15)}
\put(30,30){\framebox(15,15)}
\put(45,30){\framebox(15,15)}
\put(60,30){\framebox(15,15)}
\put(0,45){\framebox(15,15)}
\put(15,45){\framebox(15,15)}
\put(30,45){\framebox(15,15)}
\put(45,45){\framebox(15,15)}
\put(60,45){\framebox(15,15)}
\put(0,60){\framebox(15,15)}
\put(15,60){\framebox(15,15)}
\put(30,60){\framebox(15,15)}
\put(45,60){\framebox(15,15)}
\put(60,60){\framebox(15,15)}
\end{picture}
\end{center}
\vspace{-10pt}
\caption{The configuration corresponding to $h=(3,5,5,5,5)$.}
\label{picture:typeAHessenbergFunction}
\end{figure}
\end{example}

Let us denote a positive root by
\begin{equation*} 
\alpha_{i,j}=x_i-x_j \ \ \ \ \ {\rm for} \ 1 \leq i < j \leq n
\end{equation*}
and define $\Phi^+_i=\{\alpha_{i,j} \mid i < j \leq n \}$.
Then, one can see that the set of lower ideals $I \subset \Phi^+_{A_{n-1}}=\coprod_{i=1}^{n-1} \ \Phi^+_i$ and the set of Hessenberg functions $h$ for type $A_{n-1}$ are in one-to-one correspondence which sends $I$ to $h_I$ in \eqref{eq:Hessft}.
Note that we extend a Hessenberg function $h_I:[n-1] \to [n-1]$ to $h_I:[n] \to [n]$ with $h_I(n)=n$. 

\begin{example}
The lower ideal associated with a Hessenberg function $h=(3,5,5,5,5)$ consists of positive roots in shaded boxes of $h$ as shown in Figure~\ref{picture:TypeALowerIdeal}.
\begin{figure}[h]
\begin{center}
\begin{picture}(150,50)
\put(0,43){\colorbox{gray}}
\put(0,47){\colorbox{gray}}
\put(5,43){\colorbox{gray}}
\put(5,47){\colorbox{gray}}
\put(10,43){\colorbox{gray}}
\put(10,47){\colorbox{gray}}
\put(15,43){\colorbox{gray}}
\put(15,47){\colorbox{gray}}
\put(20,43){\colorbox{gray}}
\put(20,47){\colorbox{gray}}
\put(24,43){\colorbox{gray}}
\put(24,47){\colorbox{gray}}

\put(30,43){\colorbox{gray}}
\put(30,47){\colorbox{gray}}
\put(35,43){\colorbox{gray}}
\put(35,47){\colorbox{gray}}
\put(40,43){\colorbox{gray}}
\put(40,47){\colorbox{gray}}
\put(45,43){\colorbox{gray}}
\put(45,47){\colorbox{gray}}
\put(50,43){\colorbox{gray}}
\put(50,47){\colorbox{gray}}
\put(54,43){\colorbox{gray}}
\put(54,47){\colorbox{gray}}

\put(60,43){\colorbox{gray}}
\put(60,47){\colorbox{gray}}
\put(65,43){\colorbox{gray}}
\put(65,47){\colorbox{gray}}
\put(70,43){\colorbox{gray}}
\put(70,47){\colorbox{gray}}
\put(75,43){\colorbox{gray}}
\put(75,47){\colorbox{gray}}
\put(80,43){\colorbox{gray}}
\put(80,47){\colorbox{gray}}
\put(84,43){\colorbox{gray}}
\put(84,47){\colorbox{gray}}

\put(0,33){\colorbox{gray}}
\put(0,37){\colorbox{gray}}
\put(5,33){\colorbox{gray}}
\put(5,37){\colorbox{gray}}
\put(10,33){\colorbox{gray}}
\put(10,37){\colorbox{gray}}
\put(15,33){\colorbox{gray}}
\put(15,37){\colorbox{gray}}
\put(20,33){\colorbox{gray}}
\put(20,37){\colorbox{gray}}
\put(24,33){\colorbox{gray}}
\put(24,37){\colorbox{gray}}

\put(30,33){\colorbox{gray}}
\put(30,37){\colorbox{gray}}
\put(35,33){\colorbox{gray}}
\put(35,37){\colorbox{gray}}
\put(40,33){\colorbox{gray}}
\put(40,37){\colorbox{gray}}
\put(45,33){\colorbox{gray}}
\put(45,37){\colorbox{gray}}
\put(50,33){\colorbox{gray}}
\put(50,37){\colorbox{gray}}
\put(54,33){\colorbox{gray}}
\put(54,37){\colorbox{gray}}

\put(60,33){\colorbox{gray}}
\put(60,37){\colorbox{gray}}
\put(65,33){\colorbox{gray}}
\put(65,37){\colorbox{gray}}
\put(70,33){\colorbox{gray}}
\put(70,37){\colorbox{gray}}
\put(75,33){\colorbox{gray}}
\put(75,37){\colorbox{gray}}
\put(80,33){\colorbox{gray}}
\put(80,37){\colorbox{gray}}
\put(84,33){\colorbox{gray}}
\put(84,37){\colorbox{gray}}

\put(90,33){\colorbox{gray}}
\put(90,37){\colorbox{gray}}
\put(95,33){\colorbox{gray}}
\put(95,37){\colorbox{gray}}
\put(100,33){\colorbox{gray}}
\put(100,37){\colorbox{gray}}
\put(105,33){\colorbox{gray}}
\put(105,37){\colorbox{gray}}
\put(110,33){\colorbox{gray}}
\put(110,37){\colorbox{gray}}
\put(114,33){\colorbox{gray}}
\put(114,37){\colorbox{gray}}

\put(120,33){\colorbox{gray}}
\put(120,37){\colorbox{gray}}
\put(125,33){\colorbox{gray}}
\put(125,37){\colorbox{gray}}
\put(130,33){\colorbox{gray}}
\put(130,37){\colorbox{gray}}
\put(135,33){\colorbox{gray}}
\put(135,37){\colorbox{gray}}
\put(140,33){\colorbox{gray}}
\put(140,37){\colorbox{gray}}
\put(144,33){\colorbox{gray}}
\put(144,37){\colorbox{gray}}

\put(0,23){\colorbox{gray}}
\put(0,27){\colorbox{gray}}
\put(5,23){\colorbox{gray}}
\put(5,27){\colorbox{gray}}
\put(10,23){\colorbox{gray}}
\put(10,27){\colorbox{gray}}
\put(15,23){\colorbox{gray}}
\put(15,27){\colorbox{gray}}
\put(20,23){\colorbox{gray}}
\put(20,27){\colorbox{gray}}
\put(24,23){\colorbox{gray}}
\put(24,27){\colorbox{gray}}

\put(30,23){\colorbox{gray}}
\put(30,27){\colorbox{gray}}
\put(35,23){\colorbox{gray}}
\put(35,27){\colorbox{gray}}
\put(40,23){\colorbox{gray}}
\put(40,27){\colorbox{gray}}
\put(45,23){\colorbox{gray}}
\put(45,27){\colorbox{gray}}
\put(50,23){\colorbox{gray}}
\put(50,27){\colorbox{gray}}
\put(54,23){\colorbox{gray}}
\put(54,27){\colorbox{gray}}

\put(60,23){\colorbox{gray}}
\put(60,27){\colorbox{gray}}
\put(65,23){\colorbox{gray}}
\put(65,27){\colorbox{gray}}
\put(70,23){\colorbox{gray}}
\put(70,27){\colorbox{gray}}
\put(75,23){\colorbox{gray}}
\put(75,27){\colorbox{gray}}
\put(80,23){\colorbox{gray}}
\put(80,27){\colorbox{gray}}
\put(84,23){\colorbox{gray}}
\put(84,27){\colorbox{gray}}

\put(90,23){\colorbox{gray}}
\put(90,27){\colorbox{gray}}
\put(95,23){\colorbox{gray}}
\put(95,27){\colorbox{gray}}
\put(100,23){\colorbox{gray}}
\put(100,27){\colorbox{gray}}
\put(105,23){\colorbox{gray}}
\put(105,27){\colorbox{gray}}
\put(110,23){\colorbox{gray}}
\put(110,27){\colorbox{gray}}
\put(114,23){\colorbox{gray}}
\put(114,27){\colorbox{gray}}

\put(120,23){\colorbox{gray}}
\put(120,27){\colorbox{gray}}
\put(125,23){\colorbox{gray}}
\put(125,27){\colorbox{gray}}
\put(130,23){\colorbox{gray}}
\put(130,27){\colorbox{gray}}
\put(135,23){\colorbox{gray}}
\put(135,27){\colorbox{gray}}
\put(140,23){\colorbox{gray}}
\put(140,27){\colorbox{gray}}
\put(144,23){\colorbox{gray}}
\put(144,27){\colorbox{gray}}

\put(0,13){\colorbox{gray}}
\put(0,17){\colorbox{gray}}
\put(5,13){\colorbox{gray}}
\put(5,17){\colorbox{gray}}
\put(10,13){\colorbox{gray}}
\put(10,17){\colorbox{gray}}
\put(15,13){\colorbox{gray}}
\put(15,17){\colorbox{gray}}
\put(20,13){\colorbox{gray}}
\put(20,17){\colorbox{gray}}
\put(24,13){\colorbox{gray}}
\put(24,17){\colorbox{gray}}

\put(30,13){\colorbox{gray}}
\put(30,17){\colorbox{gray}}
\put(35,13){\colorbox{gray}}
\put(35,17){\colorbox{gray}}
\put(40,13){\colorbox{gray}}
\put(40,17){\colorbox{gray}}
\put(45,13){\colorbox{gray}}
\put(45,17){\colorbox{gray}}
\put(50,13){\colorbox{gray}}
\put(50,17){\colorbox{gray}}
\put(54,13){\colorbox{gray}}
\put(54,17){\colorbox{gray}}

\put(60,13){\colorbox{gray}}
\put(60,17){\colorbox{gray}}
\put(65,13){\colorbox{gray}}
\put(65,17){\colorbox{gray}}
\put(70,13){\colorbox{gray}}
\put(70,17){\colorbox{gray}}
\put(75,13){\colorbox{gray}}
\put(75,17){\colorbox{gray}}
\put(80,13){\colorbox{gray}}
\put(80,17){\colorbox{gray}}
\put(84,13){\colorbox{gray}}
\put(84,17){\colorbox{gray}}

\put(90,13){\colorbox{gray}}
\put(90,17){\colorbox{gray}}
\put(95,13){\colorbox{gray}}
\put(95,17){\colorbox{gray}}
\put(100,13){\colorbox{gray}}
\put(100,17){\colorbox{gray}}
\put(105,13){\colorbox{gray}}
\put(105,17){\colorbox{gray}}
\put(110,13){\colorbox{gray}}
\put(110,17){\colorbox{gray}}
\put(114,13){\colorbox{gray}}
\put(114,17){\colorbox{gray}}

\put(120,13){\colorbox{gray}}
\put(120,17){\colorbox{gray}}
\put(125,13){\colorbox{gray}}
\put(125,17){\colorbox{gray}}
\put(130,13){\colorbox{gray}}
\put(130,17){\colorbox{gray}}
\put(135,13){\colorbox{gray}}
\put(135,17){\colorbox{gray}}
\put(140,13){\colorbox{gray}}
\put(140,17){\colorbox{gray}}
\put(144,13){\colorbox{gray}}
\put(144,17){\colorbox{gray}}

\put(0,3){\colorbox{gray}}
\put(0,7){\colorbox{gray}}
\put(5,3){\colorbox{gray}}
\put(5,7){\colorbox{gray}}
\put(10,3){\colorbox{gray}}
\put(10,7){\colorbox{gray}}
\put(15,3){\colorbox{gray}}
\put(15,7){\colorbox{gray}}
\put(20,3){\colorbox{gray}}
\put(20,7){\colorbox{gray}}
\put(24,3){\colorbox{gray}}
\put(24,7){\colorbox{gray}}

\put(30,3){\colorbox{gray}}
\put(30,7){\colorbox{gray}}
\put(35,3){\colorbox{gray}}
\put(35,7){\colorbox{gray}}
\put(40,3){\colorbox{gray}}
\put(40,7){\colorbox{gray}}
\put(45,3){\colorbox{gray}}
\put(45,7){\colorbox{gray}}
\put(50,3){\colorbox{gray}}
\put(50,7){\colorbox{gray}}
\put(54,3){\colorbox{gray}}
\put(54,7){\colorbox{gray}}

\put(60,3){\colorbox{gray}}
\put(60,7){\colorbox{gray}}
\put(65,3){\colorbox{gray}}
\put(65,7){\colorbox{gray}}
\put(70,3){\colorbox{gray}}
\put(70,7){\colorbox{gray}}
\put(75,3){\colorbox{gray}}
\put(75,7){\colorbox{gray}}
\put(80,3){\colorbox{gray}}
\put(80,7){\colorbox{gray}}
\put(84,3){\colorbox{gray}}
\put(84,7){\colorbox{gray}}

\put(90,3){\colorbox{gray}}
\put(90,7){\colorbox{gray}}
\put(95,3){\colorbox{gray}}
\put(95,7){\colorbox{gray}}
\put(100,3){\colorbox{gray}}
\put(100,7){\colorbox{gray}}
\put(105,3){\colorbox{gray}}
\put(105,7){\colorbox{gray}}
\put(110,3){\colorbox{gray}}
\put(110,7){\colorbox{gray}}
\put(114,3){\colorbox{gray}}
\put(114,7){\colorbox{gray}}

\put(120,3){\colorbox{gray}}
\put(120,7){\colorbox{gray}}
\put(125,3){\colorbox{gray}}
\put(125,7){\colorbox{gray}}
\put(130,3){\colorbox{gray}}
\put(130,7){\colorbox{gray}}
\put(135,3){\colorbox{gray}}
\put(135,7){\colorbox{gray}}
\put(140,3){\colorbox{gray}}
\put(140,7){\colorbox{gray}}
\put(144,3){\colorbox{gray}}
\put(144,7){\colorbox{gray}}

\put(0,40){\framebox(30,10)} 
\put(30,40){\framebox(30,10){\tiny $x_1-x_2$}}
\put(60,40){\framebox(30,10){\tiny $x_1-x_3$}}
\put(90,40){\framebox(30,10){\tiny $x_1-x_4$}}
\put(120,40){\framebox(30,10){\tiny $x_1-x_5$}}

\put(0,30){\framebox(30,10)} 
\put(30,30){\framebox(30,10)}
\put(60,30){\framebox(30,10){\tiny $x_2-x_3$}}
\put(90,30){\framebox(30,10){\tiny $x_2-x_4$}}
\put(120,30){\framebox(30,10){\tiny $x_2-x_5$}}

\put(0,20){\framebox(30,10)} 
\put(30,20){\framebox(30,10)}
\put(60,20){\framebox(30,10)}
\put(90,20){\framebox(30,10){\tiny $x_3-x_4$}}
\put(120,20){\framebox(30,10){\tiny $x_3-x_5$}}

\put(0,10){\framebox(30,10)} 
\put(30,10){\framebox(30,10)}
\put(60,10){\framebox(30,10)}
\put(90,10){\framebox(30,10)}
\put(120,10){\framebox(30,10){\tiny $x_4-x_5$}}

\put(0,0){\framebox(30,10)} 
\put(30,0){\framebox(30,10)}
\put(60,0){\framebox(30,10)}
\put(90,0){\framebox(30,10)}
\put(120,0){\framebox(30,10)}
\end{picture}
\end{center}
\vspace{-10pt}
\caption{The lower ideal associated with $h=(3,5,5,5,5)$.}
\label{picture:TypeALowerIdeal}
\end{figure} 
\end{example}

We now explain the polynomials $f^{A_{n-1}}_{i,j}$ given in \cite{AHHM} which are used for describing the cohomorogy ring $H^*(\Hess(N,h))$. 
For $1 \leq i \leq j \leq n$, define a polynomial 
$$
f^{A_{n-1}}_{i,j} := \sum_{k=1}^i \left( \prod_{\ell=i+1}^{j} (x_k-x_\ell)\right) x_k
$$
with the convention $\prod_{\ell=i+1}^{j} (x_k-x_\ell)=1$ whenever $j=i$.

\begin{theorem}$($\cite[Theorem~A]{AHHM}, see also \cite[Corollary~10.4]{AHMMS}$)$ \label{theorem:cohomologyA}
Let $h$ be a Hessenberg function for type $A_{n-1}$ and $\Hess(N,h)$ the associated regular nilpotent Hessenberg variety.
Then there is an isomorphism of graded $\R$-algebras
\begin{equation*}
H^*(\Hess(N,h)) \cong \R[x_1,\ldots,x_n]/(f^{A_{n-1}}_{1,h(1)}, \ldots, f^{A_{n-1}}_{n,h(n)})
\end{equation*}
which sends each root $\alpha$ to the Euler class $e(L_{\alpha})|_{\Hess(N,h)}$.
\end{theorem}

By abuse of notation, we think of $\alpha$ as the Euler class $e(L_{\alpha})|_{\Hess(N,h)}$ for simplicity.
We restate Theorem~\ref{theorem_intro} for type $A_{n-1}$ as follows.

\begin{theorem} \label{theorem:basisA}
Let $h$ be a Hessenberg function for type $A_{n-1}$ and $\Hess(N,h)$ the associated regular nilpotent Hessenberg variety.
We fix a permutation $w^{(i)}$ on a set $\{i+1, i+2, \dots, h(i)\}$ for each $i=1,\ldots,n-1$.
Then, the following set  
\begin{equation} \label{eq:basisApermutation}
\left\{\prod_{i=1}^{n-1} \alpha_{i,w^{(i)}(h(i))} \cdot \alpha_{i,w^{(i)}(h(i)-1)} \cdots \alpha_{i,w^{(i)}(h(i)-\m_i+1)} \ \middle| \ 0 \leq \m_i \leq h(i)-i \right\}
\end{equation}
forms a basis for the cohomology $H^*(\Hess(N,h))$ over $\R$.
Here, we take the convention $\alpha_{i,w^{(i)}(h(i))} \cdot \alpha_{i,w^{(i)}(h(i)-1)} \cdots \alpha_{i,w^{(i)}(h(i)-\m_i+1)}=1$ whenever $\m_i=0$.
\end{theorem}

Note that a permutation $w^{(i)}$ in Theorem~\ref{theorem:basisA} determines an order on positive roots in $I \cap \Phi^+_i$.

\begin{example}
Let us consider a Hessenberg function $h=(3,5,5,5,5)$. 
We take the identity map as a permutation $w^{(i)}$ for each $1 \leq i \leq 5$. 
Then, the factors of $\alpha_{i,h(i)} \cdot \alpha_{i,h(i)-1} \cdots \alpha_{i,h(i)-m_i+1}$ pictorially show the right-most $m_i$'s positive roots of shaded boxes in the $i$-th row (see Figure~\ref{picture:TypeALowerIdeal}).
For example, if we take $(m_1,m_2,m_3,m_4)=(1,2,1,0)$, then the product in \eqref{eq:basisApermutation} is equal to 
$$
(\alpha_{1,3})\cdot(\alpha_{2,5} \cdot \alpha_{2,4})\cdot(\alpha_{3,5})\cdot 1=(x_1-x_3)(x_2-x_5)(x_2-x_4)(x_3-x_5).
$$
Note that the product above is nothing but the Poincar\'e dual $[\Hess(N,h_0)]$ for $h_0=(2,3,4,5,5)$ by the formula in \eqref{eq:PoincaredualNilpotent_h}.
\end{example}

The basis in Theorem~\ref{theorem:basisA} is, in fact, an extended basis of all Poincar\'e duals $[\Hess(N,h')]$ of smaller regular nilpotent Hessenberg varieties $\Hess(N,h')$ in $H^*(\Hess(N,h))$.
We take the identity maps as permutations $w^{(i)}$ in Theorem~\ref{theorem:basisA}.
For a smaller Hessenberg function $h' \subset h$, we put $m_i=h(i)-h'(i)$ for $1 \leq i \leq n-1$.
Then the cohomology class in \eqref{eq:basisApermutation} is nothing but the Poincar\'e dual $[\Hess(N,h')]$ in $H^*(\Hess(N,h))$ up to a non-zero scalar multiplication by \eqref{eq:PoincaredualNilpotent_h}.

\begin{corollary} \label{theorem:PoincaredualA}
Let $h$ be a Hessenberg function for type $A_{n-1}$ and $\Hess(N,h)$ the associated regular nilpotent Hessenberg variety. 
Then, the set of the Poincar\'e duals 
\begin{equation*} 
\{[\Hess(N,h')] \in H^*(\Hess(N,h)) \mid h' \subset h \}
\end{equation*}
is linearly independent. 
\end{corollary}

We first prove the special case of Theorem~\ref{theorem:basisA} when $\Hess(N,h)$ is the whole flag variety $G/B$. 
For this purpose we need to slightly generalize the definition of $f_{i,j}^{A_{n-1}}$. 
For $i \in [n]$ and a vector $\mathbf{a}_i=(a_{i1}, \ldots, a_{ii}) \in \R^i$, define a polynomial 
\begin{equation*}
f^{\mathbf{a}_i}_{i,n}=f^{\mathbf{a}_i}_{i,n}(x_1,\ldots,x_n):=\sum_{k=1}^i a_{ik} \left( \prod_{\ell=i+1}^{n} (x_k-x_\ell)\right) x_k \in \R[x_1,\ldots,x_n].
\end{equation*}
For $\mathbf{A}_{(n-1)}=(\mathbf{a}_i \in \R^i \mid 1 \leq i \leq n)$ we also define a ring 
\begin{equation} \label{eq:ringA}
R^{\mathbf{A}_{(n-1)}}=R^{\mathbf{A}_{(n-1)}}(x_1,\ldots,x_n):=\R[x_1,\ldots,x_n]/(f^{\mathbf{a}_i}_{i,n} \mid 1 \leq i \leq n).
\end{equation}
If $\mathbf{a}_i=(1, \ldots, 1)$ for all $i \in [n]$, then the ring $R^{\mathbf{A}_{(n-1)}}$ is isomorphic to the cohomology ring of the flag variety for type $A_{n-1}$ by Theorem~\ref{theorem:cohomologyA}.

\begin{lemma} \label{lemma:setsudouringA}
Let $j$ be a positive integer with $2 \leq j \leq n$.
Then the following ring isomorphism 
$$
R^{\mathbf{A}_{(n-1)}}(x_1,\ldots,x_n)/(x_1-x_j) \cong R^{\mathbf{A'}_{(n-2)}}(y_1,\ldots,y_{n-1})
$$
holds for some $\mathbf{A'}_{(n-2)}=(\mathbf{a}'_i \in \R^{i} \mid 1 \leq i \leq n-1)$. 
Here, the isomorphism above is realized by sending $x_{i+1}$ to $y_i$ for $1 \leq i \leq n-1$. 
\end{lemma}

\begin{proof}
It suffices to check that the defining polynomials $f^{\mathbf{a}_i}_{i,n}=f^{\mathbf{a}_i}_{i,n}(x_1,\ldots,x_n)$ of the ring $R^{\mathbf{A}_{(n-1)}}(x_1,\ldots,x_n)$ are congruent to the defining polynomials $f^{\mathbf{a}'_{i-1}}_{i-1,n-1}(y_1,\ldots,y_{n-1})$ of $R^{\mathbf{A'}_{(n-2)}}(y_1,\ldots,y_{n-1})$ modulo $x_1-x_j$. Here we take the convention $f^{\mathbf{a}'_{i-1}}_{i-1,n-1}(y_1,\ldots,y_{n-1})=0$ whenever $i=1$. \\ 
\textbf{Case (i):} Suppose that $i=1$. By the definition of $f^{\mathbf{a}_i}_{i,n}$ we have
\begin{align*}
f^{\mathbf{a}_1}_{1,n}=a_{11} \left(\prod_{\ell=2}^{n} (x_1-x_\ell)\right) x_1 
                         \equiv 0 \ \ \ \ \ \ \ ({\rm mod} \ x_1-x_j). 
\end{align*}
\textbf{Case (ii):} Suppose that $1<i<j$. 
As in the above case we may compute
\begin{align*}
f^{\mathbf{a}_i}_{i,n}&=\sum_{k=1}^i a_{ik} \left(\prod_{\ell=i+1}^{n} (x_k-x_\ell)\right) x_k \\
                         &=a_{i1} \left(\prod_{\ell=i+1}^{n} (x_1-x_\ell)\right) x_1 + \sum_{k=2}^i a_{ik} \left(\prod_{\ell=i+1}^{n} (x_k-x_\ell)\right) x_k \\
                         &\equiv \sum_{k=2}^i a_{ik} \left(\prod_{\ell=i+1}^{n} (x_k-x_\ell)\right) x_k \ \ \ \ \ \ \ ({\rm mod} \ x_1-x_j) \\
                         &=\sum_{k=1}^{i-1} a_{i \, k+1} \left(\prod_{\ell=i}^{n-1} (y_{k}-y_{\ell})\right) y_{k}=f^{\mathbf{a}'_{i-1}}_{i-1,n-1}(y_1,\ldots,y_{n-1})
\end{align*}
where $\mathbf{a}'_{i-1}=(a_{i2}, \ldots, a_{ii})$. \\ \ 
\textbf{Case (iii):} Suppose that $j \leq i \leq n$.
Then we have
\begin{align*}
f^{\mathbf{a}_i}_{i,n}&=\sum_{k=1}^i a_{ik} \left(\prod_{\ell=i+1}^{n} (x_k-x_\ell)\right) x_k \\
                         &=a_{i1} \left(\prod_{\ell=i+1}^{n} (x_1-x_\ell)\right) x_1 + \sum_{k=2}^i a_{ik} \left(\prod_{\ell=i+1}^{n} (x_k-x_\ell)\right) x_k \\
                         &\equiv a_{i1} \left(\prod_{\ell=i+1}^{n} (x_j-x_\ell)\right) x_j + \sum_{k=2}^i a_{ik} \left(\prod_{\ell=i+1}^{n} (x_k-x_\ell)\right) x_k \ \ \ \ \ ({\rm mod} \ x_1-x_j) \\
                         &= \sum_{k=2}^i \tilde a_{ik} \left(\prod_{\ell=i+1}^{n} (x_k-x_\ell)\right) x_k \ \ \ \ \ ({\rm where} \ \tilde a_{ik}:=a_{ik}+\delta_{kj}a_{i1}) \\
                         &=\sum_{k=1}^{i-1} \tilde a_{i \, k+1} \left(\prod_{\ell=i}^{n-1} (y_{k}-y_{\ell})\right) y_{k}=f^{\mathbf{a}'_{i-1}}_{i-1,n-1}(y_1,\ldots,y_{n-1})
\end{align*}
where $\mathbf{a}'_{i-1}=(\tilde a_{i2},\ldots, \tilde a_{ii})$.

In cases (i), (ii), and (iii) above we have shown that 
\begin{equation*}
f^{\mathbf{a}_i}_{i,n}(x_1,\ldots,x_n) \equiv
\begin{cases}
0 \ \ \ &{\rm if} \ i=1 \\
f^{\mathbf{a}'_{i-1}}_{i-1,n-1}(y_1,\ldots,y_{n-1}) \ \ \ &{\rm if} \ 2 \leq i \leq n 
\end{cases}
\ \ \ \ \ ({\rm mod} \ x_1-x_j) 
\end{equation*}
for some vector $\mathbf{a}'_{i-1} \in \R^{i-1} \ (2 \leq i \leq n )$.
This shows $R^{\mathbf{A}_{(n-1)}}(x_1,\ldots,x_n)/(x_1-x_j) \cong R^{\mathbf{A'}_{(n-2)}}(y_1,\ldots,y_{n-1})$, as desired.
\end{proof}

\begin{proposition} \label{proposition:basisAFlag}
Suppose that the ring $R^{\mathbf{A}_{(n-1)}}$ in \eqref{eq:ringA} is Artinian.
We fix a permutation $w^{(i)}$ on a set $\{i+1, i+2, \dots, n\}$ for each $i=1,\ldots,n-1$.
Then, the following monomials in positive roots 
\begin{equation} \label{eq:basisAFlag}
\prod_{i=1}^{n-1} \alpha_{i,w^{(i)}(n)} \cdot \alpha_{i,w^{(i)}(n-1)} \cdots \alpha_{i,w^{(i)}(n-\m_i+1)},
\end{equation}
with $0 \leq \m_i \leq n-i$, form a basis for $R^{\mathbf{A}_{(n-1)}}$ over $\R$.
\end{proposition}

\begin{proof}
From Lemma~\ref{lemma:wellknown} the Hilbert series of $R^{\mathbf{A}_{(n-1)}}$ is equal to $\prod_{i=1}^n (1+\t+\t^2+\cdots+\t^{n-i})$. 
Hence it is enough to prove the linear independence of the monomials in \eqref{eq:basisAFlag}.
We prove this by induction on $n$.

First consider the base case $n=1$.
In this case, the only possible monomial in \eqref{eq:basisAFlag} is $1$ and $R^{\mathbf{A}_{(0)}}=\R$.
This proves the base case.

We proceed to the inductive step. 
Suppose now that $n > 1$ and that the claim holds for $n - 1$, with any allowable
choices of $\mathbf{A'}_{(n-2)}=(\mathbf{a}'_i \in \R^i \mid 1 \leq i \leq n-1)$.
Assume that there exist non-zero constants $c_{m}$ such that
\begin{equation} \label{eq:linearlyindependentA}
\sum_{m=(m_1,\ldots,m_n) \atop 0 \leq \m_i \leq n-i} c_{m} \left( \prod_{i=1}^{n-1} \alpha_{i,w^{(i)}(n)} \cdot \alpha_{i,w^{(i)}(n-1)} \cdots \alpha_{i,w^{(i)}(n-\m_i+1)} \right)=0 \ \ \ \ \ {\rm in} \ R^{\mathbf{A}_{(n-1)}}.
\end{equation}
Let $m_1^{\circ}$ be the minimal number $m_1$ such that $c_{(m_1,\ldots,m_n)} \neq 0$.
We first note that $f^{\mathbf{a}_1}_{1,n}$ is non-zero since $R^{\mathbf{A}_{(n-1)}}$ is Artinian. We may assume without loss of generality that $a_{11} = 1$.
Then we have
\begin{equation} \label{eq:f_{1,n}vanishA}
\left(\prod_{\ell=2}^n \alpha_{1,w^{(1)}(n-\ell+2)} \right)x_1=\left(\prod_{\ell=2}^n \alpha_{1,\ell} \right)x_1=f^{\mathbf{a}_1}_{1,n}=0 \ \ \ \ \  {\rm in} \  R^{\mathbf{A}_{(n-1)}}.
\end{equation}
Noting that $\prod_{i=1}^{n-1} \alpha_{i,w^{(i)}(n)} \cdot \alpha_{i,w^{(i)}(n-1)} \cdots \alpha_{i,w^{(i)}(n-\m_i+1)}$ can be written as 
\begin{equation} \label{eq:productformA}
\left( \prod_{\ell=2}^{m_1+1} \alpha_{1,w^{(1)}(n-\ell+2)} \right) \left( \prod_{i=2}^{n-1} \alpha_{i,w^{(i)}(n)} \cdot \alpha_{i,w^{(i)}(n-1)} \cdots \alpha_{i,w^{(i)}(n-\m_i+1)} \right),
\end{equation}
the monomial in \eqref{eq:productformA} multiplied by the product $\prod_{\ell=m_1^{\circ}+3}^{n}(\alpha_{1,w^{(1)}(n-\ell+2)})x_1$ is equal to
\begin{equation*} 
\begin{cases}
0 & {\rm if} \ m_1 \neq m_1^{\circ} \\
g_n'' \cdot \left( \displaystyle \prod_{i=2}^{n-1} \alpha_{i,w^{(i)}(n)} \cdot \alpha_{i,w^{(i)}(n-1)} \cdots \alpha_{i,w^{(i)}(n-\m_i+1)} \right) & {\rm if} \ m_1 = m_1^{\circ} \ 
\end{cases}
\ \ \ \ \  {\rm in} \  R^{\mathbf{A}_{(n-1)}},
\end{equation*}
by the minimality of $m_1^{\circ}$ and \eqref{eq:f_{1,n}vanishA}. 
Here, the polynomial $g''_n$ is defined by 
$$
g''_n:=\prod_{2 \leq \ell \leq n \atop \ell \neq m_1^{\circ}+2}(\alpha_{1,w^{(1)}(n-\ell+2)})x_1.
$$
Hence, multiplying the both sides of \eqref{eq:linearlyindependentA} by the product $\prod_{\ell=m_1^{\circ}+3}^{n}(\alpha_{1,w^{(1)}(n-\ell+2)})x_1$, we obtain 
\begin{equation} \label{eq:linearlyindependent2A}
g_n'' \cdot \left( \sum_{m=(m_1,\ldots,m_n) \atop 0 \leq \m_i \leq n-i, \ m_1=m_1^{\circ}} c_{m} \left( \prod_{i=2}^{n-1} \alpha_{i,w^{(i)}(n)} \cdot \alpha_{i,w^{(i)}(n-1)} \cdots \alpha_{i,w^{(i)}(n-\m_i+1)} \right) \right)=0 \ \ \ \ \ {\rm in} \ R^{\mathbf{A}_{(n-1)}}.
\end{equation}
Let $g'_n:=\alpha_{1,w^{(1)}(n-m_1^{\circ})}$ and we note that $f^{\mathbf{a}_1}_{1,n}=g'_n \cdot g''_n$.
Then, Lemma~\ref{lemma:key} leads us to the injective multiplication map 
\begin{equation} \label{eq:multiinproof}
\times g''_n: \R[x_1,\ldots,x_n]/(f^{\mathbf{a}_2}_{2,n},\ldots,f^{\mathbf{a}_n}_{n,n},g'_n) \hookrightarrow \R[x_1,\ldots,x_n]/(f^{\mathbf{a}_2}_{2,n},\ldots,f^{\mathbf{a}_n}_{n,n},f^{\mathbf{a}_1}_{1,n})=R^{\mathbf{A}_{(n-1)}}.
\end{equation}
Since $\R[x_1,\ldots,x_n]/(f^{\mathbf{a}_2}_{2,n},\ldots,f^{\mathbf{a}_n}_{n,n},g'_n)$ is isomorphic to $R^{\mathbf{A}_{(n-1)}}/(\alpha_{1,w^{(1)}(n-m_1^{\circ})})$,  
\eqref{eq:linearlyindependent2A} and the injectivity of the map in \eqref{eq:multiinproof} yield the linear relation
\begin{equation} \label{eq:linearlyindependent3A}
\sum_{m=(m_1,\ldots,m_n) \atop 0 \leq \m_i \leq n-i, \ m_1=m_1^{\circ}} c_{m} \left( \prod_{i=2}^{n-1} \alpha_{i,w^{(i)}(n)} \cdot \alpha_{i,w^{(i)}(n-1)} \cdots \alpha_{i,w^{(i)}(n-\m_i+1)} \right) =0 \ \ \ {\rm in} \ R^{\mathbf{A}_{(n-1)}}/(\alpha_{1,w^{(1)}(n-m_1^{\circ})}).
\end{equation}
However, Lemma~\ref{lemma:setsudouringA} gives the isomorphism $R^{\mathbf{A}_{(n-1)}}/(\alpha_{1,w^{(1)}(n-m_1^{\circ})}) \cong R^{\mathbf{A'}_{(n-2)}}(y_1,\ldots,y_{n-1})$ which sends $x_{i+1}$ to $y_i$ for $1 \leq i \leq n-1$.
Setting $\beta_{i,j}:=y_i-y_j=x_{i+1}-x_{j+1}=\alpha_{i+1,j+1}$ for $1 \leq i < j \leq n-1$, we can rewrite \eqref{eq:linearlyindependent3A} as
\begin{equation*} 
\hspace{-10pt}
\sum_{m=(m_1^{\circ},\ell_1,\ldots,\ell_{n-1}) \atop 0 \leq \ell_i \leq n-1-i} c_{m} \left( \prod_{i=1}^{n-2} \beta_{i,v^{(i)}(n-1)} \cdot \beta_{i,v^{(i)}(n-2)} \cdots \beta_{i,v^{(i)}((n-1)-\ell_i+1)} \right) =0 \ \ \ {\rm in} \ R^{\mathbf{A'}_{(n-2)}}(y_1,\ldots,y_{n-1}),
\end{equation*}
where $v^{(i)}$ is the permutation on the set $\{i+1, i+2, \ldots, n-1\}$ defined by $v^{(i)}(s)=w^{(i+1)}(s+1)-1$.
This is a contradiction by the inductive assumption.
Therefore, we proved the linear independence of the monomials in \eqref{eq:basisAFlag}, as desired.
\end{proof}

Now, we are in the position to give a proof of Theorem~\ref{theorem:basisA}.

\begin{proof}[Proof of Theorem~\ref{theorem:basisA}.]
By \eqref{eq:PoinHess} it is enough to prove the linear independence of the cohomology classes in \eqref{eq:basisApermutation}. 
We proceed by decreasing induction on the dimension $d_h:=\sum_{i=1}^n (h(i)-i)$ of $\Hess(N,h)$. 
The base case of the flag variety follows from Proposition~\ref{proposition:basisAFlag}. 
Suppose now that $\Hess(N,h) \subsetneq G/B$ and the claim holds for any Hessenberg function $\tilde h$ with $d_h < d_{\tilde h}$.
Since $h \neq (n,n,\ldots,n)$, we can take a Hessenberg function $\tilde h$ such that $\tilde h(k)=h(k)+1$ for some $k$ and $\tilde h(i)=h(i)$ for any $i \neq k$.
By \eqref{eq:CDGysinNilpotent} the image of a monomial in the set \eqref{eq:basisApermutation} under the Gysin map $\iota_{!}: H^*(\Hess(N,h)) \to H^*(\Hess(N,\tilde h))$ is given by 
\begin{equation*} 
\alpha_{k,\tilde h(k)} \cdot \left( \prod_{i=1}^{n-1} \alpha_{i,w^{(i)}(h(i))} \cdot \alpha_{i,w^{(i)}(h(i)-1)} \cdots \alpha_{i,w^{(i)}(h(i)-\m_i+1)} \right)
\end{equation*}
up to a non-zero scalar multiplication and they are linearly independent in $H^*(\Hess(N,\tilde h))$ by the inductive assumption.
By the injectivity of the Gysin map $\iota_{!}$ (Lemma~\ref{lemma:Gysinmap}) we conclude that the cohomology classes in \eqref{eq:basisApermutation} are linearly independent in $H^*(\Hess(N,h))$. 
This completes the proof. 
\end{proof}

\bigskip

For the rest of this section, we discuss the cases of types $B,C,G$.
A \textbf{Hessenberg function for type $B_n$ and $C_n$} is defined to be a function  $h: [n] \to [2n]$ satisfying the following three conditions
\begin{enumerate}
\item $i \leq h(i) \leq 2n+1-i$ for $i \in [n]$, 
\item if $h(i) \neq 2n+1-i$, then $h(i) \leq h(i+1)$ for $i=1,\ldots, n-1$, 
\item if $h(i) = 2n+1-i$, then $h(i+1) = 2n+1-(i+1)$ for $i=1,\ldots, n-1$.
\end{enumerate}
We also define a \textbf{Hessenberg function for type $G_2$} as a function $h:\{1,2,3\} \rightarrow \{1,2,\ldots,6\}$ such that
\begin{enumerate}
\item $1 \leq h(1) \leq 6$, $2 \leq h(2) \leq 3$, $h(3)=3$, 
\item if $h(1) \geq 3$, then $h(2)=3$. 
\end{enumerate}
We define a subset $\Phi^+_i=\{\alpha_{i,j} \mid i < j \leq 2n+1-i \}$ of positive roots as follows: 
\begin{align*} 
&{\rm Type} \ B_n \ \ \ 
\alpha_{i,j}=
\begin{cases}
x_i-x_j  \ \ \ \ \ &{\rm if} \ i+1 \leq j \leq n, \\ 
x_i  \ \ \ \ \ &{\rm if} \ j=n+1, \\ 
x_i+x_{2n+2-j}  \ \ \ \ \ &{\rm if} \ n+2 \leq j \leq 2n+1-i. 
\end{cases}\\ 
&{\rm Type} \ C_n \ \ \ 
\alpha_{i,j}=
\begin{cases}
x_i-x_j  \ \ \ \ \ &{\rm if} \ i+1 \leq j \leq n, \\ 
x_i+x_{2n+1-j}  \ \ \ \ \ &{\rm if} \ n+1 \leq j \leq 2n-i, \\
2x_i  \ \ \ \ \ &{\rm if} \ j=2n+1-i. 
\end{cases}\\ 
&{\rm Type} \ G_2 \ \ \ 
\alpha_{1,2}=x_1-x_2, \ \ \ \alpha_{1,3}=-x_1+x_3, \ \ \ \alpha_{1,4}=-x_2+x_3, \\ 
& \ \ \ \ \ \ \ \ \ \ \ \ \ \, \, \, \alpha_{1,5}=x_1-2x_2+x_3, \ \ \ \alpha_{1,6}=-x_1-x_2+2x_3, \\ 
& \ \ \ \ \ \ \ \ \ \ \ \ \ \, \, \, \alpha_{2,3}=-2x_1+x_2+x_3.
\end{align*}
Similarly to type $A$, we can pictorially describe a Hessenberg function $h$ by drawing a configuration of shaded boxes such that the lower ideal associated with $h$ consists of positive roots in the shaded boxes (see \cite[Sections~10.3 and 10.4]{AHMMS} or \cite[Sections~5.2 and 5.3]{EHNT1}).
By \cite[Corollaries~10.10, 10.15, and 10.18]{AHMMS} we obtain an explicit presentation of the cohomology ring of $\Hess(N,h)$ in types $B_n$, $C_n$, and $G_2$.
A similar argument using the explicit presentation for $H^*(\Hess(N,h))$ yields the proof of Theorem~\ref{theorem_intro} for types $B,C,G$, as described below.
Recall that the argument in type $A$ is based on the following two properties.
\begin{enumerate}
\item[(I)]  The Gysin map $\iota_{!}: H^*(\Hess(N,h)) \to H^*(\Hess(N,\tilde h))$ is injective and it is the multiplication map by certain positive root if $\Hess(N,h)$ is codimension one in $\Hess(N,\tilde h)$ (cf. Lemma~\ref{lemma:Gysinmap}).
\item[(II)] The generalized rings $R^{\mathbf{A}_{(n-1)}}$ of the cohomology rings of the flag varieties which we introduced have a good property (Lemma~\ref{lemma:setsudouringA}). 
\end{enumerate}
In fact, the property~(I) implies that the statement of Theorem~\ref{theorem:basisA} can be reduced to the case when $\Hess(N,h)$ is the flag variety. 
When $\Hess(N,h)$ is the flag variety, the statement of Theorem~\ref{theorem:basisA} can be proved by the property~(II) above (Proposition~\ref{proposition:basisAFlag}).

Note that the property~(I) holds for arbitrary Lie types. In type $B_n$ one can consider generalized rings $R^{\mathbf{B}_{(n)}}$ of $H^*(G/B)$ which have a good property as in Lemma~\ref{lemma:setsudouringA}. As a consequence, we can prove Theorem~\ref{theorem:Intro_ABCDE6FG} for type $B$. 
Since the cohomology rings of the flag varieties of types $B_n$ and $C_n$ are isomorphic, one can easily see that the statement of Theorem~\ref{theorem:Intro_ABCDE6FG} holds when $\Hess(N,h)$ is the type $C$ flag variety. By using the property~(I) above, we obtain the desired additive basis of $H^*(\Hess(N,h))$ in Theorem~\ref{theorem:Intro_ABCDE6FG} for type $C$.
Also, a similar manipulation yields Theorem~\ref{theorem:Intro_ABCDE6FG} for type $G_2$.

\bigskip

\section{Type $D$}
\label{section:typeD}

In this section we give an outline of the proof of Theorem~\ref{theorem_intro_typeD}. More details are explained in Appendix~\ref{appendix:A}.

We define a \textbf{Hessenberg function for type $D_n$} by a function $h: [n] \to [2n-1]$ satisfying the following conditions:
\begin{enumerate}
\item $i \leq h(i) \leq 2n-1-i$ for $i=1,\ldots, n-1$; 
\item $n \leq h(n) \leq 2n-1$; 
\item if $h(i) \neq 2n-1-i$, then $h(i) \leq h(i+1)$ for $i=1,\ldots, n-2$;
\item if $h(i) = 2n-1-i$, then $h(i+1) = 2n-1-(i+1)$ for $i=1,\ldots, n-2$;
\item if $h(i) \geq n+1$, then $h(n) \geq 2n-i$ for $i=1,\ldots, n-2$;\footnote{The condition $(5)$ is true for $i=n-1$ because $h(n-1) = n-1$ or $h(n-1)=n$.}
\item if $h(n) \geq 2n-i$, then $h(i) \geq n-1$ for $i=1,\ldots, n-2$.
\end{enumerate}
Similarly to type $A$, we can also pictorially describe a Hessenberg function $h$ by drawing a configuration of shaded boxes in this case (\cite[Section~5.4]{EHNT1}).
For this purpose we need to explain the coordinate in type $D_n$.
Drawing a configuration of boxes on a square grid of size $n \times (2n-1)$, to each box in the $i$-th row and the $j$-th column with $i \neq n$ and $j \neq n+1$ we assign the coordinate as follows:
\begin{align*}
&(i,j) \ \ \ \ \ \ \ \, \, {\rm if} \ 1 \leq i \leq n-1 \ {\rm and} \ i \leq j \leq n, \\
&(i,j-1) \ \ \ {\rm if} \ 1 \leq i \leq n-1 \ {\rm and} \ n+2 \leq  j \leq 2n-i.
\end{align*}
We also assign to each box in the $i$-th row and the $(n+1)$-th column the coordinate $(n,2n-i)$.
Then, we may think of a Hessenberg function $h$ as a configuration of shaded boxes whose coordinates $(i,j)$ satisfy $i \in [n]$ and $i \leq j \leq h(i)$.

\begin{example}
Let $n=4$. Then $h=(3,5,4,7)$ is a Hessenberg function for type $D_4$ and the configuration of shaded boxes is shown in Figure~\ref{picture:typeDHessenbergFunction}.

\begin{figure}[h]
\begin{center}
\begin{picture}(210,75)
\put(0,63){\colorbox{gray}}
\put(5,63){\colorbox{gray}}
\put(10,63){\colorbox{gray}}
\put(15,63){\colorbox{gray}}
\put(20,63){\colorbox{gray}}
\put(24,63){\colorbox{gray}}
\put(0,68){\colorbox{gray}}
\put(5,68){\colorbox{gray}}
\put(10,68){\colorbox{gray}}
\put(15,68){\colorbox{gray}}
\put(20,68){\colorbox{gray}}
\put(24,68){\colorbox{gray}}
\put(0,73){\colorbox{gray}}
\put(5,73){\colorbox{gray}}
\put(10,73){\colorbox{gray}}
\put(15,73){\colorbox{gray}}
\put(20,73){\colorbox{gray}}
\put(24,73){\colorbox{gray}}
\put(0,77){\colorbox{gray}}
\put(5,77){\colorbox{gray}}
\put(10,77){\colorbox{gray}}
\put(15,77){\colorbox{gray}}
\put(20,77){\colorbox{gray}}
\put(24,77){\colorbox{gray}}

\put(30,63){\colorbox{gray}}
\put(35,63){\colorbox{gray}}
\put(40,63){\colorbox{gray}}
\put(45,63){\colorbox{gray}}
\put(50,63){\colorbox{gray}}
\put(54,63){\colorbox{gray}}
\put(30,68){\colorbox{gray}}
\put(35,68){\colorbox{gray}}
\put(40,68){\colorbox{gray}}
\put(45,68){\colorbox{gray}}
\put(50,68){\colorbox{gray}}
\put(54,68){\colorbox{gray}}
\put(30,73){\colorbox{gray}}
\put(35,73){\colorbox{gray}}
\put(40,73){\colorbox{gray}}
\put(45,73){\colorbox{gray}}
\put(50,73){\colorbox{gray}}
\put(54,73){\colorbox{gray}}
\put(30,77){\colorbox{gray}}
\put(35,77){\colorbox{gray}}
\put(40,77){\colorbox{gray}}
\put(45,77){\colorbox{gray}}
\put(50,77){\colorbox{gray}}
\put(54,77){\colorbox{gray}}

\put(60,63){\colorbox{gray}}
\put(65,63){\colorbox{gray}}
\put(70,63){\colorbox{gray}}
\put(75,63){\colorbox{gray}}
\put(80,63){\colorbox{gray}}
\put(84,63){\colorbox{gray}}
\put(60,68){\colorbox{gray}}
\put(65,68){\colorbox{gray}}
\put(70,68){\colorbox{gray}}
\put(75,68){\colorbox{gray}}
\put(80,68){\colorbox{gray}}
\put(84,68){\colorbox{gray}}
\put(60,73){\colorbox{gray}}
\put(65,73){\colorbox{gray}}
\put(70,73){\colorbox{gray}}
\put(75,73){\colorbox{gray}}
\put(80,73){\colorbox{gray}}
\put(84,73){\colorbox{gray}}
\put(60,77){\colorbox{gray}}
\put(65,77){\colorbox{gray}}
\put(70,77){\colorbox{gray}}
\put(75,77){\colorbox{gray}}
\put(80,77){\colorbox{gray}}
\put(84,77){\colorbox{gray}}

\put(120,63){\colorbox{gray}}
\put(125,63){\colorbox{gray}}
\put(130,63){\colorbox{gray}}
\put(135,63){\colorbox{gray}}
\put(140,63){\colorbox{gray}}
\put(144,63){\colorbox{gray}}
\put(120,68){\colorbox{gray}}
\put(125,68){\colorbox{gray}}
\put(130,68){\colorbox{gray}}
\put(135,68){\colorbox{gray}}
\put(140,68){\colorbox{gray}}
\put(144,68){\colorbox{gray}}
\put(120,73){\colorbox{gray}}
\put(125,73){\colorbox{gray}}
\put(130,73){\colorbox{gray}}
\put(135,73){\colorbox{gray}}
\put(140,73){\colorbox{gray}}
\put(144,73){\colorbox{gray}}
\put(120,77){\colorbox{gray}}
\put(125,77){\colorbox{gray}}
\put(130,77){\colorbox{gray}}
\put(135,77){\colorbox{gray}}
\put(140,77){\colorbox{gray}}
\put(144,77){\colorbox{gray}}


\put(30,43){\colorbox{gray}}
\put(35,43){\colorbox{gray}}
\put(40,43){\colorbox{gray}}
\put(45,43){\colorbox{gray}}
\put(50,43){\colorbox{gray}}
\put(54,43){\colorbox{gray}}
\put(30,48){\colorbox{gray}}
\put(35,48){\colorbox{gray}}
\put(40,48){\colorbox{gray}}
\put(45,48){\colorbox{gray}}
\put(50,48){\colorbox{gray}}
\put(54,48){\colorbox{gray}}
\put(30,53){\colorbox{gray}}
\put(35,53){\colorbox{gray}}
\put(40,53){\colorbox{gray}}
\put(45,53){\colorbox{gray}}
\put(50,53){\colorbox{gray}}
\put(54,53){\colorbox{gray}}
\put(30,57){\colorbox{gray}}
\put(35,57){\colorbox{gray}}
\put(40,57){\colorbox{gray}}
\put(45,57){\colorbox{gray}}
\put(50,57){\colorbox{gray}}
\put(54,57){\colorbox{gray}}

\put(60,43){\colorbox{gray}}
\put(65,43){\colorbox{gray}}
\put(70,43){\colorbox{gray}}
\put(75,43){\colorbox{gray}}
\put(80,43){\colorbox{gray}}
\put(84,43){\colorbox{gray}}
\put(60,48){\colorbox{gray}}
\put(65,48){\colorbox{gray}}
\put(70,48){\colorbox{gray}}
\put(75,48){\colorbox{gray}}
\put(80,48){\colorbox{gray}}
\put(84,48){\colorbox{gray}}
\put(60,53){\colorbox{gray}}
\put(65,53){\colorbox{gray}}
\put(70,53){\colorbox{gray}}
\put(75,53){\colorbox{gray}}
\put(80,53){\colorbox{gray}}
\put(84,53){\colorbox{gray}}
\put(60,57){\colorbox{gray}}
\put(65,57){\colorbox{gray}}
\put(70,57){\colorbox{gray}}
\put(75,57){\colorbox{gray}}
\put(80,57){\colorbox{gray}}
\put(84,57){\colorbox{gray}}

\put(90,43){\colorbox{gray}}
\put(95,43){\colorbox{gray}}
\put(100,43){\colorbox{gray}}
\put(105,43){\colorbox{gray}}
\put(110,43){\colorbox{gray}}
\put(114,43){\colorbox{gray}}
\put(90,48){\colorbox{gray}}
\put(95,48){\colorbox{gray}}
\put(100,48){\colorbox{gray}}
\put(105,48){\colorbox{gray}}
\put(110,48){\colorbox{gray}}
\put(114,48){\colorbox{gray}}
\put(90,53){\colorbox{gray}}
\put(95,53){\colorbox{gray}}
\put(100,53){\colorbox{gray}}
\put(105,53){\colorbox{gray}}
\put(110,53){\colorbox{gray}}
\put(114,53){\colorbox{gray}}
\put(90,57){\colorbox{gray}}
\put(95,57){\colorbox{gray}}
\put(100,57){\colorbox{gray}}
\put(105,57){\colorbox{gray}}
\put(110,57){\colorbox{gray}}
\put(114,57){\colorbox{gray}}

\put(120,43){\colorbox{gray}}
\put(125,43){\colorbox{gray}}
\put(130,43){\colorbox{gray}}
\put(135,43){\colorbox{gray}}
\put(140,43){\colorbox{gray}}
\put(144,43){\colorbox{gray}}
\put(120,48){\colorbox{gray}}
\put(125,48){\colorbox{gray}}
\put(130,48){\colorbox{gray}}
\put(135,48){\colorbox{gray}}
\put(140,48){\colorbox{gray}}
\put(144,48){\colorbox{gray}}
\put(120,53){\colorbox{gray}}
\put(125,53){\colorbox{gray}}
\put(130,53){\colorbox{gray}}
\put(135,53){\colorbox{gray}}
\put(140,53){\colorbox{gray}}
\put(144,53){\colorbox{gray}}
\put(120,57){\colorbox{gray}}
\put(125,57){\colorbox{gray}}
\put(130,57){\colorbox{gray}}
\put(135,57){\colorbox{gray}}
\put(140,57){\colorbox{gray}}
\put(144,57){\colorbox{gray}}

\put(150,43){\colorbox{gray}}
\put(155,43){\colorbox{gray}}
\put(160,43){\colorbox{gray}}
\put(165,43){\colorbox{gray}}
\put(170,43){\colorbox{gray}}
\put(174,43){\colorbox{gray}}
\put(150,48){\colorbox{gray}}
\put(155,48){\colorbox{gray}}
\put(160,48){\colorbox{gray}}
\put(165,48){\colorbox{gray}}
\put(170,48){\colorbox{gray}}
\put(174,48){\colorbox{gray}}
\put(150,53){\colorbox{gray}}
\put(155,53){\colorbox{gray}}
\put(160,53){\colorbox{gray}}
\put(165,53){\colorbox{gray}}
\put(170,53){\colorbox{gray}}
\put(174,53){\colorbox{gray}}
\put(150,57){\colorbox{gray}}
\put(155,57){\colorbox{gray}}
\put(160,57){\colorbox{gray}}
\put(165,57){\colorbox{gray}}
\put(170,57){\colorbox{gray}}
\put(174,57){\colorbox{gray}}

%

\put(60,23){\colorbox{gray}}
\put(65,23){\colorbox{gray}}
\put(70,23){\colorbox{gray}}
\put(75,23){\colorbox{gray}}
\put(80,23){\colorbox{gray}}
\put(84,23){\colorbox{gray}}
\put(60,28){\colorbox{gray}}
\put(65,28){\colorbox{gray}}
\put(70,28){\colorbox{gray}}
\put(75,28){\colorbox{gray}}
\put(80,28){\colorbox{gray}}
\put(84,28){\colorbox{gray}}
\put(60,33){\colorbox{gray}}
\put(65,33){\colorbox{gray}}
\put(70,33){\colorbox{gray}}
\put(75,33){\colorbox{gray}}
\put(80,33){\colorbox{gray}}
\put(84,33){\colorbox{gray}}
\put(60,37){\colorbox{gray}}
\put(65,37){\colorbox{gray}}
\put(70,37){\colorbox{gray}}
\put(75,37){\colorbox{gray}}
\put(80,37){\colorbox{gray}}
\put(84,37){\colorbox{gray}}

\put(90,23){\colorbox{gray}}
\put(95,23){\colorbox{gray}}
\put(100,23){\colorbox{gray}}
\put(105,23){\colorbox{gray}}
\put(110,23){\colorbox{gray}}
\put(114,23){\colorbox{gray}}
\put(90,28){\colorbox{gray}}
\put(95,28){\colorbox{gray}}
\put(100,28){\colorbox{gray}}
\put(105,28){\colorbox{gray}}
\put(110,28){\colorbox{gray}}
\put(114,28){\colorbox{gray}}
\put(90,33){\colorbox{gray}}
\put(95,33){\colorbox{gray}}
\put(100,33){\colorbox{gray}}
\put(105,33){\colorbox{gray}}
\put(110,33){\colorbox{gray}}
\put(114,33){\colorbox{gray}}
\put(90,37){\colorbox{gray}}
\put(95,37){\colorbox{gray}}
\put(100,37){\colorbox{gray}}
\put(105,37){\colorbox{gray}}
\put(110,37){\colorbox{gray}}
\put(114,37){\colorbox{gray}}

\put(120,23){\colorbox{gray}}
\put(125,23){\colorbox{gray}}
\put(130,23){\colorbox{gray}}
\put(135,23){\colorbox{gray}}
\put(140,23){\colorbox{gray}}
\put(144,23){\colorbox{gray}}
\put(120,28){\colorbox{gray}}
\put(125,28){\colorbox{gray}}
\put(130,28){\colorbox{gray}}
\put(135,28){\colorbox{gray}}
\put(140,28){\colorbox{gray}}
\put(144,28){\colorbox{gray}}
\put(120,33){\colorbox{gray}}
\put(125,33){\colorbox{gray}}
\put(130,33){\colorbox{gray}}
\put(135,33){\colorbox{gray}}
\put(140,33){\colorbox{gray}}
\put(144,33){\colorbox{gray}}
\put(120,37){\colorbox{gray}}
\put(125,37){\colorbox{gray}}
\put(130,37){\colorbox{gray}}
\put(135,37){\colorbox{gray}}
\put(140,37){\colorbox{gray}}
\put(144,37){\colorbox{gray}}

\put(120,3){\colorbox{gray}}
\put(125,3){\colorbox{gray}}
\put(130,3){\colorbox{gray}}
\put(135,3){\colorbox{gray}}
\put(140,3){\colorbox{gray}}
\put(144,3){\colorbox{gray}}
\put(120,8){\colorbox{gray}}
\put(125,8){\colorbox{gray}}
\put(130,8){\colorbox{gray}}
\put(135,8){\colorbox{gray}}
\put(140,8){\colorbox{gray}}
\put(144,8){\colorbox{gray}}
\put(120,13){\colorbox{gray}}
\put(125,13){\colorbox{gray}}
\put(130,13){\colorbox{gray}}
\put(135,13){\colorbox{gray}}
\put(140,13){\colorbox{gray}}
\put(144,13){\colorbox{gray}}
\put(120,17){\colorbox{gray}}
\put(125,17){\colorbox{gray}}
\put(130,17){\colorbox{gray}}
\put(135,17){\colorbox{gray}}
\put(140,17){\colorbox{gray}}
\put(144,17){\colorbox{gray}}

\put(0,00){\framebox(30,20)}
\put(0,20){\framebox(30,20)} 
\put(0,40){\framebox(30,20)}
\put(0,60){\framebox(30,20){\tiny $(1,1)$}}

\put(30,00){\framebox(30,20)}
\put(30,20){\framebox(30,20)}
\put(30,40){\framebox(30,20){\tiny $(2,2)$}}
\put(30,60){\framebox(30,20){\tiny $(1,2)$}} 

\put(60,00){\framebox(30,20)} 
\put(60,20){\framebox(30,20){\tiny $(3,3)$}} 
\put(60,40){\framebox(30,20){\tiny $(2,3)$}} 
\put(60,60){\framebox(30,20){\tiny $(1,3)$}}

\put(90,00){\framebox(30,20)} 
\put(90,20){\framebox(30,20){\tiny $(3,4)$}} 
\put(90,40){\framebox(30,20){\tiny $(2,4)$}} 
\put(90,60){\framebox(30,20){\tiny $(1,4)$}}

\put(120,00){\framebox(30,20){\tiny $(4,4)$}} 
\put(120,20){\framebox(30,20){\tiny $(4,5)$}} 
\put(120,40){\framebox(30,20){\tiny $(4,6)$}} 
\put(120,60){\framebox(30,20){\tiny $(4,7)$}}

\put(150,00){\framebox(30,20)} 
\put(150,20){\framebox(30,20)} 
\put(150,40){\framebox(30,20){\tiny $(2,5)$}} 
\put(150,60){\framebox(30,20){\tiny $(1,5)$}}

\put(180,0){\framebox(30,20)} 
\put(180,20){\framebox(30,20)} 
\put(180,40){\framebox(30,20)} 
\put(180,60){\framebox(30,20){\tiny $(1,6)$}}
\end{picture}
\end{center}
\caption{The configuration corresponding to $h=(3,5,4,7)$.}
\label{picture:typeDHessenbergFunction}
\end{figure} 
\end{example}
We denote positive roots by
\begin{equation*} 
\alpha_{i,j}=
\begin{cases}
x_i-x_j  \ \ \ \ \ &{\rm if} \ 1 \leq i \leq n-1 \ {\rm and} \ i+1 \leq j \leq n \\ 
x_i+x_{2n-j}  \ \ \ \ \ &{\rm if} \ 1 \leq i \leq n-1 \ {\rm and} \ n+1 \leq j \leq 2n-i-1 
\end{cases}
\end{equation*}
and 
\begin{equation*} 
\alpha_{n,j}=x_{2n-j}+x_n  \ \ \ \ \ {\rm if} \ n+1 \leq j \leq 2n-1. 
\end{equation*}
For $i \in [n-1]$, define a subset $\Phi^+_i=\{\alpha_{i,j} \mid i < j \leq 2n-i-1 \}$ and $\Phi^+_n=\{\alpha_{n,j} \mid n < j \leq 2n-1 \}$ of positive roots. 
Then, the set of lower ideals $I \subset \Phi^+_{D_n}=\coprod_{i=1}^{n} \ \Phi^+_i$ is in one-to-one correspondence with the set of Hessenberg functions $h$ for type $D_n$. 

\begin{example}
The lower ideal associated with a Hessenberg function $h=(3,5,4,7)$ can be pictorially shown in Figure~\ref{picture:TypeDLowerIdeal}.
\begin{figure}[h]
\begin{center}
\begin{picture}(150,40)
\put(0,33){\colorbox{gray}}
\put(0,37){\colorbox{gray}}
\put(5,33){\colorbox{gray}}
\put(5,37){\colorbox{gray}}
\put(10,33){\colorbox{gray}}
\put(10,37){\colorbox{gray}}
\put(15,33){\colorbox{gray}}
\put(15,37){\colorbox{gray}}
\put(20,33){\colorbox{gray}}
\put(20,37){\colorbox{gray}}
\put(24,33){\colorbox{gray}}
\put(24,37){\colorbox{gray}}

\put(30,33){\colorbox{gray}}
\put(30,37){\colorbox{gray}}
\put(35,33){\colorbox{gray}}
\put(35,37){\colorbox{gray}}
\put(40,33){\colorbox{gray}}
\put(40,37){\colorbox{gray}}
\put(45,33){\colorbox{gray}}
\put(45,37){\colorbox{gray}}
\put(50,33){\colorbox{gray}}
\put(50,37){\colorbox{gray}}
\put(54,33){\colorbox{gray}}
\put(54,37){\colorbox{gray}}

\put(60,33){\colorbox{gray}}
\put(60,37){\colorbox{gray}}
\put(65,33){\colorbox{gray}}
\put(65,37){\colorbox{gray}}
\put(70,33){\colorbox{gray}}
\put(70,37){\colorbox{gray}}
\put(75,33){\colorbox{gray}}
\put(75,37){\colorbox{gray}}
\put(80,33){\colorbox{gray}}
\put(80,37){\colorbox{gray}}
\put(84,33){\colorbox{gray}}
\put(84,37){\colorbox{gray}}


\put(120,33){\colorbox{gray}}
\put(120,37){\colorbox{gray}}
\put(125,33){\colorbox{gray}}
\put(125,37){\colorbox{gray}}
\put(130,33){\colorbox{gray}}
\put(130,37){\colorbox{gray}}
\put(135,33){\colorbox{gray}}
\put(135,37){\colorbox{gray}}
\put(140,33){\colorbox{gray}}
\put(140,37){\colorbox{gray}}
\put(144,33){\colorbox{gray}}
\put(144,37){\colorbox{gray}}


\put(30,23){\colorbox{gray}}
\put(30,27){\colorbox{gray}}
\put(35,23){\colorbox{gray}}
\put(35,27){\colorbox{gray}}
\put(40,23){\colorbox{gray}}
\put(40,27){\colorbox{gray}}
\put(45,23){\colorbox{gray}}
\put(45,27){\colorbox{gray}}
\put(50,23){\colorbox{gray}}
\put(50,27){\colorbox{gray}}
\put(54,23){\colorbox{gray}}
\put(54,27){\colorbox{gray}}

\put(60,23){\colorbox{gray}}
\put(60,27){\colorbox{gray}}
\put(65,23){\colorbox{gray}}
\put(65,27){\colorbox{gray}}
\put(70,23){\colorbox{gray}}
\put(70,27){\colorbox{gray}}
\put(75,23){\colorbox{gray}}
\put(75,27){\colorbox{gray}}
\put(80,23){\colorbox{gray}}
\put(80,27){\colorbox{gray}}
\put(84,23){\colorbox{gray}}
\put(84,27){\colorbox{gray}}

\put(90,23){\colorbox{gray}}
\put(90,27){\colorbox{gray}}
\put(95,23){\colorbox{gray}}
\put(95,27){\colorbox{gray}}
\put(100,23){\colorbox{gray}}
\put(100,27){\colorbox{gray}}
\put(105,23){\colorbox{gray}}
\put(105,27){\colorbox{gray}}
\put(110,23){\colorbox{gray}}
\put(110,27){\colorbox{gray}}
\put(114,23){\colorbox{gray}}
\put(114,27){\colorbox{gray}}

\put(120,23){\colorbox{gray}}
\put(120,27){\colorbox{gray}}
\put(125,23){\colorbox{gray}}
\put(125,27){\colorbox{gray}}
\put(130,23){\colorbox{gray}}
\put(130,27){\colorbox{gray}}
\put(135,23){\colorbox{gray}}
\put(135,27){\colorbox{gray}}
\put(140,23){\colorbox{gray}}
\put(140,27){\colorbox{gray}}
\put(144,23){\colorbox{gray}}
\put(144,27){\colorbox{gray}}

\put(150,23){\colorbox{gray}}
\put(150,27){\colorbox{gray}}
\put(155,23){\colorbox{gray}}
\put(155,27){\colorbox{gray}}
\put(160,23){\colorbox{gray}}
\put(160,27){\colorbox{gray}}
\put(165,23){\colorbox{gray}}
\put(165,27){\colorbox{gray}}
\put(170,23){\colorbox{gray}}
\put(170,27){\colorbox{gray}}
\put(174,23){\colorbox{gray}}
\put(174,27){\colorbox{gray}}

%

\put(60,13){\colorbox{gray}}
\put(60,17){\colorbox{gray}}
\put(65,13){\colorbox{gray}}
\put(65,17){\colorbox{gray}}
\put(70,13){\colorbox{gray}}
\put(70,17){\colorbox{gray}}
\put(75,13){\colorbox{gray}}
\put(75,17){\colorbox{gray}}
\put(80,13){\colorbox{gray}}
\put(80,17){\colorbox{gray}}
\put(84,13){\colorbox{gray}}
\put(84,17){\colorbox{gray}}

\put(90,13){\colorbox{gray}}
\put(90,17){\colorbox{gray}}
\put(95,13){\colorbox{gray}}
\put(95,17){\colorbox{gray}}
\put(100,13){\colorbox{gray}}
\put(100,17){\colorbox{gray}}
\put(105,13){\colorbox{gray}}
\put(105,17){\colorbox{gray}}
\put(110,13){\colorbox{gray}}
\put(110,17){\colorbox{gray}}
\put(114,13){\colorbox{gray}}
\put(114,17){\colorbox{gray}}

\put(120,13){\colorbox{gray}}
\put(120,17){\colorbox{gray}}
\put(125,13){\colorbox{gray}}
\put(125,17){\colorbox{gray}}
\put(130,13){\colorbox{gray}}
\put(130,17){\colorbox{gray}}
\put(135,13){\colorbox{gray}}
\put(135,17){\colorbox{gray}}
\put(140,13){\colorbox{gray}}
\put(140,17){\colorbox{gray}}
\put(144,13){\colorbox{gray}}
\put(144,17){\colorbox{gray}}

\put(120,3){\colorbox{gray}}
\put(120,7){\colorbox{gray}}
\put(125,3){\colorbox{gray}}
\put(125,7){\colorbox{gray}}
\put(130,3){\colorbox{gray}}
\put(130,7){\colorbox{gray}}
\put(135,3){\colorbox{gray}}
\put(135,7){\colorbox{gray}}
\put(140,3){\colorbox{gray}}
\put(140,7){\colorbox{gray}}
\put(144,3){\colorbox{gray}}
\put(144,7){\colorbox{gray}}

\put(0,30){\framebox(30,10)} 
\put(30,30){\framebox(30,10){\tiny $x_1-x_2$}}
\put(60,30){\framebox(30,10){\tiny $x_1-x_3$}}
\put(90,30){\framebox(30,10){\tiny $x_1-x_4$}}
\put(120,30){\framebox(30,10){\tiny $x_1+x_4$}}
\put(150,30){\framebox(30,10){\tiny $x_1+x_3$}}
\put(180,30){\framebox(30,10){\tiny $x_1+x_2$}}

\put(0,20){\framebox(30,10)} 
\put(30,20){\framebox(30,10)}
\put(60,20){\framebox(30,10){\tiny $x_2-x_3$}}
\put(90,20){\framebox(30,10){\tiny $x_2-x_4$}}
\put(120,20){\framebox(30,10){\tiny $x_2+x_4$}}
\put(150,20){\framebox(30,10){\tiny $x_2+x_3$}}
\put(180,20){\framebox(30,10)}

\put(0,10){\framebox(30,10)} 
\put(30,10){\framebox(30,10)}
\put(60,10){\framebox(30,10)}
\put(90,10){\framebox(30,10){\tiny $x_3-x_4$}}
\put(120,10){\framebox(30,10){\tiny $x_3+x_4$}}
\put(150,10){\framebox(30,10)}
\put(180,10){\framebox(30,10)}

\put(0,0){\framebox(30,10)} 
\put(30,0){\framebox(30,10)}
\put(60,0){\framebox(30,10)}
\put(90,0){\framebox(30,10)}
\put(120,0){\framebox(30,10)}
\put(150,0){\framebox(30,10)}
\put(180,0){\framebox(30,10)}
\end{picture}
\end{center}
\vspace{-10pt}
\caption{The lower ideal associated with $h=(3,5,4,7)$.}
\label{picture:TypeDLowerIdeal}
\end{figure} 
\end{example}

We now explain the polynomials $f^{D_n}_{i,j}$ given in \cite{EHNT1} which are used to describe the cohomology ring $H^*(\Hess(N,h))$.
For $1 \leq i \leq n-1$ the polynomial $f^{D_n}_{i,j}$ is defined by 
\begin{align}
f^{D_n}_{i,j}=&\sum_{k=1}^{i} (x_k-x_{i+1})\cdots(x_k-x_j) x_k \ \ \ \ \ \ {\rm if} \ i \leq j \leq n-2, \label{eq:fD1} \\
f^{D_n}_{i,n-1}=&\sum_{k=1}^{i} \left((x_k-x_{i+1})\cdots(x_k-x_{n-1})(x_k+x_n) \right) +(-1)^{n-i} n \, x_{i+1} \cdots x_n, \label{eq:fD2} \\
f^{D_n}_{i,n+j}=&\sum_{k=1}^{i} \big((x_k-x_{i+1})\cdots(x_k-x_{n})(x_k+x_n)\cdots(x_k+x_{n-j}) \big) \label{eq:fD3} \\
&+(-1)^{n-i+1} n \, x_{i+1}\cdots x_{n-1-j} x_{n-j}^2\cdots x_n^2  \ \ \ \ \ \ {\rm if} \ 0 \leq j \leq n-1-i. \notag 
\end{align}
We also define the polynomial $f^{D_n}_{n,2n-1-r} \ (0 \leq r \leq n-1)$ as
\begin{align}
f^{D_n}_{n,2n-1-r}=&\sum_{k=1}^{r} \left((-1)^{n-r+1} (x_k-x_{r+1})\cdots(x_k-x_{n-1})(x_k-x_n) \right)  + n \, x_{r+1}\cdots x_n. \label{eq:fD4}
\end{align}

\begin{theorem} $($\cite[Corollary~7.4]{EHNT1}$)$ \label{theorem:cohomologyD}
Let $h$ be a Hessenberg function for type $D_n$ and $\Hess(N,h)$ the associated regular nilpotent Hessenberg variety.
Then there is an isomorphism of graded $\R$-algebras
\begin{equation*}
H^*(\Hess(N,h)) \cong \R[x_1,\ldots,x_n]/(f^{D_n}_{1,h(1)}, \ldots f^{D_n}_{n,h(n)})
\end{equation*}
which sends each root $\alpha$ to the Euler class $e(L_{\alpha})|_{\Hess(N,h)}$.
\end{theorem}

Our basis of $H^*(\Hess(N,h))$ for types $A,B,C,G$ consists of monomials of positive roots (Theorem~\ref{theorem_intro}). 
As explained in the end of Section~\ref{section:Main theorem}, an idea to construct our basis is based on the properties (I) and (II) in Section~\ref{section:Main theorem}.
The property~(I) holds for type $D$, but the property~(II) is more complicated in type $D$.  
For this reason, we construct an additive basis for $H^*(\Hess(N,h))$ in type $D$ which differ from monomials of positive roots.
We first give an example below why the arguments of Section~\ref{section:Main theorem} can not be replicated. It may be helpful for the reader to understand why a basis in type $D$ is more complicated.

\begin{example} \label{example:D3}
Consider the case $n=3$ and take $h_{\max}=(4,3,5)$, which is the maximal Hessenberg function. Namely, $\Hess(N,h_{\max})$ is the flag variety in type $D_3$.
By Theorem~\ref{theorem:cohomologyD}, we can describe
\begin{equation*}
H^*(\Hess(N,h_{\max})) \cong \R[x_1,x_2,x_3]/(f_{1,4}, f_{2,3}, f_{3,5})
\end{equation*}
where 
\begin{equation} \label{eq:fij_flag_D3}
\begin{split} 
f_{1,4}&=f^{D_3}_{1,4}=(x_1-x_2)(x_1-x_3)(x_1+x_3)(x_1+x_2)-3x_2^2x_3^2, \\
f_{2,3}&=f^{D_3}_{2,3}=(x_1-x_3)(x_1+x_3)+(x_2-x_3)(x_2+x_3)+3x_3^2, \\ 
f_{3,5}&=f^{D_3}_{3,5}=3x_1x_2x_3.
\end{split}
\end{equation}
The positive roots in type $D_3$ are shown in Figure~\ref{picture:TypeD3Root}.
\begin{figure}[h]
\begin{center}
\begin{picture}(320,60)
\put(00,30){\framebox(80,30){$\alpha_{1,2}=x_1-x_2$}}
\put(80,30){\framebox(80,30){$\alpha_{1,3}=x_1-x_3$}}
\put(160,30){\framebox(80,30){$\alpha_{3,5}=x_1+x_3$}}
\put(240,30){\framebox(80,30){$\alpha_{1,4}=x_1+x_2$}}
\put(80,0){\framebox(80,30){$\alpha_{2,3}=x_2-x_3$}}
\put(160,0){\framebox(80,30){$\alpha_{3,4}=x_2+x_3$}}
\end{picture}
\end{center}
\vspace{-10pt}
\caption{The positive roots in type $D_3$.}
\label{picture:TypeD3Root}
\end{figure} 

We can naturally ask whether the following set  
\begin{equation*} 
\left\{\prod_{i=1}^{3} \alpha_{i,h_{\max}(i)} \cdot \alpha_{i,h_{\max}(i)-1} \cdots \alpha_{i,h_{\max}(i)-\m_i+1} \ \middle| \ 0 \leq \m_i \leq h_{\max}(i)-i \right\}
\end{equation*}
forms a basis for the cohomology $H^*(\Hess(N,h_{\max}))$ over $\R$, as in type $A, B, C, G$.
(For simplicity, we take each permutation $w^{(i)}$ as the identity in Theorem~\ref{theorem_intro}.)
As in the proof of Proposition~\ref{proposition:basisAFlag}, for $\mathbf{d}_1:=(d_{11}, p_1) \in \R^{2}$ and $\mathbf{d}_2=(d_{21}, d_{22}, p_2) \in \R^{3}$, we consider the polynomials 
\begin{align*}
f^{\mathbf{d}_1}_{1,4}&:=d_{11} (x_1-x_2)(x_1-x_3)(x_1+x_3)(x_1+x_2) - p_1 \, x_2^2x_3^2; \\
f^{\mathbf{d}_2}_{2,3}&:=d_{21}(x_1-x_3)(x_1+x_3)+d_{22}(x_2-x_3)(x_2+x_3) + p_2 \, x_{3}^2, 
\end{align*}
and define the generalized ring for $\mathbf{D}_{(3)}=(\mathbf{d}_i \in \R^{i+1} \mid i=1, 2)$ by
\begin{align*} 
R^{\mathbf{D}_{(3)}}:=\R[x_1,x_2,x_3]/(f^{\mathbf{d}_1}_{1,4}, f^{\mathbf{d}_2}_{2,3}, x_1x_2x_3). 
\end{align*}
Note that the generator $x_1x_2x_3$ of the ideal in the quotient ring above means $f_{3,5}$ up to a non-zero scalar multiplication.
A key point of the proof of Proposition~\ref{proposition:basisAFlag} (in type $A$) is that the multiplication by a root induces the map from $R^{\mathbf{A'}_{(n-2)}}$ to $R^{\mathbf{A}_{(n-1)}}$ (see \eqref{eq:multiinproof}). 
In other words, since the generator $f^{\mathbf{a}_1}_{1,n}$ of the ideal in $R^{\mathbf{A}_{(n-1)}}$ is of the form of a product, we can apply Lemma~\ref{lemma:key} so that the multiplication $g_n''$ is a root.
However, since $f^{\mathbf{d}_1}_{1,4}$ is not of the form of a product and any factor of $x_1x_2x_3$ is not a positive root in type $D$, it is not easy to apply Lemma~\ref{lemma:key} to $R^{\mathbf{D}_{(3)}}$.

We are able to overcome this obstacle by changing a candidate of a basis for $H^*(\Hess(N,h))$. 
Consider a sequence of Hessenberg functions
$$
h_1:=(1,3,4) \subset h_2:=(2,3,4) \subset h_3:=(3,3,4) \subset \tilde{h}_3:=(3,3,5) \subset \tilde{h}_4:=(4,3,5)
$$
where $\tilde{h}_4$ is the maximal Hessenberg function for type $D_3$ and we may regard $h_1$ as the maximal Hessenberg function for type $D_2$. 
Then, Lemma~\ref{lemma:Gysinmap} leads us to the following composition of the Gysin maps
\begin{align*} 
H^*(\Hess(N,h_1)) \hookrightarrow H^*(\Hess(N,h_2)) \hookrightarrow 
H^*(\Hess(N,h_3)) \hookrightarrow H^*(\Hess(N,\tilde{h}_3)) \hookrightarrow 
H^*(\Hess(N,\tilde{h}_4)). 
\end{align*}
To describe the cohomology rings above, we need \eqref{eq:fij_flag_D3} and the following polynomials
\begin{align*}
&f_{1,1}=x_1, \ f_{1,2}=(x_1-x_2)(x_1+x_3) + 3x_2x_3, \ f_{1,3}=(x_1-x_2)(x_1-x_3)(x_1+x_3)-3x_2x_3^2, \\ 
&f_{3,4}=-(x_1-x_2)(x_1-x_3)+3x_2x_3
\end{align*}
from Theorem~\ref{theorem:cohomologyD}.
As usual, for $\mathbf{d}_1:=(d_{11}, p_1) \in \R^{2}$, we consider 
\begin{align*}
&f^{\mathbf{d}_1}_{1,1}=d_{11}x_1, \ f^{\mathbf{d}_1}_{1,2}=d_{11}(x_1-x_2)(x_1+x_3) + p_1 \, x_2x_3, \\
&f^{\mathbf{d}_1}_{1,3}=d_{11}(x_1-x_2)(x_1-x_3)(x_1+x_3)-p_1 \, x_2x_3^2, \ f^{\mathbf{d}_1}_{3,4}=d_{11}(x_1-x_2)(x_1-x_3)-p_1 \, x_2x_3 
\end{align*}
and define 
\begin{align*} 
R^{\mathbf{D}_{(3)}}_{h_j}:=&\R[x_1,x_2,x_3]/(f^{\mathbf{d}_1}_{1,j}, f^{\mathbf{d}_i}_{2,3}, f^{\mathbf{d}_1}_{3,4}) \ \ \ {\rm for} \ j=1,2,3, \\
R^{\mathbf{D}_{(3)}}_{\tilde{h}_j}:=&\R[x_1,x_2,x_3]/(f^{\mathbf{d}_1}_{1,j}, f^{\mathbf{d}_i}_{2,3}, x_1x_2x_3) \ \ \ {\rm for} \ j=3,4
\end{align*}
for $\mathbf{D}_{(3)}=(\mathbf{d}_i \in \R^{i+1} \mid i =1,2)$ where we may assume that these are Artinian rings.
Note that $R^{\mathbf{D}_{(3)}}_{\tilde{h}_{4}}=R^{\mathbf{D}_{(3)}}$ and we may know inductively an additive basis for $R^{\mathbf{D}_{(3)}}_{h_1}$ since $h_1$ can be regarded as the maximal Hessenberg function for type $D_2$. 
Recall that our aim is to construct an additive basis of $R^{\mathbf{D}_{(3)}}_{\tilde{h}_{4}}$. 
Since $f^{\mathbf{d}_1}_{1,4} = (x_1+x_2) f^{\mathbf{d}_1}_{1,3} + p_1 \, x_3 \cdot (x_1 x_2 x_3)$, one can write 
\begin{align*}
R^{\mathbf{D}_{(3)}}_{\tilde{h}_4} = \R[x_1,x_2,x_3]/((x_1+x_2) f^{\mathbf{d}_1}_{1,3}, f^{\mathbf{d}_i}_{2,3}, x_1x_2x_3).
\end{align*}
It then follows from Lemma~\ref{lemma:key} that 
\begin{align*} 
0 \rightarrow R^{\mathbf{D}_{(3)}}_{\tilde{h}_{3}} \xrightarrow{\times (x_1+x_2)} R^{\mathbf{D}_{(3)}}_{\tilde{h}_4} \rightarrow R^{\mathbf{D}_{(3)}}_{\tilde{h}_4}/(x_1+x_2) \rightarrow 0
\end{align*}
is exact.
If we know an additive basis of the quotient ring $R^{\mathbf{D}_{(3)}}_{\tilde{h}_4}/(x_1+x_2)$, then our aim reduces to construct an additive basis for $R^{\mathbf{D}_{(3)}}_{\tilde{h}_{3}}$ by the exact sequence above. 
Next, since $f^{\mathbf{d}_1}_{1,3} = (x_1+x_3) f^{\mathbf{d}_1}_{3,4} + p_1 \, x_1 x_2 x_3$, we obtain the following exact sequence
\begin{align*} 
0 \rightarrow R^{\mathbf{D}_{(3)}}_{h_{3}} \xrightarrow{\times (x_1+x_3)} R^{\mathbf{D}_{(3)}}_{\tilde{h}_3} \rightarrow R^{\mathbf{D}_{(3)}}_{\tilde{h}_3}/(x_1+x_3) \rightarrow 0.
\end{align*}
Thus, we aim to know an additive basis for $R^{\mathbf{D}_{(3)}}_{h_{3}}$ if we have an additive basis of the quotient ring $R^{\mathbf{D}_{(3)}}_{\tilde{h}_3}/(x_1+x_3)$.
Finally, using the following formulas
\begin{align*}
2f^{\mathbf{d}_1}_{1,3} = (x_1-x_3) f^{\mathbf{d}_1}_{1,2} + (x_1+x_3)f^{\mathbf{d}_1}_{3,4} \ \textrm{and} \ f^{\mathbf{d}_1}_{1,2} = 2(x_1-x_2) f^{\mathbf{d}_1}_{1,1} - f^{\mathbf{d}_1}_{3,4}, 
\end{align*}
we obtain the exact sequence
\begin{equation*} 
0 \rightarrow R^{\mathbf{D}_{(3)}}_{h_{j-1}} \xrightarrow{\times (x_1-x_j)} R^{\mathbf{D}_{(3)}}_{h_j} \rightarrow R^{\mathbf{D}_{(3)}}_{h_j}/(x_1-x_j) \rightarrow 0
\end{equation*}
for $j =2,3$. 
Therefore, if we know an additive basis of the quotient rings $R^{\mathbf{D}_{(3)}}_{h_j}/(x_1-x_j)$, then our goal is to construct an additive basis for $R^{\mathbf{D}_{(3)}}_{h_1}$. 
However, recalling that $h_1$ can be naturally regarded as the maximal Hessenberg function for type $D_2$, an additive basis for $R^{\mathbf{D}_{(3)}}_{h_1}$ may be inductively known. 
Thus, we can expect an inductive construction of an additive basis for $R^{\mathbf{D}_{(3)}}$ although it remains to construct an additive basis for the quotient rings $R^{\mathbf{D}_{(3)}}_{\tilde{h}_4}/(x_1+x_2), R^{\mathbf{D}_{(3)}}_{\tilde{h}_3}/(x_1+x_3)$, and $R^{\mathbf{D}_{(3)}}_{h_j}/(x_1-x_j)$. 
Since these quotient rings have a good property as in Lemma~\ref{lemma:setsudouringA} but the property is complicated, a candidate of a basis for $H^*(\Hess(N,h))$ should be changed in type $D$. 
The basis which we will construct is obtained by trial and error, and a benefit of our basis is to give the linear independence of the set of Poincar\'e duals of smaller regular nilpotent Hessenberg varieties $\Hess(N,h')$, as described below.
\end{example}

In order to construct an additive basis of the cohomology $H^*(\Hess(N,h))$ for type $D_n$, we need some preliminaries. 
For $n \geq 2$, define a set 
\begin{align*}
L_n:=\{(\ell_1,\ldots,\ell_n) \in \Z^n \mid i \leq \ell_i \leq 2n-1-i \ {\rm for} \ i\in [n-1] \ {\rm and} \ n \leq \ell_n \leq 2n-1 \}.
\end{align*}
We construct $\ell' \in L_{n-1}$ from $\ell \in L_n$ as follows. \\
\textbf{Case (i):} Suppose that $\ell_n < 2n-1$. Then we define $\ell' \in L_{n-1}$ by
\begin{align}
& \ell':=(\ell_2-1,\ldots,\ell_{n-1}-1,\ell_n-1) \ \ \ {\rm if} \ \ell_1 \leq n; \label{eq:proc1r} \\
& \ell':=(\ell_2-1,\ldots,\ell_{n-1}-1,\ell_n-1) \ \ \ {\rm if} \ \ell_1 \geq n+1. \label{eq:proc3r}
\end{align}
Remark that the results of $\ell'$ in both cases are same, but it is convenient to distinguish the two cases for the purpose of definition for $\alpha_{i,j}^{(\ell)}$, as described below. \\ \ 
\textbf{Case (ii):} Suppose that $\ell_n = 2n-1$. Then
\begin{align}
& \ell':=(\ell_2-1,\ldots,\ell_{n-1}-1,\ell_1+n-2) \ \ \ {\rm if} \ \ell_1 \leq n-1; \label{eq:proc2r} \\
& \textrm{$\ell'$ is undefined} \ \ \ {\rm if} \ \ell_1 \geq n. \label{eq:proc2} 
\end{align}
If $\ell'$ is obtained from $\ell$ by a procedure $(6.*)$ ($*=5,6,7$), then we write $\ell \xrightarrow{(6.*)} \ell'$ or simply $\ell \rightarrow\ell'$.

Let $h$ be a Hessenberg function for type $D_n$.
For $m=(m_1,\ldots,m_n)$ with $0 \leq m_i \leq h(i)-i$ for all $i \in [n]$, we define 
\begin{equation} \label{eq:vmhD}
v_m^{(h)}:=\prod_{i=1}^{n} \alpha_{i,h(i)}^{(h-m)} \cdot \alpha_{i,h(i)-1}^{(h-m)} \cdots \alpha_{i,h(i)-m_i+1}^{(h-m)}
\end{equation}
with the convention $\alpha_{i,h(i)}^{(h-m)} \cdot \alpha_{i,h(i)-1}^{(h-m)} \cdots \alpha_{i,h(i)-m_i+1}^{(h-m)}=1$ whenever $\m_i=0$.
Here, $\alpha_{i,j}^{(h-m)}$ is defined as follows.
We put $\ell:=h-m=(h(1)-m_1,\ldots,h(n)-m_n)$. Then we have $\ell \in L_n$ and consider the finitely many procedures $(6.*)$ ($*=5,6,7,8$) starting at $\ell$: 
$$
\ell=\ell^{(1)} \rightarrow \ell^{(2)} \rightarrow \cdots \rightarrow \ell^{(N)}.
$$
For $1 \leq i \leq n-1$ and $i+1 \leq j \leq 2n-1-i$ we put $k:=n-j+i$ and define 
\begin{align*}
\alpha^{(\ell)}_{i,j}:=\begin{cases}
x_k \ \ \ & {\rm if} \  i<k, \ \ell^{(i)} \xrightarrow{\eqref{eq:proc2r}} \ell^{(i+1)}, \ {\rm and} \ \ell^{(k)} \xrightarrow{\eqref{eq:proc3r}} \ell^{(k+1)}, \\
x_i-x_k \ \ \ & {\rm if} \ i<k, \ \ell^{(i)} \xrightarrow{\eqref{eq:proc2r}} \ell^{(i+1)}, \ {\rm and} \  \ell^{(k)} \xrightarrow{\eqref{eq:proc1r}} \ell^{(k+1)}, \\
\alpha_{i,j} \ \ \ &{\rm otherwise}.
\end{cases}
\end{align*}
For $1 \leq r \leq n-1$ we also define 
\begin{align*}
\alpha^{(\ell)}_{n,2n-r}:=\begin{cases}
x_r \ \ \ &{\rm if} \ \ell^{(r)} \xrightarrow{\eqref{eq:proc3r}} \ell^{(r+1)},  \\
\alpha_{n,2n-r}=x_r+x_n \ \ \ &{\rm otherwise}.
\end{cases}
\end{align*}

\begin{example}
Let us consider a Hessenberg function $h=(3,5,4,7)$ for type $D_4$. \\ \ 
(1) If we take $m=(m_1,\ldots,m_4)=(1,0,0,2)$, then we have $\ell:=h-m=(2,5,4,5)$.
The procedures starting at $\ell$ can be computed by 
$$
\ell=\ell^{(1)}=(2,5,4,5) \xrightarrow{\eqref{eq:proc1r}} \ell^{(2)}=(4,3,4) \xrightarrow{\eqref{eq:proc3r}} \ell^{(3)}=(2,3). 
$$
Hence, the product $v_m^{(h)}$ in \eqref{eq:vmhD} is equal to 
$$
v_m^{(h)}=(\alpha^{(\ell)}_{1,3})\cdot 1 \cdot 1 \cdot(\alpha^{(\ell)}_{4,7} \cdot \alpha^{(\ell)}_{4,6})=(x_1-x_3)(x_1+x_4)x_2.
$$
(2) If we take $m'=(m'_1,\ldots,m'_4)=(1,2,0,2)$, then we have $\ell':=h-m'=(2,3,4,5)$ and 
$$
\ell'=\ell'^{(1)}=(2,3,4,5) \xrightarrow{\eqref{eq:proc1r}} \ell'^{(2)}=(2,3,4) \xrightarrow{\eqref{eq:proc1r}} \ell'^{(3)}=(2,3).
$$
The product $v_{m'}^{(h)}$ in \eqref{eq:vmhD} can be expressed as
$$
v_{m'}^{(h)}=(\alpha^{(\ell')}_{1,3})\cdot(\alpha^{(\ell')}_{2,5} \cdot \alpha^{(\ell')}_{2,4})\cdot 1 \cdot(\alpha^{(\ell')}_{4,7} \cdot \alpha^{(\ell')}_{4,6})=(x_1-x_3)(x_2+x_3)(x_2-x_4)(x_1+x_4)(x_2+x_4).
$$
Note that this product is equal to the Poincar\'e dual $[\Hess(N,h')]$ for $h'=(2,3,4,5)$ by \eqref{eq:PoincaredualNilpotent_h}.
\end{example}

The following lemma is straightforward.

\begin{lemma} \label{lemma:5-18}
Let $\ell \in L_n$.
For $2 \leq i \leq n-1$ and $i+1 \leq j \leq 2n-1-i$, 
\begin{align*}
\alpha_{i,j}^{(\ell)}(x_1,\ldots,x_n)=\begin{cases}
\alpha_{i-1,j-1}^{(\ell')}(x_2,\ldots,x_n) \ \ \ {\rm if} \ \ell \rightarrow \ell' \ {\rm is \ defined}, \\
\alpha_{i-1,j-1}^{D_{n-1}}(x_2,\ldots,x_n) \ \ \ {\rm if} \ \ell' \ {\rm is \ undefined}.
\end{cases}
\end{align*}
For $2 \leq r \leq n-1$, 
\begin{align*}
\alpha_{n,2n-r}^{(\ell)}(x_1,\ldots,x_n)=\begin{cases}
\alpha_{n-1,2(n-1)-(r-1)}^{(\ell')}(x_2,\ldots,x_n) \ \ \ {\rm if} \ \ell \rightarrow \ell' \ {\rm is \ defined}, \\
\alpha_{n-1,2(n-1)-(r-1)}^{D_{n-1}}(x_2,\ldots,x_n) \ \ \ {\rm if} \ \ell' \ {\rm is \ undefined}.
\end{cases}
\end{align*}
Here, $\alpha_{i,j}^{D_{n-1}}(x_2,\ldots,x_n)$ denotes the positive root for type $D_{n-1}$ in the variables $x_2,\ldots,x_n$.
\end{lemma}

We restate Theorem~\ref{theorem_intro_typeD} as follows.

\begin{theorem} \label{theorem:basisD}
Let $h$ be a  Hessenberg function for type $D_n$ and $\Hess(N,h)$ the associated regular nilpotent Hessenberg variety.
Then, the cohomology classes $v_m^{(h)}$ with $0 \leq \m_i \leq h(i)-i$, form a basis for the cohomology $H^*(\Hess(N,h))$ over $\R$.
\end{theorem}

The following lemma tells us that it is enough to prove the special case of Theorem~\ref{theorem:basisD} when $\Hess(N,h)$ is the whole flag variety $G/B$. 
Note that the maximal Hessenberg function $h^{(n)}_{\max}$ for type $D_n$ is given by 
\begin{equation*} 
h^{(n)}_{\max}:=(2n-2,2n-3,\ldots,n,2n-1).
\end{equation*}

\begin{lemma} \label{lemma:D1}
If Theorem~$\ref{theorem:basisD}$ holds for $h=h^{(n)}_{\max}$, then Theorem~$\ref{theorem:basisD}$ is true for any Hessenberg function $h$.
\end{lemma}

\begin{proof}
By \eqref{eq:PoinHess} it is enough to prove that the cohomology classes $v_m^{(h)}$ in \eqref{eq:vmhD} are linearly independent in $H^*(\Hess(N,h))$. 
We proceed by decreasing induction on the dimension $d_h:=\sum_{i=1}^n (h(i)-i)$ of $\Hess(N,h)$. The base case $h=h^{(n)}_{\max}$ is nothing but the assumption of Lemma~\ref{lemma:D1}.
Now suppose that $d_h<d_{h^{(n)}_{\max}}$ and the claim holds for arbitrary Hessenberg function $\tilde h$ with $d_h < d_{\tilde h}$.
Since $h$ is not maximal, we can take a Hessenberg function $\tilde h$ such that $j:=\tilde h(i)=h(i)+1$ for some $i$ and $\tilde h(p)=h(p)$ for any $p \neq i$.
Then we show that 
\begin{equation} \label{eq:D1.1}
\alpha^{(\tilde h-m)}_{i,j}=\alpha_{i,j}
\end{equation}
for arbitrary $m=(m_1,\ldots,m_n)$ with $0 \leq m_k \leq h(k)-k$ for all $k \in [n]$.

\noindent
\textbf{Case (i):} Suppose that $1 \leq i \leq n-1$.
Let $k=n-j+i$. 
If $k \leq i$, then one has $\alpha^{(\tilde h-m)}_{i,j}=\alpha_{i,j}$ by the definition of $\alpha^{(\ell)}_{i,j}$.
We may assume that $k >i$, namely $n > j$.
Since $h(i)=j-1 < n-1$, one obtains $\tilde h(n)=h(n) < 2n-i$ by the definition $(6)$ of the Hessenberg function for type $D_n$.
Then, the last entry of $\ell:=\tilde h-m=(\tilde h(1)-m_1,\ldots,\tilde h(n)-m_n)$ satisfies the inequality
$$
\tilde h(n)-m_n < 2n-i-m_n.
$$
This means that each step of procedures $\ell=\ell^{(1)} \rightarrow \cdots \rightarrow \ell^{(i+1)}$ is either \eqref{eq:proc1r} or \eqref{eq:proc3r} even if $\ell^{(i+1)}$ is defined. 
Therefore, we obtain $\alpha^{(\tilde h-m)}_{i,j}=\alpha_{i,j}$ by the definition of $\alpha^{(\ell)}_{i,j}$.

\noindent
\textbf{Case (ii):} Suppose that $i = n$. 
We put $j=2n-r$ for some $r$ with $1 \leq r \leq n-1$.
Since $h(n)=\tilde h(n)-1=j-1=2n-r-1< 2n-r$, we have 
\begin{equation} \label{eq:D1.2}
\tilde h(r)=h(r) \leq n
\end{equation} 
by the definition $(5)$ of Hessenberg functions for type $D_n$.
The last entry of $\ell:=\tilde h-m=(\tilde h(1)-m_1,\ldots,\tilde h(n)-m_n)$ satisfies 
$$
\tilde h(n)-m_n=j-m_n=2n-r-m_n \leq 2n-r,
$$
which implies that procedures $\ell=\ell^{(1)} \rightarrow \cdots \rightarrow \ell^{(r)}$ are defined and each step is either \eqref{eq:proc1r} or \eqref{eq:proc3r}.
By \eqref{eq:D1.2} the last and the first entries of $\ell^{(r)}$ are given by
\begin{align*}
&\tilde h(n)-m_n-(r-1)=2n-r-m_n-(r-1) \leq 2(n-r+1)-1, \\
&\tilde h(r)-m_r-(r-1) \leq h'(r)-(r-1) \leq n-r+1,
\end{align*}
respectively.
This deduces that $\ell^{(r+1)}$ is defined and $\ell^{(r)} \rightarrow \ell^{(r+1)}$ is obtained from \eqref{eq:proc1r}, and hence $\alpha^{(\tilde h-m)}_{n,j}=\alpha_{n,j}$.

By \eqref{eq:D1.1} the image of $v_m^{(h)}$ under the Gysin map $\iota_{!}: H^*(\Hess(N,h)) \rightarrow H^*(\Hess(N,\tilde h))$ is equal to
\begin{equation} 
\alpha_{i,\tilde h(i)} \cdot v_m^{(h)} = \alpha^{(\tilde h-m)}_{i,\tilde h(i)} \cdot v_m^{(h)}=v_{\tilde m}^{(\tilde h)}
\end{equation}
up to a non-zero scalar multiplication where $\tilde m=(m_1,\ldots,m_{i-1},m_i+1,m_{i+1},\ldots,m_n)$.
By the inductive hypothesis on $d_h$ we know that these classes $v_{\tilde m}^{(\tilde h)}$ are linearly independent in $H^*(\Hess(N,\tilde h))$.
By the injectivity of $\iota_{!}$ (Lemma~\ref{lemma:Gysinmap}) the cohomology classes $v_{m}^{(h)}$ are linearly independent in $H^*(\Hess(N,h))$, as desired. 
\end{proof}

The proof of Theorem~\ref{theorem:basisD} for the case when $h=h^{(n)}_{\max}$ is more technical, so we here sketch the outline of the proof. For more details, see Appendix~\ref{appendix:A}.
As discussed in type $A$, we slightly generalize the definition of $f_{i,j}^{D_{n}}$.
For $i \in [n-1]$ and $\mathbf{d}_i=(d_{i1}, \ldots, d_{ii}, p_i) \in \R^{i+1}$ we define a polynomial 
\begin{equation*}
f^{\mathbf{d}_i}_{i,2n-1-i}:=\sum_{k=1}^i d_{ik} (x_k^2-x_{i+1}^2)\cdots(x_k^2-x_n^2) + (-1)^{n-i+1}p_i \, x_{i+1}^2 \cdots x_n^2. 
\end{equation*}
For $\mathbf{D}_{(n)}=(\mathbf{d}_i \in \R^{i+1} \mid 1 \leq i \leq n-1)$ we define a ring
\begin{align*} 
R^{\mathbf{D}_{(n)}}:=\R[x_1,\ldots,x_n]/(f^{\mathbf{d}_i}_{i,2n-1-i}, x_1 \cdots x_n \mid 1 \leq i \leq n-1). 
\end{align*}
Note that if $\mathbf{d}_i=(1, \ldots, 1, 1) \in \R^{i+1}$ for all $i \in [n-1]$, then $R^{\mathbf{D}_{(n)}}$ is isomorphic to $H^*(G/B)$ by Theorem~\ref{theorem:cohomologyD}.
We construct the additive basis of $R^{\mathbf{D}_{(n)}}$ by induction on $n$.
For this purpose, we consider the special Hessenberg functions 
\begin{align}
h_j:=&(j,h^{(n)}_{\max}(2),\ldots,h^{(n)}_{\max}(n-1),h^{(n)}_{\max}(n)-1) \ \ \ {\rm for} \ j=1,2,\ldots,n; \label{eq:hj_typeD} \\
\tilde{h}_j:=&(j,h^{(n)}_{\max}(2),\ldots,h^{(n)}_{\max}(n-1),h^{(n)}_{\max}(n)) \ \ \ \ \ \ \ \ {\rm for} \ j=n,n+1,\ldots,2n-2 \label{eq:hj_tilde_typeD} 
\end{align}
and the sequence 
$$
h_1 \subset h_2 \subset \cdots \subset h_n \subset \tilde{h}_n \subset \tilde{h}_{n+1} \subset \cdots \subset \tilde{h}_{2n-2}=h^{(n)}_{\max}.
$$
Note that the Hessenberg function $h_1$ can be regarded as the maximal Hessenberg function $h^{(n-1)}_{\max}$ for type $D_{n-1}$.
For $1 \leq j \leq 2n-2$ and $\mathbf{d}_1:=(d_{11}, p_1) \in \R^{2}$ we define polynomials 
\begin{equation*}
f^{\mathbf{d}_1}_{1,j}:=\begin{cases}
d_{11} (x_1-x_2)\cdots(x_1-x_j)x_1 \ \ \ \ \ \ {\rm if} \ 1 \leq j \leq n-2, \\ 
d_{11} (x_1-x_2)\cdots(x_1-x_{n-1})(x_1+x_n) + (-1)^{n-1}p_1 \ x_2 \cdots x_n \ \ \ \ \ {\rm if} \ j = n-1, \\ 
d_{11} (x_1-x_2)\cdots(x_1-x_{n-1-k})(x_1^2-x_{n-k}^2)\cdots(x_1^2-x_n^2) \\
+ (-1)^{n}p_1 \, x_2 \cdots x_{n-1-k} x_{n-k}^2 \cdots x_n^2 \ \ \ \ \ \ {\rm if} \ j=n+k \ {\rm with} \ 0 \leq k \leq n-2,
\end{cases} 
\end{equation*}
and
\begin{equation*}
f^{\mathbf{d}_1}_{n,2n-2}:=d_{11} (x_1-x_2)\cdots(x_1-x_n) + (-1)^{n}p_1 \ x_2 \cdots x_n.
\end{equation*}
Similarly, for $\mathbf{D}_{(n)}=(\mathbf{d}_i \in \R^{i+1} \mid 1 \leq i \leq n-1)$ and the Hessenberg function $h_j$ (or $\tilde h_j$) we define the following two rings
\begin{align*} 
R^{\mathbf{D}_{(n)}}_{h_j}:=&\R[x_1,\ldots,x_n]/(f^{\mathbf{d}_1}_{1,j}, f^{\mathbf{d}_i}_{i,2n-1-i}, f^{\mathbf{d}_1}_{n,2n-2} \mid 2 \leq i \leq n-1) \ \ \ \\
&{\rm for} \ j=1,\ldots,n, \\
R^{\mathbf{D}_{(n)}}_{\tilde{h}_j}:=&\R[x_1,\ldots,x_n]/(f^{\mathbf{d}_1}_{1,j}, f^{\mathbf{d}_i}_{i,2n-1-i}, x_1 \cdots x_n \mid 2 \leq i \leq n-1) \ \ \ \\
&{\rm for} \ j=n,\ldots,2n-2.
\end{align*}
In what follows, we assume that these rings are Artinian.
Note that $R^{\mathbf{D}_{(n)}}_{\tilde{h}_{2n-2}}=R^{\mathbf{D}_{(n)}}$.
We can also see that 
\begin{equation} \label{eq:D3-1}
R^{\mathbf{D}_{(n)}}_{h_1} \cong R^{\mathbf{D}'_{(n-1)}}
\end{equation}
for some $\mathbf{D}'_{(n-1)}=(\mathbf{d}'_i \in \R^{i+1} \mid 1 \leq i \leq n-2)$
where we regard the variables in $R^{\mathbf{D}'_{(n-1)}}$ as $x_2,\ldots,x_n$.

Lemma~\ref{lemma:key} leads us to the following exact sequence
\begin{equation} \label{eq:exactD3-1}
0 \rightarrow R^{\mathbf{D}_{(n)}}_{h_{j-1}} \xrightarrow{\times (x_1-x_j)} R^{\mathbf{D}_{(n)}}_{h_j} \rightarrow R^{\mathbf{D}_{(n)}}_{h_j}/(x_1-x_j) \rightarrow 0
\end{equation}
for $2 \leq j \leq n$ since it follows that
\begin{align*}
&f^{\mathbf{d}_1}_{1,j} = (x_1-x_j) f^{\mathbf{d}_1}_{1,j-1} \ \ \ {\rm for} \ 2 \leq j \leq n-2, \\
&f^{\mathbf{d}_1}_{1,n-1}+f^{\mathbf{d}_1}_{n,2n-2} = 2(x_1-x_{n-1}) f^{\mathbf{d}_1}_{1,n-2}, \\
&2f^{\mathbf{d}_1}_{1,n}-(x_1+x_n)f^{\mathbf{d}_1}_{n,2n-2} = (x_1-x_n) f^{\mathbf{d}_1}_{1,n-1}. 
\end{align*}
Similarly, we obtain the exact sequences
\begin{align} 
&0 \rightarrow R^{\mathbf{D}_{(n)}}_{h_{n}} \xrightarrow{\times (x_1+x_n)} R^{\mathbf{D}_{(n)}}_{\tilde{h}_n} \rightarrow R^{\mathbf{D}_{(n)}}_{\tilde{h}_n}/(x_1+x_n) \rightarrow 0, \label{eq:exactD3-2} \\
&0 \rightarrow R^{\mathbf{D}_{(n)}}_{\tilde{h}_{j-1}} \xrightarrow{\times (x_1+x_{2n-j})} R^{\mathbf{D}_{(n)}}_{\tilde{h}_j} \rightarrow R^{\mathbf{D}_{(n)}}_{\tilde{h}_j}/(x_1+x_{2n-j}) \rightarrow 0, \label{eq:exactD3-3} 
\end{align}
for $n+1 \leq j \leq 2n-2$ because $(-1)^n p_1 \, x_1 \cdots x_n + f^{\mathbf{d}_1}_{1,n} = (x_1+x_n) f^{\mathbf{d}_1}_{n,2n-2}$ and $f^{\mathbf{d}_1}_{1,n+k} + (-1)^n p_1 \, x_{n-k+1} \cdots x_n \cdot (x_1 \cdots x_n) = (x_1+x_{n-k}) f^{\mathbf{d}_1}_{1,n+k-1}$, respectively.

Using \eqref{eq:exactD3-1}, \eqref{eq:exactD3-2}, and \eqref{eq:exactD3-3}, we construct the additive basis of $R^{\mathbf{D}_{(n)}}$ by the inductive step (cf. Example~\ref{example:D3}).
In order to proceed to the inductive step, we must know a basis of the cokernel $R^{\mathbf{D}_{(n)}}_{h_j}/(x_1-x_j)$ and $R^{\mathbf{D}_{(n)}}_{\tilde{h}_j}/(x_1+x_{2n-j})$. 
We discuss an additive basis for the cokernel in \S\ref{subsect:A.1}. 
Using the exact sequences \eqref{eq:exactD3-1}, \eqref{eq:exactD3-2}, and \eqref{eq:exactD3-3} repeatedly, we can eventually construct the additive basis of $R^{\mathbf{D}_{(n)}}$. (See \S\ref{subsect:A.2} for more details.)

\bigskip

\section{Linear independence}
\label{section:Poincareduals}

We obtained Theorem~\ref{theorem:Intro_ABCDE6FG} for types $A,B,C,G$ as a corollary of Theorem~\ref{theorem_intro}.
Our additive basis of $H^*(\Hess(N,h))$ in type $D_n$ is not necessarily a monomial in positive roots, unlike type $A, B, C, G$ cases.
Nevertheless, a basis $\{v_m^{(h)} \mid 0 \leq \m_i \leq h(i)-i \}$ of $H^*(\Hess(N,h))$ contains the set $\{[\Hess(N,h')] \in H^*(\Hess(N,h)) \mid h' \subset h \}$ of all Poincar\'e duals of smaller regular nilpotent Hessenberg varieties.

\begin{lemma} \label{lemma:linear independence_typeD}
Let $h$ and $h'$ be Hessenberg functions for type $D_n$ with $h' \subset h$.
Then, we have 
$$
\alpha^{(h')}_{i,j}=\alpha_{i,j}
$$
for $1 \leq i \leq n$ and $h'(i)+1 \leq j \leq h(i)$.
\end{lemma}

\begin{proof}
We put $\ell=h'$. 
Note that if the procedure $\ell \rightarrow \ell'$ can be defined, then $\ell'$ is a Hessenberg function for type $D_{n-1}$ because $\ell$ is a Hessenberg function for type $D_n$. \\
\textbf{Case (i):} Suppose that $1 \leq i \leq n-1$.
If $\ell'$ is undefined, then $\alpha^{(\ell)}_{i,j}$ is the positive root $\alpha_{i,j}$ from Lemma~\ref{lemma:5-18} and $\alpha^{D_n}_{i,j}(x_1,\ldots,x_n)=\alpha^{D_{n-1}}_{i-1,j-1}(x_2,\ldots,x_n)$.
We may assume that the procedure $\ell \rightarrow \ell'$ is defined. 
Then, one has $\alpha^{(\ell)}_{i,j}(x_1,\ldots,x_n)=\alpha^{(\ell')}_{i-1,j-1}(x_2,\ldots,x_n)$ by  Lemma~\ref{lemma:5-18} again. 
Since $\ell'$ is a Hessenberg function for type $D_{n-1}$ and $\alpha^{D_n}_{i,j}(x_1,\ldots,x_n)=\alpha^{D_{n-1}}_{i-1,j-1}(x_2,\ldots,x_n)$, it is enough to prove the desired equality for the case $i=1$, namely $\alpha^{(\ell)}_{1,j}=\alpha_{1,j}$. 
By the definition of $\alpha_{1,j}^{(\ell)}$, we may assume that $\ell \rightarrow \ell'$ is obtained from \eqref{eq:proc2r}. 
Then, since $\ell_n=2n-1$ and $\ell$ is a Hessenberg function, one has $\ell_1=n-1$ and hence  $k:=n-j+1 \leq n-(h'(1)+1)+1=n-\ell_1=1$. 
This implies that $\alpha^{(\ell)}_{1,j}=\alpha_{1,j}$. \\
\textbf{Case (ii):} Suppose that $i = n$. We show that $\alpha^{(\ell)}_{n,2n-r}=\alpha_{n,2n-r}$. 
As in the above case, it suffices to show that $\alpha^{(\ell)}_{n,2n-1}=\alpha_{n,2n-1}$. 
If $\ell_n < 2n-1$, then $\ell_1 \leq n$ because $\ell$ is a Hessenberg function. 
Hence, $\ell \xrightarrow{\eqref{eq:proc3r}} \ell'$ can not happen. 
This implies that $\alpha^{(\ell)}_{n,2n-1}=\alpha_{n,2n-1}$, as desired. 
\end{proof}
Fix a Hessenberg function $h$ for type $D_n$.
For a smaller Hessenberg function $h' \subset h$ we put $m=h-h'$.
Then, $v^{(h)}_m$ is the Poincar\'e dual $[\Hess(N,h')]$ in $H^*(\Hess(N,h))$ up to a non-zero scalar multiplication by Lemma~\ref{lemma:linear independence_typeD} and \eqref{eq:PoincaredualNilpotent_h}.
Therefore, we obtain Theorem~\ref{theorem:Intro_ABCDE6FG} for type $D$ as a corollary of Theorem~\ref{theorem:basisD}.

We do expect the analogue of the linear independence for other exceptional types.
In fact, using an explicit presentation of the cohomology ring $H^*(\Hess(N,h))$ given in \cite[Corollary~7.2]{EHNT1}, we can prove the linear independence in types $F_4$ and $E_6$ by Maple\footnote{The program is available at https://researchmap.jp/ehrhart/Database/.}.
We restate Theorem~\ref{theorem:Intro_ABCDE6FG} as follows.

\begin{theorem} \label{theorem:PoincaredualABCDE6FG}
Let $h$ be a Hessenberg function for types $A_n,B_n,C_n,D_n, E_6, F_4$, or $G_2$ and $\Hess(N,h)$ the associated regular nilpotent Hessenberg variety. 
Then, the set of the Poincar\'e duals 
\begin{equation*} 
\{[\Hess(N,h')] \in H^*(\Hess(N,h)) \mid h' \subset h \}
\end{equation*}
is linearly independent. 
\end{theorem}

\bigskip

\appendix
\section{More details for type $D$} \label{appendix:A}

This appendix accounts for the details of the proof of Theorem~\ref{theorem:basisD}.

\subsection{Construction of a basis for $R^{(\mathbf{B}_{(n)},\mathbf{D}_{(n)})}(j)$} \label{subsect:A.1}

Recall from Section~\ref{section:typeD} that the polynomials $f^{\mathbf{d}_i}_{i,2n-1-i} \ (1 \leq i \leq n-1)$ and the associated ring $R^{\mathbf{D}_{(n)}}$ are defined as
\begin{align*}
f^{\mathbf{d}_i}_{i,2n-1-i}=&\sum_{k=1}^i d_{ik} (x_k^2-x_{i+1}^2)\cdots(x_k^2-x_n^2) + (-1)^{n-i+1}p_i \, x_{i+1}^2 \cdots x_n^2 \\
& {\rm for} \ i \in [n-1] \ {\rm and} \ \mathbf{d}_i=(d_{i1},\ldots,d_{ii}, p_i) \in \R^{i+1}, \\
R^{\mathbf{D}_{(n)}}=&\R[x_1,\ldots,x_n]/(f^{\mathbf{d}_i}_{i,2n-1-i}, x_1 \cdots x_n \mid 1 \leq i \leq n-1) \\
& {\rm for} \ \mathbf{D}_{(n)}=(\mathbf{d}_i \in \R^{i+1} \mid 1 \leq i \leq n-1).
\end{align*}
In this section we begin with the definition of analogues of the polynomial and the ring above.
For $i \in [n]$ and $\mathbf{b}_i=(b_{i1},\ldots,b_{ii}) \in \R^i$, define a polynomial
\begin{equation*}
g^{\mathbf{b}_i}_{i,2n+1-i}:=\sum_{k=1}^i b_{ik} (x_k^2-x_{i+1}^2)\cdots(x_k^2-x_n^2) x_k^2.
\end{equation*}
Note that the polynomial $g^{\mathbf{b}_i}_{i,2n+1-i}$ is a slight generalization of the polynomial which is used to describe the cohomology ring of $\Hess(N,h)$ in type $B_n$ (see \cite[Section~10.3]{AHMMS}).
For $\mathbf{D}_{(n)}=(\mathbf{d}_i \in \R^{i+1} \mid 1 \leq i \leq n-1)$ and $\mathbf{B}_{(n)}=(\mathbf{b}_i \in \R^i \mid 1 \leq i \leq n)$ we define rings 
\begin{align*} 
R^{\mathbf{D}_{(n)}}(j):=&\R[x_1,\ldots,x_n]/(f^{\mathbf{d}_i}_{i,2n-1-i}, (x_1 \cdots x_n)x_j \mid 1 \leq i \leq n-1) \ \ \ \\
&{\rm for} \ j=1,\ldots,n, \\
R^{(\mathbf{B}_{(n)},\mathbf{D}_{(n)})}(j):=&\R[x_1,\ldots,x_n]/(g^{\mathbf{b}_p}_{p,2n+1-p}, f^{\mathbf{d}_{q-1}}_{q-1,2n-1-(q-1)} \mid 1 \leq p \leq j-1, j \leq q \leq n) \ \ \ \\
&{\rm for} \ j=1,\ldots,n+1,
\end{align*}
where $f^{\mathbf{d}_{0}}_{0,2n-1}:=x_1^2 \cdots x_n^2$.

The following lemma can be proved by using the same argument for the proof of Lemma~\ref{lemma:setsudouringA}.

\begin{lemma} \label{lemma:D1.5}
For $\mathbf{D}_{(n)}=(\mathbf{d}_i \mid 1 \leq i \leq n-1)$ and $1 \leq j \leq n$, we have 
$$
R^{\mathbf{D}_{(n)}}(j)/(x_j) \cong R^{(\mathbf{B}'_{(n-1)},\mathbf{D}'_{(n-1)})}(j)
$$
for some $\mathbf{B}'_{(n-1)}=(\mathbf{b}'_i \in \R^i \mid 1 \leq i \leq n-1)$ and $\mathbf{D}'_{(n-1)}=(\mathbf{d}'_i \in \R^{i+1} \mid 1 \leq i \leq n-2)$. 
Here, we regard the variables in the ring $R^{(\mathbf{B}'_{(n-1)},\mathbf{D}'_{(n-1)})}(j)$ as $x_1,\ldots,\widehat{x_{j}},\ldots,x_n$.
The caret sign \ $\widehat{}$ \ over $x_j$ means that the entry is omitted.
\end{lemma}

Assume that $R^{\mathbf{D}_{(n)}}(j)$ is Artinian.
Then, from Lemmas~\ref{lemma:key} and \ref{lemma:D1.5} one obtains the following exact sequence 
\begin{equation} \label{eq:exactD1}
0 \rightarrow R^{\mathbf{D}_{(n)}} \xrightarrow{\times x_j} R^{\mathbf{D}_{(n)}}(j) \rightarrow R^{(\mathbf{B}'_{(n-1)},\mathbf{D}'_{(n-1)})}(j) \rightarrow 0
\end{equation}
for some $\mathbf{B}'_{(n-1)}$ and $\mathbf{D}'_{(n-1)}$. 
Here, we regard the variables in the ring $R^{(\mathbf{B}'_{(n-1)},\mathbf{D}'_{(n-1)})}(j)$ as $x_1,\ldots,\widehat{x_{j}},\ldots,x_n$.

We now discuss an additive basis of the ring $R^{(\mathbf{B}_{(n)},\mathbf{D}_{(n)})}(j)$.
Fix a positive integer $j$ with $1 \leq j \leq n+1$.
For $1 \leq i \leq n$ and $i+1 \leq k \leq 2(n+1)-1-i$, we define
$$
\beta^{(j)}_{i,k}:=\alpha^{D_{n+1}}_{i,k}(x_1,\ldots,x_{j-1},0,x_j,\ldots,x_n),
$$
where the right hand side $\alpha^{D_{n+1}}_{i,k}(y_1,\ldots,y_{n+1})$ means the positive root for type $D_{n+1}$ in the variables $y_1,\ldots,y_{n+1}$.
The list of $\beta^{(j)}_{i,k}$ is shown in Figures~\ref{picture:PositiveRootTypeBandD} and \ref{picture:PositiveRootTypeB}.
Note that the $\beta^{(n+1)}_{i,k}$ is nothing but the positive root for type $B_n$.

\begin{figure}[h]
\begin{center}
\begin{picture}(500,110)
\hspace{-25pt}
\put(0,100){\framebox(30,10){\tiny $x_1-x_2$}} 
\put(35,85){$\ddots$}
\put(35,100){$\cdots$} 
\put(55,100){\framebox(50,10){\tiny $x_1-x_{j-1}$}} 
\put(80,85){$\vdots$} 
\put(55,70){\framebox(50,10){\tiny $x_{j-2}-x_{j-1}$}} 
\put(105,100){\framebox(20,10){\tiny $x_1$}} 
\put(115,85){$\vdots$} 
\put(105,70){\framebox(20,10){\tiny $x_{j-2}$}} 
\put(105,60){\framebox(20,10){\tiny $x_{j-1}$}} 
\put(125,100){\framebox(40,10){\tiny $x_1-x_j$}} 
\put(145,85){$\vdots$} 
\put(125,70){\framebox(40,10){\tiny $x_{j-2}-x_j$}} 
\put(125,60){\framebox(40,10){\tiny $x_{j-1}-x_j$}} 
\put(125,50){\framebox(40,10){\tiny $-x_j$}} 
\put(165,100){\framebox(50,10){\tiny $x_1-x_{j+1}$}} 
\put(185,85){$\vdots$} 
\put(165,70){\framebox(50,10){\tiny $x_{j-2}-x_{j+1}$}} 
\put(165,60){\framebox(50,10){\tiny $x_{j-1}-x_{j+1}$}} 
\put(165,50){\framebox(50,10){\tiny $-x_{j+1}$}} 
\put(165,40){\framebox(50,10){\tiny $x_j-x_{j+1}$}} 
\put(220,100){$\cdots$} 
\put(220,70){$\cdots$} 
\put(220,60){$\cdots$} 
\put(220,50){$\cdots$} 
\put(220,40){$\cdots$} 
\put(220,25){$\ddots$} 
\put(240,100){\framebox(40,10){\tiny $x_1-x_{n}$}} 
\put(260,85){$\vdots$} 
\put(240,70){\framebox(40,10){\tiny $x_{j-2}-x_{n}$}} 
\put(240,60){\framebox(40,10){\tiny $x_{j-1}-x_{n}$}} 
\put(240,50){\framebox(40,10){\tiny $-x_{n}$}} 
\put(240,40){\framebox(40,10){\tiny $x_j-x_{n}$}} 
\put(260,25){$\vdots$} 
\put(240,10){\framebox(40,10){\tiny $x_{n-1}-x_{n}$}} 
\put(280,100){\framebox(50,10){\tiny $x_1+x_{n-1}$}} 
\put(300,85){$\vdots$} 
\put(280,70){\framebox(50,10){\tiny $x_{j-2}+x_{n-1}$}} 
\put(280,60){\framebox(50,10){\tiny $x_{j-1}+x_{n-1}$}} 
\put(280,50){\framebox(50,10){\tiny $x_{n-1}$}} 
\put(280,40){\framebox(50,10){\tiny $x_j+x_{n-1}$}} 
\put(300,33){$\cdot$} 
\put(300,31){$\cdot$} 
\put(300,29){$\cdot$} 
\put(280,20){\framebox(50,10){\tiny $x_{n-2}+x_{n-1}$}} 
\put(335,100){$\cdots$} 
\put(335,70){$\cdots$} 
\put(335,60){$\cdots$} 
\put(335,50){$\cdots$} 
\put(331,40){$\cdot$} 
\put(333,40){$\cdot$}
\put(335,40){$\cdot$}
\put(335,34){$\cdot$} 
\put(340,36){$\cdot$}
\put(345,38){$\cdot$}
\put(350,100){\framebox(40,10){\tiny $x_1+x_j$}} 
\put(370,85){$\vdots$} 
\put(350,70){\framebox(40,10){\tiny $x_{j-2}+x_j$}} 
\put(350,60){\framebox(40,10){\tiny $x_{j-1}+x_j$}} 
\put(350,50){\framebox(40,10){\tiny $x_j$}} 
\put(390,100){\framebox(20,10){\tiny $x_1$}} 
\put(400,85){$\vdots$} 
\put(390,70){\framebox(20,10){\tiny $x_{j-2}$}} 
\put(390,60){\framebox(20,10){\tiny $x_{j-1}$}} 
\put(410,100){\framebox(50,10){\tiny $x_1+x_{j-1}$}} 
\put(435,85){$\vdots$} 
\put(410,70){\framebox(50,10){\tiny $x_{j-2}+x_{j-1}$}} 
\put(465,85){$\cdot$}
\put(470,87.5){$\cdot$}
\put(475,90){$\cdot$}
\put(465,100){$\cdots$} 
\put(480,100){\framebox(30,10){\tiny $x_1+x_2$}} 
\end{picture}
\end{center}
\vspace{-10pt}
\caption{The list of $\beta^{(j)}_{i,k}$ for $1 \leq j \leq n$.}
\label{picture:PositiveRootTypeBandD}
\end{figure} 

\begin{figure}[h]
\begin{center}
\begin{picture}(440,80)
\put(400,70){\framebox(40,10){\tiny $x_1+x_2$}}
\put(335,70){$\cdots$}
\put(160,70){\framebox(40,10){\tiny $x_1-x_n$}} 
\put(200,70){\framebox(40,10){\tiny $x_1$}} 
\put(240,70){\framebox(40,10){\tiny $x_1+x_n$}} 
\put(95,70){$\cdots$}
\put(0,70){\framebox(40,10){\tiny $x_1-x_2$}} 
\put(55,55){$\ddots$}
\put(218,55){$\vdots$}
\put(375,55){$\cdot$} 
\put(380,57.5){$\cdot$}
\put(385,60){$\cdot$}

\put(320,40){\framebox(40,10){\tiny $x_i+x_{i+1}$}} 
\put(160,40){\framebox(40,10){\tiny $x_i-x_n$}}
\put(295,40){$\cdots$} 
\put(200,40){\framebox(40,10){\tiny $x_i$}}
\put(135,40){$\cdots$} 
\put(240,40){\framebox(40,10){\tiny $x_i+x_n$}}
\put(80,40){\framebox(40,10){\tiny $x_i-x_{i+1}$}}
\put(135,25){$\ddots$}
\put(218,25){$\vdots$}
\put(295,25){$\cdot$} 
\put(300,27.5){$\cdot$}
\put(305,30){$\cdot$}
\put(240,10){\framebox(40,10){\tiny $x_{n-1}+x_n$}}
\put(200,10){\framebox(40,10){\tiny $x_{n-1}$}}
\put(160,10){\framebox(40,10){\tiny $x_{n-1}-x_n$}}
\put(200,0){\framebox(40,10){\tiny $x_n$}}
\end{picture}
\end{center}
\vspace{-10pt}
\caption{The list of $\beta^{(n+1)}_{i,k}$.}
\label{picture:PositiveRootTypeB}
\end{figure} 
For $m=(m_1,\ldots,m_n)$ with $0 \leq m_i \leq 2(n-i)+1$ for any $i \in [n]$, we define 
$$
w_m^{(j)}:=\prod_{i=1}^{n} \beta_{i,2n+1-i}^{(j)} \cdot \beta_{i,2n-i}^{(j)} \cdots \beta_{i,2n+1-i-m_i+1}^{(j)}
$$
with the convention $\beta_{i,2n+1-i}^{(j)} \cdot \beta_{i,2n-i}^{(j)} \cdots \beta_{i,2n+1-i-m_i+1}^{(j)}=1$ whenever $\m_i=0$.
Then, we have the following proposition whose proof is the similar to the proof of Proposition~\ref{proposition:basisAFlag}.

\begin{proposition} \label{proposition:D2}
Let $1 \leq j \leq n+1$.
Assume that $R^{(\mathbf{B}_{(n)},\mathbf{D}_{(n)})}(j)$ is Artinian. 
Then, the classes $w_m^{(j)}$ with $0 \leq \m_i \leq 2(n-i)+1$, form a basis for $R^{(\mathbf{B}_{(n)},\mathbf{D}_{(n)})}(j)$ over $\R$.
\end{proposition}

\subsection{Proof of Theorem~\ref{theorem:basisD}} \label{subsect:A.2}

As discussed in Section~\ref{section:typeD}, we prove Theorem~\ref{theorem:basisD} for the case when $h=h^{(n)}_{\max}$ by using the exact sequences \eqref{eq:exactD3-1}, \eqref{eq:exactD3-2}, and \eqref{eq:exactD3-3}.
Noting that each cokernel in \eqref{eq:exactD3-1}, \eqref{eq:exactD3-2}, and \eqref{eq:exactD3-3} is isomorphic to 
\begin{align*}
&R^{\mathbf{D}_{(n)}}_{h_j}/(x_1-x_j) \cong R^{\mathbf{D}'_{(n-1)}} \ \ \ {\rm for} \ 2 \leq j \leq n, \\
&R^{\mathbf{D}_{(n)}}_{\tilde{h}_n}/(x_1+x_n) \cong R^{\mathbf{D}'_{(n-1)}}(n-1), \\
&R^{\mathbf{D}_{(n)}}_{\tilde{h}_j}/(x_1+x_{2n-j}) \cong R^{\mathbf{D}'_{(n-1)}}(2n-j-1) \ \ \ {\rm for} \ n+1 \leq j \leq 2n-2, 
\end{align*}
we derive the exact sequences 
\begin{align} 
&0 \rightarrow R^{\mathbf{D}_{(n)}}_{h_{j-1}} \xrightarrow{\times (x_1-x_j)} R^{\mathbf{D}_{(n)}}_{h_j} \rightarrow R^{\mathbf{D}'_{(n-1)}} \rightarrow 0 \ \ \ {\rm for} \ 2 \leq j \leq n, \label{eq:exactD3-1'} \\
&0 \rightarrow R^{\mathbf{D}_{(n)}}_{h_{n}} \xrightarrow{\times (x_1+x_n)} R^{\mathbf{D}_{(n)}}_{\tilde{h}_n} \rightarrow R^{\mathbf{D}'_{(n-1)}}(n-1) \rightarrow 0, \label{eq:exactD3-2'} \\
&0 \rightarrow R^{\mathbf{D}_{(n)}}_{\tilde{h}_{j-1}} \xrightarrow{\times (x_1+x_{2n-j})} R^{\mathbf{D}_{(n)}}_{\tilde{h}_j} \rightarrow R^{\mathbf{D}'_{(n-1)}}(2n-j-1) \rightarrow 0 \ \ \ {\rm for} \ n+1 \leq j \leq 2n-2, \label{eq:exactD3-3'}
\end{align}
for some $\mathbf{D}'_{(n-1)}$. 
Remark that the variables in the rings $R^{\mathbf{D}'_{(n-1)}}$, $R^{\mathbf{D}'_{(n-1)}}(n-1)$, and $R^{\mathbf{D}'_{(n-1)}}(2n-j-1)$ are regarded as $x_2,\ldots,x_n$.
We now give a proof of Theorem~\ref{theorem:basisD}. 

\begin{proof}[Proof of Theorem~\ref{theorem:basisD}]
By Lemma~\ref{lemma:D1}, it is enough to prove that the elements $v_m^{(h^{(n)}_{\max})}$ with $0 \leq \m_i \leq h^{(n)}_{\max}(i)-i$, form an additive basis for the ring $R^{\mathbf{D}_{(n)}}$ with the assumption that $R^{\mathbf{D}_{(n)}}$ is Artinian. 
We prove this by induction on $n$.
The base case $n=2$ is clear. 
Now we assume that $n>2$ and the claim holds for $n-1$, and any $\mathbf{D}'_{(n-1)}=(\mathbf{d}'_i \in \R^{i+1} \mid 1 \leq i \leq n-2)$.

\smallskip

\noindent
\textit{Claim~1 The set $\{v_m^{(h_1)} \mid 0 \leq \m_i \leq h_{1}(i)-i\}$ forms a basis of $R^{\mathbf{D}_{(n)}}_{h_1}$.} 

We put $m'=(m'_1,\ldots,m'_{n-1})$ with $m'_i=m_{i+1}$ for $i \in [n-1]$. 
Then we have $0 \leq m'_i \leq h^{(n)}_{\max}(i+1)-(i+1) = h^{(n-1)}_{\max}(i)-i $ for $i \in [n-2]$, and $0 \leq m'_{n-1} \leq h^{(n)}_{\max}(n)-1-n = h^{(n-1)}_{\max}(n-1)-(n-1)$.
We see that the isomorphism $R^{\mathbf{D}_{(n)}}_{h_1} \cong R^{\mathbf{D}'_{(n-1)}}$ in \eqref{eq:D3-1} maps $v_{m}^{(h_1)}$ to $v_{m'}^{(h^{(n-1)}_{\max})}$.
Since the first and the last entries $\ell_1$ and $\ell_n$ of $\ell:=h_1-m$ satisfy
\begin{align*}
&\ell_1=1-m_1=1 \leq n, \\
&\ell_n=h^{(n)}_{\max}(n)-1-m_n<h^{(n)}_{\max}(n)=2n-1,
\end{align*}
$\ell'$ is defined and the procedure $\ell \rightarrow \ell'$ is obtained from $\eqref{eq:proc1r}$, and hence
\begin{align*}
\ell'=(h^{(n)}_{\max}(2)-m_2-1, \ldots, h^{(n)}_{\max}(n-1)-m_{n-1}-1, h^{(n)}_{\max}(n)-1-m_n-1)=h^{(n-1)}_{\max}-m'.
\end{align*}
By Lemma~\ref{lemma:5-18} we obtain $v_{m}^{(h_1)}(x_1,\ldots,x_n) = v_{m'}^{(h^{(n-1)}_{\max})}(x_2,\ldots,x_n)$ under the isomorphism $R^{\mathbf{D}_{(n)}}_{h_1} \cong R^{\mathbf{D}'_{(n-1)}}$.
However, $R^{\mathbf{D}'_{(n-1)}}$ has a basis $\{v_{m'}^{(h^{(n-1)}_{\max})} \mid 0 \leq \m'_i \leq h^{(n-1)}_{\max}(i)-i\}$ by the inductive assumption on $n$, which proves Claim~1. 

\smallskip

\noindent
\textit{Claim~2 The set $\{v_m^{(h_j)} \mid 0 \leq \m_i \leq h_{j}(i)-i\}$ forms a basis of $R^{\mathbf{D}_{(n)}}_{h_j}$ for $1 \leq j \leq n$.} 

We prove Claim~2 by induction on $j$.
The base case $j=1$ is nothing but Claim~1.
Now we assume that $j>1$ and Claim~2 holds for $j-1$.
Consider the exact sequence in \eqref{eq:exactD3-1'}.
By the inductive hypothesis on $j$, a basis of $R^{\mathbf{D}_{(n)}}_{h_{j-1}}$ is given by 
\begin{align} 
\{v_s^{(h_{j-1})} \mid 0 \leq s_i \leq h_{j-1}(i)-i \ (1 \leq i \leq n) \}.  \label{eq:D3Claim1-2'} 
\end{align}
On the other hand, by the inductive assumption on $n$ we can take as a basis of $R^{\mathbf{D}'_{(n-1)}}$ in the variables $x_2,\ldots,x_n$ the set 
\begin{align*} 
\{v_{m'}^{(h^{(n-1)}_{\max})}(x_2,\ldots,x_n) \mid 0 \leq \m'_i \leq h^{(n-1)}_{\max}(i)-i \ (1 \leq i \leq n-1) \}. 
\end{align*}
By an argument similar to Claim~1 the set above is the image of a set 
\begin{align} 
\{v_m^{(h_j)} \mid m_1=0, 0 \leq \m_i \leq h^{(n)}_{\max}(i)-i \ (2 \leq i \leq n-1), 0 \leq \m_n \leq h^{(n)}_{\max}(n)-1-n \} \label{eq:D3Claim1-1} 
\end{align}
under the surjection $R^{\mathbf{D}_{(n)}}_{h_j} \twoheadrightarrow R^{\mathbf{D}'_{(n-1)}}$. 
It then follows from the exact sequence in \eqref{eq:exactD3-1'} 
that a basis of $R^{\mathbf{D}_{(n)}}_{h_j}$ can be obtained by combining the sets \eqref{eq:D3Claim1-1} and \eqref{eq:D3Claim1-2'}, except that the
set \eqref{eq:D3Claim1-2'} must be multiplied by $x_1-x_j$.
We conclude that
\begin{align*}
\big(\eqref{eq:D3Claim1-2'} \times (x_1-x_j) \big) \cup \eqref{eq:D3Claim1-1} 
\end{align*}
is an additive basis of $R^{\mathbf{D}_{(n)}}_{h_j}$.
In order to see that this set coincides with the set given in the statement of Claim~2, it suffices to show that 
\begin{align*}
v_m^{(h_j)}= (x_1-x_j) \cdot v_s^{(h_{j-1})}
\end{align*}
for $m_1=s_1+1$ and $m_i=s_i \ (2 \leq i \leq n)$ where $s$ runs over the condition in \eqref{eq:D3Claim1-2'}.
We put $\ell=h_j-m$ and $u=h_{j-1}-s$. 
Then one has $h_{j-1}-s=h_j-m$, and hence
\begin{align} \label{eq:D3Claim1-3}
\alpha^{(\ell)}_{i,j}=\alpha^{(u)}_{i,j}.
\end{align}
Since the first and the last entries $\ell_1$ and $\ell_n$ of $\ell$ satisfy 
$\ell_1=j-m_1 \leq j \leq n$ and 
$\ell_n=h^{(n)}_{\max}(n)-1-m_n<h^{(n)}_{\max}(n)=2n-1$, respectively, 
$\ell'$ is defined and the procedure $\ell \rightarrow \ell'$ is obtained from $\eqref{eq:proc1r}$. Thus we have $\alpha^{(\ell)}_{1,j} = \alpha_{1,j} = x_1-x_j$.
This together with \eqref{eq:D3Claim1-3} leads us to the equality 
\begin{align*}
v_m^{(h_j)}=& (\alpha^{(\ell)}_{1,j} \cdot \alpha^{(\ell)}_{1,j-1} \cdots \alpha^{(\ell)}_{1,j-m_1+1}) \cdot (\prod_{i=2}^n \alpha^{(\ell)}_{i,h^{(n)}_{\max}(i)} \cdots \alpha^{(\ell)}_{i,h^{(n)}_{\max}(i)-m_i+1}) \\
=& \big((x_1-x_j) \cdot \alpha^{(u)}_{1,j-1} \cdots \alpha^{(u)}_{1,(j-1)-s_1+1}\big) \cdot \big(\prod_{i=2}^n \alpha^{(u)}_{i,h^{(n)}_{\max}(i)} \cdots \alpha^{(u)}_{i,h^{(n)}_{\max}(i)-m_i+1}\big) \\
=& (x_1-x_j) \cdot v_s^{(h_{j-1})},
\end{align*} 
as desired.

\smallskip

\noindent
\textit{Claim~3 The set $\{v_m^{(\tilde{h}_n)} \mid 0 \leq \m_i \leq \tilde{h}_{n}(i)-i\}$ forms a basis of $R^{\mathbf{D}_{(n)}}_{\tilde{h}_n}$.} 

Consider the exact sequence in \eqref{eq:exactD3-2'}.
It then follows from Claim~2 that the set 
\begin{align} 
\{v_s^{(h_{n})} \mid 0 \leq s_i \leq h_{n}(i)-i \ (1 \leq i \leq n) \}  \label{eq:D3Claim2-2'} 
\end{align}
forms an additive basis for $R^{\mathbf{D}_{(n)}}_{h_{n}}$. 
We prove the following subclaim.

\smallskip

\noindent
\textit{Subclaim~3 The image of a set 
\begin{align}
\{v_m^{(\tilde{h}_n)} \mid m_n=0, 0 \leq \m_1 \leq n-1, 0 \leq \m_i \leq h^{(n)}_{\max}(i)-i \ (2 \leq i \leq n) \} \label{eq:D3Claim2-1}
\end{align}
under the surjection $R^{\mathbf{D}_{(n)}}_{\tilde{h}_n} \twoheadrightarrow R^{\mathbf{D}'_{(n-1)}}(n-1)$ forms an additive basis for $R^{\mathbf{D}'_{(n-1)}}(n-1)$.} 

\smallskip

If we prove the subclaim above, then Claim~3 holds. 
In fact, as in the argument of Claim~2, 
the exact sequence in \eqref{eq:exactD3-2'} deduces that 
\begin{align*}
\big(\eqref{eq:D3Claim2-2'} \times (x_1+x_n) \big) \cup \eqref{eq:D3Claim2-1}
\end{align*}
is an additive basis of $R^{\mathbf{D}_{(n)}}_{\tilde{h}_n}$.
However, one can see from an argument similar to Claim~2 that 
$$
(x_1+x_n) \cdot v_s^{(h_{n})} = v_m^{(\tilde{h}_n)}
$$
for $m_n=s_n+1$ and $m_i=s_i \ (1 \leq i \leq n-1)$ where $s$ runs over the condition in \eqref{eq:D3Claim2-2'}.
In fact, since the last entry of $\ell=\tilde{h}_n-m$ is strictly less than $2n-1$ by $m_n>0$, we have the procedure $\ell \xrightarrow{\eqref{eq:proc1r}} \ell'$ and hence $\alpha^{(\tilde{h}_n-m)}_{n,2n-1}=x_1+x_n$.
Thus Claim~3 follows from Subclaim~3.

We now prove Subclaim~3. 
In what follows, the image of $v_m^{(\tilde{h}_n)}$ under the surjection $R^{\mathbf{D}_{(n)}}_{\tilde{h}_n} \twoheadrightarrow R^{\mathbf{D}'_{(n-1)}}(n-1)$ is also denoted by the same notation $v_m^{(\tilde{h}_n)}$.
Consider the exact sequence \eqref{eq:exactD1} for $n-1$ 
\begin{equation} \label{eq:exactD3Claim2}
0 \rightarrow R^{\mathbf{D}'_{(n-1)}} \xrightarrow{\times x_n} R^{\mathbf{D}'_{(n-1)}}(n-1) \rightarrow R^{(\mathbf{B}''_{(n-2)},\mathbf{D}''_{(n-2)})}(n-1) \rightarrow 0,
\end{equation}
where we may regard the variables in the rings $R^{\mathbf{D}'_{(n-1)}}, R^{\mathbf{D}'_{(n-1)}}(n-1)$ as $x_2,\ldots,x_n$, and the variables in the ring $R^{(\mathbf{B}''_{(n-2)},\mathbf{D}''_{(n-2)})}(n-1)$ as $x_2,\ldots,x_{n-1}$.
From Proposition~\ref{proposition:D2} one can take as a basis of $R^{(\mathbf{B}''_{(n-2)},\mathbf{D}''_{(n-2)})}(n-1)$ the set  
\begin{align} \label{eq:D3Claim2-3} 
\{w_{m'}^{(n-1)}(x_2,\ldots,x_{n-1}) \mid 0 \leq \m'_i \leq 2(n-2-i)+1 \ (1 \leq i \leq n-2) \}. 
\end{align}
We first show that an element $w_{m'}^{(n-1)}(x_2,\ldots,x_{n-1})$ of the set \eqref{eq:D3Claim2-3} is the image of $v_m^{(\tilde{h}_n)}$ for $m_n=0$, $m_1=0$, and  $m_i=m'_{i-1} \ (2 \leq i \leq n-1)$ under the surjection $R^{\mathbf{D}'_{(n-1)}}(n-1) \twoheadrightarrow R^{(\mathbf{B}''_{(n-2)},\mathbf{D}''_{(n-2)})}(n-1)$.
Put $\ell:=\tilde{h}_n-m=(n,h^{(n)}_{\max}(2)-m_2,\ldots,h^{(n)}_{\max}(n-1)-m_{n-1},2n-1)$.
Then $\ell'$ is undefined by \eqref{eq:proc2} and hence
\begin{align*}
v_m^{(\tilde{h}_n)}&=\prod_{i=2}^{n-1} \alpha^{D_n}_{i,h^{(n)}_{\max}(i)}(x_1,\ldots,x_n) \cdots \alpha^{D_n}_{i,h^{(n)}_{\max}(i)-m_i+1}(x_1,\ldots,x_n) \\
&=\prod_{i=1}^{n-2} \alpha^{D_{n-1}}_{i,h^{(n-1)}_{\max}(i)}(x_2,\ldots,x_{n-1},x_n) \cdots \alpha^{D_{n-1}}_{i,h^{(n-1)}_{\max}(i)-m'_i+1}(x_2,\ldots,x_{n-1},x_n) \ \ \ ({\rm by \ Lemma~\ref{lemma:5-18}}). 
\end{align*}
Thus, the surjection $R^{\mathbf{D}'_{(n-1)}}(n-1) \twoheadrightarrow R^{\mathbf{D}'_{(n-1)}}(n-1)/(x_n) \cong R^{(\mathbf{B}''_{(n-2)},\mathbf{D}''_{(n-2)})}(n-1)$ sends $v_m^{(\tilde{h}_n)}$ to 
$$
\prod_{i=1}^{n-2} \alpha^{D_{n-1}}_{i,h^{(n-1)}_{\max}(i)}(x_2,\ldots,x_{n-1},0) \cdots \alpha^{D_{n-1}}_{i,h^{(n-1)}_{\max}(i)-m'_i+1}(x_2,\ldots,x_{n-1},0)=w_{m'}^{(n-1)}(x_2,\ldots,x_{n-1}), 
$$
as desired. 
Hence, the surjection $R^{\mathbf{D}'_{(n-1)}}(n-1) \twoheadrightarrow R^{(\mathbf{B}''_{(n-2)},\mathbf{D}''_{(n-2)})}(n-1)$ sends a set 
\begin{equation} \label{eq:D3Claim2-3'}
\{ v_m^{(\tilde{h}_n)} \mid m_n=0, m_1=0, 0 \leq \m_i \leq h^{(n)}_{\max}(i)-i \ (2 \leq i \leq n-1) \}
\end{equation}
to the set in \eqref{eq:D3Claim2-3}.

On the other hand, by the inductive assumption on $n$ the ring $R^{\mathbf{D}'_{(n-1)}}$ has an additive basis 
\begin{align} \label{eq:D3Claim2-4} 
\{v_{t}^{(h^{(n-1)}_{\max})}(x_2,\ldots,x_n) \mid 0 \leq t_i \leq h^{(n-1)}_{\max}(i)-i \ (1 \leq i \leq n-1) \}.
\end{align}
By the exactness of the sequence \eqref{eq:exactD3Claim2} we obtain that the union of \eqref{eq:D3Claim2-4} multiplied by $x_n$ and \eqref{eq:D3Claim2-3'} is an
additive basis of $R^{\mathbf{D}'_{(n-1)}}(n-1)$. 
In order to prove Subclaim~3, it suffices to prove that $v_m^{(\tilde{h}_n)}$ is congruent to
\begin{align*}
x_n \cdot v_{t}^{(h^{(n-1)}_{\max})}(x_2,\ldots,x_n) \ \ \ {\rm mod} \ x_1+x_n
\end{align*}
up to a non-zero scalar multiplication for $m_n=0$, $m_1=t_{n-1}+1$, and $m_i=t_{i-1} \ (2 \leq i \leq n-1)$ where $t$ runs over the condition in \eqref{eq:D3Claim2-4}. 
In what follows, we prove that 
\begin{equation} \label{eq:D3Claim2-5}
v_m^{(\tilde{h}_{n})}(x_1,\ldots,x_n) \equiv \pm 2x_n \cdot v_t^{(h^{(n-1)}_{\max})}(x_2,\ldots,x_n) \ \ \ \ \ {\rm mod} \ x_1+x_n.
\end{equation}
Put $\ell=\tilde{h}_n-m$. 
Since the first and the last entries $\ell_1$ and $\ell_n$ of $\ell$ satisfy
$\ell_1=n-m_1 \leq n-1$ and
$\ell_n=h^{(n)}_{\max}(n)-m_n=2n-1$
, one has $\ell \xrightarrow{\eqref{eq:proc2r}} \ell'$ and hence 
$$
\ell'=(h^{(n)}_{\max}(2)-m_2-1,\ldots,h^{(n)}_{\max}(n-1)-m_{n-1}-1,2n-2-m_1)=h^{(n-1)}_{\max}-t.
$$
It follows from Lemma~\ref{lemma:5-18} that
\begin{align} \label{eq:D3Claim2-6}
\alpha^{(\ell)}_{i+1,j}(x_1,\ldots,x_n) = \alpha^{(\ell')}_{i,j-1}(x_2,\ldots,x_n) 
\end{align}
for $1 \leq i \leq n-2$ and $i+1 < j \leq 2n-1-(i+1)$.
One also has
\begin{equation} \label{eq:D3Claim2-7}
\alpha^{(\ell)}_{1,n}=\alpha_{1,n}=x_1-x_n
\end{equation}
by the definition of $\alpha^{(\ell)}_{i,j}$. 
In order to prove \eqref{eq:D3Claim2-5}, we first prove  
\begin{align} \label{eq:D3Claim2-8} 
\alpha^{(\ell)}_{1,n-p}(x_1,\ldots,x_n) \equiv \pm \alpha^{(\ell')}_{n-1,h^{(n-1)}_{\max}(n-1)-p+1}(x_2,\ldots,x_n) \ \ \ {\rm mod} \ x_1+x_n 
\end{align}
for $1 \leq p < m_1$.
Let us consider the two procedures starting at $\ell$ and $\ell'$:
\begin{align*}
&\ell=\ell^{(1)} \rightarrow \ell^{(2)} \rightarrow \cdots; \\
&\ell'=\ell'^{(1)} \rightarrow \ell'^{(2)} \rightarrow \cdots. 
\end{align*}
Note that $\ell'^{(i)}=\ell^{(i+1)}$ if they can be defined.
One can see from the inductive argument that the procedure 
$$
\ell'=\ell'^{(1)} \rightarrow \ell'^{(2)} \rightarrow \cdots \rightarrow \ell'^{(p)} \rightarrow \ell'^{(p+1)} \rightarrow \cdots \rightarrow \ell'^{(m_1-1)} \rightarrow \ell'^{(m_1)}
$$
can be defined and each step is either \eqref{eq:proc1r} or \eqref{eq:proc3r}. 
In fact, the last entry $\ell'^{(p)}_{n-p}$ of $\ell'^{(p)}$ for $1 \leq p < m_1$ is inductively given by
$$
\ell'^{(p)}_{n-p}=2n-2-m_1-(p-1) < 2n-2-p-(p-1) =2(n-p)-1.
$$
Hence, the procedure 
$$
\ell=\ell^{(1)} \xrightarrow{\eqref{eq:proc2r}} \ell^{(2)} \xrightarrow{\eqref{eq:proc1r} \ {\rm or} \ \eqref{eq:proc3r}} \cdots \xrightarrow{\eqref{eq:proc1r} \ {\rm or} \ \eqref{eq:proc3r}} \ell^{(m_1)} \xrightarrow{\eqref{eq:proc1r} \ {\rm or} \ \eqref{eq:proc3r}} \ell^{(m_1+1)}
$$
can be defined. 
This derives
\begin{align*}
\alpha^{(\ell)}_{1,n-p}(x_1,\ldots,x_n)=\begin{cases}
x_{p+1} \ \ \ &{\rm if} \ \ell^{(p+1)} \xrightarrow{\eqref{eq:proc3r}} \ell^{(p+2)} \\
x_1-x_{p+1} \ \ \ &{\rm if} \ \ell^{(p+1)} \xrightarrow{\eqref{eq:proc1r}} \ell^{(p+2)} 
\end{cases} 
\end{align*}
and
\begin{align*}
\alpha^{(\ell')}_{n-1,h^{(n-1)}_{\max}(n-1)-p+1}(x_2,\ldots,x_n) =&\alpha^{(\ell')}_{n-1,2(n-1)-p}(x_2,\ldots,x_n) \\
=&\begin{cases}
x_{p+1} \ \ \ &{\rm if} \ \ell'^{(p)} \xrightarrow{\eqref{eq:proc3r}} \ell'^{(p+1)} \\
x_{p+1}+x_n \ \ \ &{\rm if} \ \ell'^{(p)} \xrightarrow{\eqref{eq:proc1r}} \ell'^{(p+1)} 
\end{cases} 
\end{align*}
for $1 \leq p < m_1$.
Thus, we obtain \eqref{eq:D3Claim2-8}.
It then follows from \eqref{eq:D3Claim2-7}, \eqref{eq:D3Claim2-8}, and \eqref{eq:D3Claim2-6} that 
\begin{align*}
\hspace{-30pt}
v_m^{(\tilde{h}_n)} =&\alpha^{(\ell)}_{1,n}(x_1,\ldots,x_n) \cdot \alpha^{(\ell)}_{1,n-1}(x_1,\ldots,x_n) \cdots \alpha^{(\ell)}_{1,n-m_1+1}(x_1,\ldots,x_n) \\
&\cdot \big(\prod_{i=2}^{n-1} \alpha^{(\ell)}_{i,h^{(n)}_{\max}(i)}(x_1,\ldots,x_n) \cdots \alpha^{(\ell)}_{i,h^{(n)}_{\max}(i)-m_{i}+1}(x_1,\ldots,x_n)\big) \\
=&(x_1-x_n) \cdot \alpha^{(\ell)}_{1,n-1}(x_1,\ldots,x_n) \cdots \alpha^{(\ell)}_{1,n-m_1+1}(x_1,\ldots,x_n) \\
&\cdot \big(\prod_{i=1}^{n-2} \alpha^{(\ell)}_{i+1,h^{(n)}_{\max}(i+1)}(x_1,\ldots,x_n) \cdots \alpha^{(\ell)}_{i+1,h^{(n)}_{\max}(i+1)-m_{i+1}+1}(x_1,\ldots,x_n)\big) \\
\equiv & \pm (x_1-x_n) \cdot \alpha^{(\ell')}_{n-1,h^{(n-1)}_{\max}(n-1)}(x_2,\ldots,x_n) \cdots \alpha^{(\ell')}_{n-1,h^{(n-1)}_{\max}(n-1)-(m_1-1)+1}(x_2,\ldots,x_n) \\
&\cdot \big(\prod_{i=1}^{n-2} \alpha^{(\ell')}_{i,h^{(n-1)}_{\max}(i)}(x_2,\ldots,x_n) \cdots \alpha^{(\ell')}_{i,h^{(n-1)}_{\max}(i)-m_{i+1}+1}(x_2,\ldots,x_n)\big) \ \ \ {\rm mod} \ x_1+x_n \ \ \ ({\rm by \ Lemma~\ref{lemma:5-18}})\\
=& \pm (x_1-x_n) \cdot v_{t}^{(h^{(n-1)}_{\max})}(x_2,\ldots,x_n) \ \ \ ({\rm because} \ t_{n-1}=m_1-1, t_i=m_{i+1} \ {\rm for} \ 1 \leq i \leq n-2)\\
=& \pm2x_n \cdot v_{t}^{(h^{(n-1)}_{\max})}(x_2,\ldots,x_n) \ \ \ {\rm mod} \ x_1+x_n.
\end{align*}

We proved \eqref{eq:D3Claim2-5} and hence Subclaim~3 holds. 
This completes the proof of Claim~3.

\smallskip

\noindent
\textit{Claim~4 The set $\{v_m^{(\tilde{h}_j)} \mid 0 \leq \m_i \leq \tilde{h}_{j}(i)-i\}$ forms a basis of $R^{\mathbf{D}_{(n)}}_{\tilde{h}_j}$ for $n \leq j \leq 2n-2$.} 

We prove Claim~4 by induction on $j$.
As the base case $j=n$, it is exactly Claim~3. 
We now assume that $j>n$ and Claim~4 holds for $j-1$.
Consider the exact sequence in \eqref{eq:exactD3-3'} in this case.
By similar considerations as in the previous claim, the image of a set 
\begin{align}
\{v_m^{(\tilde{h}_j)} \mid m_1=0, 0 \leq \m_i \leq h^{(n)}_{\max}(i)-i \ (2 \leq i \leq n) \}  \label{eq:D3Claim3-1}
\end{align}
under the surjection $R^{\mathbf{D}_{(n)}}_{\tilde{h}_j} \twoheadrightarrow R^{\mathbf{D}'_{(n-1)}}(2n-j-1)$ forms an additive basis for the ring $R^{\mathbf{D}'_{(n-1)}}(2n-j-1)$.
In fact, a similar argument using the exact sequence
\begin{equation} \label{eq:exactD3Claim3}
0 \rightarrow R^{\mathbf{D}'_{(n-1)}} \xrightarrow{\times x_{2n-j}} R^{\mathbf{D}'_{(n-1)}}(2n-j-1) \rightarrow R^{(\mathbf{B}''_{(n-2)},\mathbf{D}''_{(n-2)})}(2n-j-1) \rightarrow 0
\end{equation}
yields the claim in this case.
More specifically, by Proposition~\ref{proposition:D2} we know that the ring $R^{(\mathbf{B}''_{(n-2)},\mathbf{D}''_{(n-2)})}(2n-j-1)$ has a basis
\begin{align*} 
\{w_{m'}^{(2n-j-1)}(x_2,\ldots,\widehat{x_{2n-j}},\ldots,x_n) \mid 0 \leq \m'_i \leq 2(n-2-i)+1 \ (1 \leq i \leq n-2) \},
\end{align*}
and the set above is the image of  
\begin{equation} \label{eq:D3Claim3-3'}
\{ v_m^{(\tilde{h}_j)} \mid m_1=0, m_n=0, 0 \leq \m_i \leq h^{(n)}_{\max}(i)-i \ (2 \leq i \leq n-1) \}
\end{equation}
under the surjection $R^{\mathbf{D}'_{(n-1)}}(2n-j-1) \twoheadrightarrow R^{(\mathbf{B}''_{(n-2)},\mathbf{D}''_{(n-2)})}(2n-j-1)$, by an argument similar to Claim~3.
Furthermore, by the inductive hypothesis on $n$ the ring $R^{\mathbf{D}'_{(n-1)}}$ has a basis 
\begin{align} \label{eq:D3Claim3-4} 
\{v_{t}^{(h^{(n-1)}_{\max})}(x_2,\ldots,x_n) \mid 0 \leq t_i \leq h^{(n-1)}_{\max}(i)-i \ (1 \leq i \leq n-1) \},
\end{align}
and we can prove, in a similar manner to Claim~3, that 
\begin{equation*} 
v_m^{(\tilde{h}_{j})}(x_1,\ldots,x_n) \equiv -x_{2n-j} \cdot v_t^{(h^{(n-1)}_{\max})}(x_2,\ldots,x_n) \ \ \ \ \ {\rm mod} \ x_1+x_{2n-j}
\end{equation*}
for $m_1=0$, $m_i=t_{i-1} \ (2 \leq i \leq n-1)$, and $m_n=t_{n-1}+1$ where $t$ runs over the condition in \eqref{eq:D3Claim3-4}.
Therefore, by the exactness of the sequence \eqref{eq:exactD3Claim3} we conclude that the image of the set \eqref{eq:D3Claim3-1} forms an additive basis for $R^{\mathbf{D}'_{(n-1)}}(2n-j-1)$.

On the other hand, by the inductive hypothesis on $j$, we know that $R^{\mathbf{D}_{(n)}}_{\tilde{h}_{j-1}}$ has an additive basis 
\begin{align} 
\{v_s^{(\tilde{h}_{j-1})} \mid 0 \leq s_i \leq \tilde{h}_{j-1}(i)-i \ (1 \leq i \leq n) \}.  \label{eq:D3Claim3-2'} 
\end{align}
Note that we have $(x_1+x_{2n-j}) \cdot v_s^{(\tilde{h}_{j-1})} = v_m^{(\tilde{h}_j)}$ for $m_1=s_1+1$ and $m_i=s_i \ (2 \leq i \leq n)$ where $s$ runs over the condition in \eqref{eq:D3Claim3-2'}, by similar manipulations to the previous claim. 
A similar argument using the exact sequence in \eqref{eq:exactD3-3'} yields the result in Claim~4.

\smallskip

Now the proof of Theorem~\ref{theorem:basisD} follows immediately from Lemma~\ref{lemma:D1} and  the special case $j=2n-2$ of Claim~4 because $R^{\mathbf{D}_{(n)}}_{\tilde{h}_{2n-2}}$ is equal to $R^{\mathbf{D}_{(n)}}$ which is isomorphic to $H^*(G/B)$ whenever $\mathbf{D}_{(n)}=\{ \mathbf{d}_i=(1, \ldots, 1, 1) \in \R^{i+1} \mid i \in [n-1] \}$. 
\end{proof}

\end{document}